\PassOptionsToPackage{lowtilde}{url}		
\documentclass[11pt]{amsart}

\usepackage[english]{babel}
\usepackage{amsfonts}
\usepackage{amsmath} 
\usepackage{amssymb}
\usepackage{amsthm}
\usepackage{abraces}
\usepackage{mathtools}                          
\usepackage{mathrsfs}				    
\usepackage{stmaryrd}
\usepackage{verbatim}				    
\usepackage[all]{xy}    	          
\SelectTips{cm}{11}                  
\usepackage[T1]{fontenc}        
\usepackage{xfrac}		   		    
\usepackage{graphicx, subcaption}	
\usepackage[normalem]{ulem}          
\usepackage{microtype}               
\usepackage{tikz}
\usepackage{tikz-cd}
\usetikzlibrary{calc}
\usetikzlibrary{shapes, positioning}

\usepackage{enumitem}
\setenumerate[1]{label=(\roman*),ref=\thesubsection.{\roman*}}    

\usepackage{fullpage}

\usepackage{hyperref}
\hypersetup{
  colorlinks   = true,          
  urlcolor     = blue,          
  linkcolor    = blue,          
  citecolor   = red             
}

\usepackage{pdfsync}

\newcommand{\A}{{\mathbb{A}}}
\newcommand{\N}{{\mathbb{N}}}
\newcommand{\Z}{{\mathbb{Z}}}
\newcommand{\Q}{{\mathbb{Q}}}
\newcommand{\R}{\mathbb{R}}			
\newcommand{\C}{\mathbb{C}}
\newcommand{\F}{{\mathbb F}}
\newcommand{\G}{{\mathbb G}}

\newcommand{\Qp}{{\mathbb{Q}_p}}		

\renewcommand{\AA}{{\mathbb A}}
\newcommand{\PP}{{\mathbb P}}
\newcommand{\aank}{{\mathbb{A}^{1,\mathrm{an}}_k}}	
\newcommand{\pank}{{\mathbb{P}^{1,\mathrm{an}}_k}}	
\newcommand{\aanwka}{{\mathbb{A}^{1,\mathrm{an}}_{\widehat{k^a}}}}	 
\newcommand{\panwka}{{\mathbb{P}^{1,\mathrm{an}}_{\widehat{k^a}}}} 
\newcommand{\pana}[1]{{\mathbb{P}^{1,\mathrm{an}}_{#1}}}

\newcommand{\calA}{{\mathcal A}}
\newcommand{\calB}{{\mathcal B}}
\newcommand{\calC}{{\mathcal C}}
\newcommand{\calD}{{\mathcal D}}

\newcommand{\calH}{{\mathcal H}}
\newcommand{\calI}{{\mathcal I}}
\newcommand{\calJ}{{\mathcal J}}

\newcommand{\calO}{{\mathcal O}}

\newcommand{\calS}{{\mathcal S}}
\newcommand{\calT}{{\mathcal T}}

\newcommand{\calX}{{\mathcal X}}

\newcommand{\scrH}{\mathscr H}		
\newcommand{\scrC}{\mathscr C}

\newcommand{\id}{\mathrm{id}}
	
\newcommand{\an}{\mathrm{an}}
\newcommand\wc{{\mkern 2mu\cdot\mkern 2mu}}
\newcommand\va{|\wc|}
\newcommand\nm{\|\wc\|}
\newcommand\eps{\varepsilon}
\renewcommand\le{\leqslant}
\renewcommand\ge{\geqslant}

\newcommand\st{\mathrm{st}}
\newcommand\wka{\widehat{k^a}}

\newcommand\too{\longrightarrow}
\newcommand\mapstoo{\longmapsto}
\newcommand\simtoo{\overset{\sim}{\too}} 
\newcommand{\cn}[2]{\ensuremath{\llbracket{#1},{#2}\rrbracket}}

\DeclareMathOperator{\Spec}{Spec}

\DeclareMathOperator{\Hom}{Hom}
\DeclareMathOperator{\Mod}{Mod}

\DeclareMathOperator{\Char}{char}
\DeclareMathOperator{\Aut}{Aut}

\DeclareMathOperator{\Gal}{Gal}
\DeclareMathOperator{\GL}{GL}
\DeclareMathOperator{\PGL}{PGL}

\DeclareMathOperator{\trdeg}{tr. deg.}
\DeclareMathOperator{\ord}{ord}

\newtheoremstyle{plain2}    
  {}            
  {}            
  {\itshape}    
  {}            
  {\bfseries}   
  {.}           
  {5pt plus 1pt minus 1pt}  
  {{\thmnumber{(#2)} \thmname{#1}{\thmnote{ (#3)}}}}          
\newtheorem{theorem}{Theorem}[section]
\newtheorem{corollary}[theorem]{Corollary}
\newtheorem{lemma}[theorem]{Lemma}
\newtheorem{proposition}[theorem]{Proposition}

\newtheoremstyle{definition2}    
  {}   
  {}   
  {\normalfont}  
  {}       
  {\bfseries} 
  {.}        
  {5pt plus 1pt minus 1pt} 
  {{(\thmnumber{#2}) \thmname{#1}{\thmnote{#3}}}}          

\theoremstyle{definition}
\newtheorem{definition}[theorem]{Definition}
\newtheorem{notation}[theorem]{Notation}
\newtheorem*{notation*}{Notation}
\newtheorem*{ack}{Acknowledgments}
\newtheorem{example}[theorem]{Example}

\theoremstyle{remark}
\newtheorem{remark}[theorem]{Remark}

\newtheoremstyle{stepstyle}
  {}     {}   
  {\normalfont}  
  {\parindent}       
  {\itshape} 
  {}         
  {5pt plus 1pt minus 1pt} 
  {{\thmname{#1} \thmnumber{#2}:{\thmnote{#3}}}}          
\theoremstyle{stepstyle}

\newtheoremstyle{point}
  {}     {}   
  {\normalfont}  
  {}       
  {\bfseries} 
  {}         
  {5pt plus 1pt minus 1pt} 
  {{\thmname{#1}\thmnumber{#2}.\thmnote{ #3.}}}          
\theoremstyle{point}
\newtheorem{point}[subsection]{}

\newtheoremstyle{point*}
  {}     {}   
  {\normalfont}  
  {}       
  {\bfseries} 
  {}         
  {5pt plus 1pt minus 1pt} 
  {{\thmname{#1}\thmnote{ #3.}}}          
\theoremstyle{point*}
\newtheorem{point*}[subsubsection]{}

\numberwithin{equation}{subsection}

\newtheoremstyle{subpoint}
  {}     {}            
  {\normalfont}  
  {}                   
  {\normalfont} 
  {}         
  {5pt plus 1pt minus 1pt} 
  {{\thmname{#1}(\thmnumber{#2})\thmnote{ #3.}}}          
\theoremstyle{subpoint}
\newtheorem{subpoint}[equation]{}

\tikzset{
  htree leaves/.initial=4,
  sibling angle/.initial=40,
  htree level/.initial={}
}

\makeatletter

\def\htree@growth{%
  \pgftransformrotate{%
    (\pgfkeysvalueof{/tikz/sibling angle})*(.2-.4*\tikznumberofchildren
      +.5*\tikznumberofcurrentchild)}%
  \pgftransformxshift{\the\tikzleveldistance}%
  \pgfkeysvalueof{/tikz/htree level}%
}

\tikzstyle{htree}=[
  growth function=\htree@growth,
  sibling angle=100,
  htree level={
    \tikzleveldistance=.707\tikzleveldistance
    \pgfsetlinewidth{.2}
  }
]

\tikzstyle{btree}=[
  growth function=\htree@growth,
  sibling angle=140,
  htree level={
    \tikzleveldistance=.37\tikzleveldistance
    \pgfsetlinewidth{.2}
  }
]

\long\def\ge@addto@macro#1#2{%
  \begingroup
  \toks@\expandafter\expandafter\expandafter{\expandafter#1#2}%
  \xdef#1{\the\toks@}%
  \endgroup}

\newcommand{\htree}[2][]{%
  \def\htree@start{\noexpand\coordinate}
  \def\htree@end{}
  \foreach \l in {0,...,#2} {
    \g@addto@macro\htree@start{child foreach \noexpand\x in {1,2,3,4,5} {\iffalse}\fi}
    \g@addto@macro\htree@end{\iffalse{\fi}}
    \global\let\htree@start\htree@start
    \global\let\htree@end\htree@end
  }
  \edef\htree@cmd{\htree@start\htree@end;}
  \begin{scope}[htree,#1]
  \htree@cmd
  \end{scope}
}

\makeatother

\title{Berkovich curves and Schottky uniformization}

\author{J\'er\^ome Poineau}
\address{Laboratoire de math\'ematiques Nicolas Oresme, Universit\'e de Caen - Normandie, France}
\email{\href{mailto:jerome.poineau@unicaen.fr}{jerome.poineau@unicaen.fr}}
\urladdr{\url{https://poineau.users.lmno.cnrs.fr/}}

\author{Daniele Turchetti}
\address{Dalhousie University, Department of Mathematics \& Statistics, Halifax, Nova Scotia, Canada}
\email{\href{mailto:daniele.turchetti@dal.ca}{daniele.turchetti@dal.ca}}
\urladdr{\url{https://www.mathstat.dal.ca/~dturchetti/}}

\date{\today}

\usepackage{etoolbox}
\makeatletter
\patchcmd{\@maketitle}
  {\ifx\@empty\@dedicatory}
  {\ifx\@empty\@date \else {\vskip3ex \centering\footnotesize\@date\par\vskip1ex}\fi
   \ifx\@empty\@dedicatory}
  {}{}
\patchcmd{\@adminfootnotes}
  {\ifx\@empty\@date\else \@footnotetext{\@setdate}\fi}
  {}{}{}
\makeatother

\makeatletter
\@addtoreset{section}{part}
\makeatother

\begin{document}

\begin{abstract}
This text is an exposition of non-Archimedean curves and Schottky uniformization from the point of view of Berkovich geometry. It consists of two parts, the first one of an introductory nature, and the second one more advanced.
The first part is meant to be an introduction to the theory of Berkovich spaces focused on the case of the affine line. We define the Berkovich affine line and present its main properties, with many details: classification of points, path-connectedness, metric structure, variation of rational functions, etc. Contrary to many other introductory texts, we do not assume that the base field is algebraically closed.
The second part is devoted to the theory of Mumford curves and Schottky uniformization. We start by briefly reviewing the theory of Berkovich curves, then introduce Mumford curves in a purely analytic way (without using formal geometry). We define Schottky groups acting on the Berkovich projective line, highlighting how geometry and group theory come together to prove that the quotient by the action of a Schottky group is an analytic Mumford curve. Finally, we present an analytic proof of Schottky uniformization, showing that any analytic Mumford curves can be described as a quotient of this kind. 
The guiding principle of our exposition is to stress notions and fully prove results in the theory of non-Archimedean curves that, to our knowledge, are not fully treated in other texts.
\end{abstract}

\maketitle

\setcounter{tocdepth}{1}

\vspace{-.5cm}

\tableofcontents



\section*{Introduction}

%

The purpose of the present survey is to provide an introduction to non-Archimedean analytic geometry from the perspective of uniformization of curves.
The main characters of this compelling story are \emph{analytic curves over a non-Archimedean complete valued field} $(k,\va)$, and \emph{Schottky groups}.
The main difficulty in establishing a theory of non-Archimedean analytic spaces over~$k$ is that the natural topology induced over $k$ by the absolute value~$\va$ gives rise to totally disconnected spaces, that are therefore not suitable for defining analytic notions, such as that of a function locally expandable in power series.


However, in the late 1950's, J.~Tate managed to develop the basics of such a theory (see \cite{TateRigid}), and christened the resulting spaces under the name \emph{rigid analytic spaces}. To bypass the difficulty mentioned above, those spaces are not defined as usual topological spaces, but as spaces endowed with a so-called Grothendieck topology: some open subsets and some coverings are declared admissible and are the only ones that may be used to define local notions. For instance, one may define an analytic function by prescribing its restrictions to the members of an admissible covering. We refer to \cite[\S 5.1]{Pellarin} in this volume for a short introduction to rigid analytic spaces.

Towards the end of the 1980's, V.~Berkovich provided another definition of non-Archimedean analytic spaces. One of the advantages of his approach is that the resulting spaces are true topological spaces, endowed with a topology that makes them especially nice: they are Hausdorff, locally compact, and locally path-connected. This is the theory that we will use in these notes.

In the case of curves, one can combine topology, algebra, and combinatorics to get a very satisfactory description of such spaces. 
If $k$ is algebraically closed, for instance, one can show that a smooth compact Berkovich curve $X$ can always be decomposed into a finite graph and an infinite number of open discs. If the genus of~$X$ is positive, there exists a smallest graph satisfying this property. It is classically called the \emph{skeleton} of $X$, an invariant that encodes a surprising number of properties of $X$. As an example, if the Betti number of the skeleton of~$X$ is equal to the genus of~$X$ and is at least~2, then the curve $X$ can be described analytically as a quotient $\Gamma \backslash O$, where $O$ is an open dense subset of the projective analytic line $\pank$ and $\Gamma$ a suitable subgroup of $\PGL_2(k)$. This phenomenon is known as \emph{Schottky uniformization}, and it is the consequence of a celebrated theorem of D. Mumford, which is the main result of \cite{Mumford72}. 

Obviously, D.~Mumford's proof did not make use of Berkovich spaces, as they were not yet introduced at that time, but rather of formal geometry and the theory of Bruhat-Tits trees. A few years later, L.~Gerritzen and M.~van der Put recasted the theory purely in the language of rigid analytic geometry (using in a systematic way the notion of reduction of a rigid analytic curve). We refer the reader to the reference manuscript~\cite{GerritzenPut80} for a detailed account of the theory and related topics, enriched with examples and applications.

In this text, we develop the whole theory of \emph{Schottky groups} and \emph{Mumford curves} from scratch, in a purely analytic manner, relying in a crucial way on the nice topological properties of Berkovich spaces, and the tools that they enable us to use: the theory of proper actions of groups on topological spaces, of fundamental groups, etc. We are convinced that those features, and Berkovich's point of view in general, will help improve our understanding of Schottky uniformization.

\medbreak

These notes are structured as follows. In Part~\ref{part:A1k}, we introduce the Berkovich affine line over a non-Archimedean valued field~$k$ and study its properties. Contrary to several introductory texts, we do not assume that $k$ is algebraically closed. This part is completely introductory, includes many details, and could be read by an undergraduate student with a minimal knowledge of abstract algebra and valuation theory. 
We develop the theory of the Berkovich affine line~$\aank$ over~$k$ starting from scratch: the definition (Section~\ref{sec:underlyingset}), classification of points with several examples (Section~\ref{sec:classification}), basic topological properties, such as local compactness or a description of bases of neighborhoods of points (Section~\ref{sec:topology}), and the definition of analytic functions (Section~\ref{sec:analyticfunctions}). We then move to more subtle aspects of the theory such as extensions of scalars, including a proof that the Berkovich affine line over~$k$ is the quotient of the Berkovich affine line over~$\wka$ (the completion of an algebraic closure of~$k$) by the absolute Galois group of~$k$ (Section~\ref{sec:extension}) and connectedness properties, culminating with the tree structure of~$\aank$ (Section~\ref{sec:connected}). We finally investigate even finer aspects of~$\aank$ by considering virtual discs and annuli, and their retractions onto points and intervals respectively (Section~\ref{sect:retractions}), defining canonical lengths of intervals inside~$\aank$ (Section~\ref{sec:lengths}), and ending with results on variations of rational functions, which are the very first steps of potential theory (Section~\ref{sec:variation}).

In Part~\ref{part:curves}, we investigate the theory of uniformization of curves under the viewpoint of Berkovich geometry. Compared to Part~\ref{part:A1k}, we have allowed ourselves to be sometimes more sketchy, but this should not cause any trouble to anyone familiar enough with the theory of Berkovich curves. We begin again by reviewing standard material. In Section~\ref{sec:BerkprojMobius}, we define the Berkovich projective line $\pank$ over~$k$, consider its group of $k$-linear automorphisms $\PGL_2(k)$ and introduce the Koebe coordinates for the loxodromic transformations. In Section~\ref{sec:kanal}, we give an introduction to the theory of Berkovich analytic curves, starting with those that locally look like the affine line. For the more general curves, we review the theory without proofs, but provide some references. We conclude this section by an original purely analytic definition of Mumford curves. In Section~\ref{sec:Schottky}, we propose two definitions of Schottky groups, first using the usual description of their fundamental domains and second, \emph{via} their group theoretical properties, using their action of~$\pank$. We show that they coincide by relying on the nice topological properties of Berkovich spaces. In Section~\ref{sec:uniformization}, we prove that every Mumford curve may be uniformized by a dense open subset of~$\pank$ with a group of deck transformations that is a Schottky group. Once again, our proof is purely analytic, relying ultimately on arguments from potential theory. To the best of our knowledge, this is the first complete proof of this result. We conclude the section by investigating automorphisms of Mumford curves and giving explicit examples.


\medbreak

We put a great effort in providing a self-contained presentation of the results above and including details that are often omitted in the literature. However, both the theories of Berkovich curves and Schottky uniformization have a great amount of ramifications and interactions with other branches of mathematics.
For the interested readers, we provide an appendix with a series of references that will hopefully help them to navigate through this jungle of wonderful mathematical objects.

The idea of writing down these notes came to the first author when he was taking part to the VIASM School on Number Theory in June 2018 in Hanoi. 
Just as the school was, the material presented here is primarily aimed at graduate students, although we also cover some of the most advanced developments in the field. Moreover, we have included questions that we believe could be interesting topics for young researchers (see Remarks~\ref{rem:Indra} and~\ref{rem:Liu}). The appendix provides additional material leading to active subjects of research and open problems.

The different chapters in this volume are united by the use of analytic techniques in the study of arithmetic geometry.
While they treat different topics, we encourage the reader to try to understand how they are related and may shed light on each other. In particular, the lecture notes of F.~Pellarin~\cite{Pellarin}, about Drinfeld modular forms, mention several topics related to ours, although phrased in the language of rigid analytic spaces, such as Schottky groups (\S 5) or quotient spaces (\S 6). It would be interesting to investigate to what extent the viewpoint of Berkovich geometry presented here could provide a useful addition to this theory. 
\medbreak

\begin{ack} 
We warmly thank Marco Maculan for his numerous comments on an earlier version of this text, and the anonymous referees for their remarks and corrections.
The Appendix~\ref{app:Shimura} greatly benefited from insights and remarks by Jan Vonk and Henri Darmon.

While writing this text, the authors were supported by the ERC Starting Grant ``TOSSIBERG'' (grant agreement 637027). 
The second author was partially funded by a prize of the \emph{Fondation des Treilles}.
The \emph{Fondation des Treilles} created by Anne Gruner-Schlumberger aims to foster the dialogue between sciences and arts in order to further research and creation. She welcomes researchers and creators in her domain of Treilles (Var, France).
\end{ack}

\medbreak

\begin{notation*}
Let $(\ell,\va)$ be a non-Archimedean valued field. 

We set $\ell^\circ := \{a\in \ell \,\colon\, |a|\le 1\}$. It is a subring of~$\ell$ with maximal ideal $\ell^{\circ\circ} := \{a\in \ell \,\colon\, |a|< 1\}$. We denote the quotient by~$\tilde\ell$ and call it the \emph{residue field} of~$\ell$.

We set $|\ell^\times| := \{|a|,\ a \in \ell^\times\}$. It is a multiplicative subgroup of~$\R_{>0}$ that we call the \emph{value group} of~$\ell$. We denote its divisible closure by $|\ell^\times|^\Q$. It is naturally endowed with a structure of $\Q$-vector space.
\end{notation*}

Once and for all the paper, we fix a non-Archimedean complete valued field $(k,\va)$, a separable closure~$k^s$ of~$k$ and the corresponding algebraic closure~$k^a$. 
The absolute value $\va$ on~$k$ extends uniquely to an absolute value on~$k^a$, thanks to the fact that it extends uniquely to any given finite extension.\footnote{The reader can find a proof of this classical result in many textbooks, for example in \cite[Chapter~7]{Cassels86}.} We denote by~$\widehat{k^a}$ the completion of~$k^a$: it is algebraically closed and coincides with the completion~$\widehat{k^s}$ of~$k^s$. We still denote by~$\va$ the induced absolute value on~$\widehat{k^a}$. We have $|\widehat{k^a}{}^\times| = |k^\times|^\Q$. 

\vskip1cm

\begin{center}
\textbf{\Large{Part I: The Berkovich affine line}}
\end{center}
\refstepcounter{part}
\addcontentsline{toc}{chapter}{\textbf{I. The Berkovich affine line}}
\label{part:A1k}

\medbreak

The first object we introduce in our exposition of non-Archimedean analytic geometry is the Berkovich affine line.
This is already an excellent source of knowledge of properties of Berkovich curves, such as local path-connectedness, local compactness, classification of points, and behaviour under base change.
Other properties, such as global contractibility, do not generalize, but will be useful later to study curves that ``locally look like the affine line'' (see Section \ref{sec:A1like}).

Our main reference for this section is V.~Berkovich's foundational book \cite{Berkovich90}. 
We have also borrowed regularly from A.~Ducros's thorough manuscript~\cite{DucrosRSS}.

\section{The underlying set}\label{sec:underlyingset}

\begin{definition}
The \emph{Berkovich affine line} $\aank$ is the set of multiplicative seminorms on~$k[T]$ that induce the given absolute value~$\va$ on~$k$.
\end{definition}

In more concrete terms, a point of~$\aank$ is a map $\va_{x} \colon k[T] \to \R_{+}$ satisfying the following properties:
\begin{enumerate}
\item $\forall P,Q \in k[T],\ |P+Q|_{x} \le \max(|P|_{x},|Q|_{x})$;
\item $\forall P,Q \in k[T],\ |PQ|_{x} = |P|_{x}|Q|_{x}$;
\item $\forall \alpha\in k,\ |\alpha|_{x} = |\alpha|$.
\end{enumerate}
With a slight abuse of notation, we set 
\[\ker(\va_{x}) := \{P\in k[T]\, \colon\, |P|_{x} = 0\}.\]
It follows from the multiplicativity of~$\va_{x}$ that $\ker(\va_{x})$~is a prime ideal of~$k[T]$.

In the following, we often denote a point of~$\aank$ by~$x$ and by~$\va_{x}$ the corresponding seminorm. This is purely for psychological and notational comfort since~$x$ and~$\va_{x}$ are really the same thing.

\begin{example}\label{ex:rationalpoint}
Each element~$\alpha$ of~$k$ gives rise to a point of~$\aank$ \emph{via} the seminorm
\[\va_{\alpha} \colon P\in k[T] \longmapsto |P(\alpha)| \in \R_{+}.\]
We denote it by~$\alpha$ again. Such a point is called a \emph{$k$-rational point} of~$\aank$. 

Note that, conversely, the element~$\alpha$ may be recovered from~$\va_{\alpha}$ since $\ker(\va_{\alpha}) = (T-\alpha)$. It follows that the construction provides an injection $k\hookrightarrow \aank$.
\end{example}

\begin{example}\label{ex:rigidpoint}
The construction of the previous example still makes sense if we start with a point $\alpha\in k^a$ and consider the seminorm
\[\va_{\alpha} \colon P\in k[T] \longmapsto |P(\alpha)| \in \R_{+}.\]
Such a point is called a \emph{rigid point} of~$\aank$.

However, it is no longer possible to recover~$\alpha$ from~$\va_{\alpha}$ in general. Indeed, in this case, we have $\ker(\va_{a}) = (\mu_{\alpha})$, where~$\mu_{\alpha}$ denotes the minimal polynomial of~$\alpha$ over~$k$ and, if~$\sigma$ is a $k$-linear automorphism of~$k^a$, then, by uniqueness of the extension of the absolute value, we get $\va_{\sigma(\alpha)} = \va_{\alpha}$. One can check that we obtain an injection $k^a/\Aut(k^a/k)\hookrightarrow \aank$.


%
%
%
\end{example}

Readers familiar with scheme theory will notice that the rigid points of~$\aank$ correspond exactly to the closed points of the schematic affine line~$\AA^1_{k}$. However the Berkovich affine line contains many more points, as the following examples show.

\begin{example}\label{ex:type1point}
Each element~$\alpha$ of~$\wka$ gives rise to a point of~$\aank$ \emph{via} the seminorm
\[\va_{\alpha} \colon P\in k[T] \longmapsto |P(\alpha)| \in \R_{+}.\]
This is similar to the construction of rigid points but, if $\alpha$ is transcendental over~$k$, then we have $\ker(\va_{\alpha}) = (0)$ (\emph{i.e.} $\va_{\alpha}$ is an absolute value) and the set of elements~$\alpha'$ in~$\wka$ such that $\va_{\alpha'} = \va_{\alpha}$ is infinite.

\end{example}


There also are examples of a different nature: points that look like ``generic points'' of discs. 

\begin{lemma}\label{lem:etaalphar}
Let $\alpha\in k$ and $r \in \R_{>0}$. The map 
\[\begin{array}{ccccc}
\va_{\alpha,r}& \colon & k[T] & \longrightarrow & \R_{\ge 0}\\
&&\sum_{i\ge 0} a_{i} (T-\alpha)^i &\longmapsto& \max_{i\ge 0} (|a_{i}| r^i)
\end{array}\]
is an absolute value on $k[T]$.

For $\alpha,\beta \in k$ and $r,s\in \R_{>0}$, we have $\va_{\alpha,r} = \va_{\beta,s}$ if, and only if, $|\alpha-\beta| \le r$ and $r=s$.
\end{lemma}
\begin{proof}
It is easy to check that $\va_{\alpha,r}$ is a norm. It remains to prove that it is multiplicative.

Let $P = \sum_{i\ge 0} a_{i} T^i$ and $Q = \sum_{j\ge 0} b_{j} T^j$. We may assume that $PQ \ne 0$. Let $i_{0}$ be the minimal index such that $|a_{i_{0}}| r^{i_{0}} = |P|_{r}$ and $j_{0}$ be the minimal index such that $|b_{j_{0}}| r^{j_{0}} = |Q|_{r}$. 

For $\ell\in \N$, the coefficient of degree~$\ell$ in $PQ$ is
\[ c_{\ell} :=  \sum_{i+j = \ell} a_{i} b_{j}, \]
hence we have 
\[ |c_{\ell}| r^\ell \le \max_{i+j = \ell}  (|a_{i}|r^i \, |a_{j}|r^j) \le |P|_{r}\, |Q|_{r}.\] 
For $\ell = \ell_{0} :=  i_{0} + j_{0}$, we find 
\[ c_{\ell_{0}} = a_{i_{0}} b_{j_{0}} + \sum_{\substack{i+j = \ell_{0} \\ (i,j) \ne (i_{0},j_{0})}} a_{i} b_{j}.\]
For each $(i,j)  \ne (i_{0},j_{0})$ with $i+j = \ell_{0}$, we must have $i< i_{0}$ or $j <j_{0}$, hence $|a_{i}| r^i< |P|_{r}$ or $|b_{j}| r^j< |Q|_{r}$ and, in any case,
$|a_{i} b_{j}| r^{\ell_{0}} < |P|_{r} |Q|_{r}$.
We now deduce from the equality case in the non-Archimedean triangle inequality that $|c_{\ell_{0}}| r^{\ell_{0}} =  |P|_{r} |Q|_{r}$. The result follows.

\medbreak

Let $\alpha,\beta \in k$ and $r,s\in \R_{>0}$. Assume that we have $\va_{\alpha,r} = \va_{\beta,s}$. Applying the equality to $T-\alpha$ and $T-\beta$, we get
\[r = \max (|\alpha-\beta|,s) \textrm{ and } \max (|\alpha-\beta|,r) = s.\]
We deduce that $r=s$ and $|\alpha-\beta| \le r$, as claimed.

\medbreak

Conversely, assume that we have $r=s$ and $|\alpha-\beta| \le r$. Arguing by symmetry, it it enough to prove that, for each $P \in k[T]$, we have $|P|_{\beta,r} \le |P|_{\alpha,r}$. Let $P  = \sum_{i\ge 0} a_{i} (T-\alpha)^i \in k[T]$. We have
\[|T-\alpha|_{\beta,r} = \max(|\alpha-\beta|,r) = r\]
and, since $\va_{\beta,r}$ is multiplicative, $|(T-\alpha)^i|_{\beta,r} = r^i$ for each $i\ge 0$. Applying the non-Archimedean triangle inequality, we now get 
\[|P|_{\beta,r} \le \max_{i\ge 0}  (|a_{i}| r^i) = |P|_{\alpha,r}.\]
The result follows.
\end{proof}

\begin{example}\label{ex:genericpointdisc}
Let $\alpha\in k$ and $r \in \R_{>0}$. The map 
\[\va_{\alpha,r} \colon  \sum_{i\ge 0} a_{i} (T-\alpha)^i \mapsto \max_{i\ge 0} (|a_{i}| r^i)\]
from Lemma~\ref{lem:etaalphar} is an absolute value, hence gives rise to a point of~$\aank$, which we will denote by~$\eta_{\alpha,r}$. To ease notation, we set $\va_{r} := \va_{0,r}$ and $\eta_{r} := \eta_{0,r}$.

Note that the relation characterizing the equality between $\eta_{\alpha,r}$ and $\eta_{\beta,s}$ is the same that characterizes the equality between the disc with center~$\alpha$ and radius~$r$ and the disc with center~$\beta$ and radius~$s$ in a non-Archimedean setting. This is no coincidence and one can actually prove that, if~$k$ is not trivially valued, the absolute value~$\va_{\alpha,r}$ is equal to the supremum norm on the closed disc of radius~$\alpha$ and center~$r$ in the algebraic closure~$k^a$ of~$k$.

%
%
%
%
\end{example}

\begin{remark}
The definitions of~$\va_{r}$ and~$\va_{\alpha,r}$ from Example~\ref{ex:genericpointdisc} still make sense for~$r=0$. In this case, the points $\eta_{0} $ and $\eta_{\alpha,0}$ that we find are the rational points associated to~0 and~$\alpha$ respectively. It will sometimes be convenient to use this notation. 
\end{remark}

Note that we could combine the techniques of Examples~\ref{ex:rigidpoint} and~\ref{ex:genericpointdisc} to define even more points.

\begin{figure}[h]

\includegraphics[scale=.7]{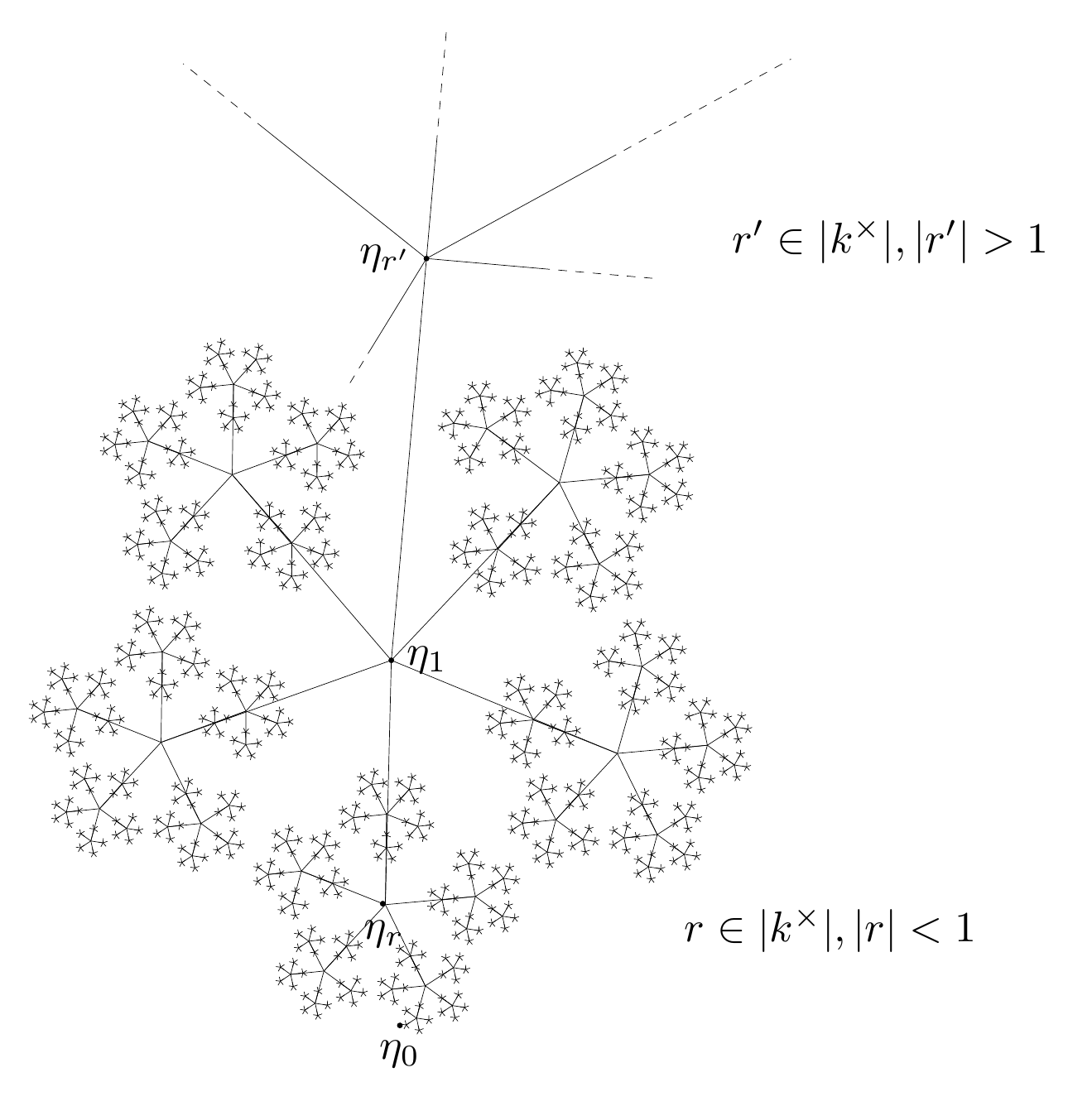}

\caption{The Berkovich affine line $\aank$ when $k$ is an algebraically closed, complete, valued field.} 
\end{figure}

\section{Classification of points}\label{sec:classification} 

In this section, we give a classification of the points of the Berkovich affine line~$\aank$. Let us first introduce a definition.

\begin{definition}\label{def:Hx}
Let $x\in \aank$. The \emph{completed residue field}~$\scrH(x)$ of~$x$ is the completion of the fraction field of $k[T]/\ker(\va_{x})$ with respect to the absolute value induced by~$\va_{x}$. It is a complete valued extension of~$k$. 

We will simply denote by~$\va$ the absolute value induced by~$\va_{x}$ on $\scrH(x)$. 
\end{definition}


The construction provides a canonical morphism of $k$-algebras $\chi_{x} \colon k[T] \to \scrH(x)$. We think of it as an evaluation morphism (into a field that varies with the point). For $P \in k[T]$, we set $P(x) := \chi_{x}(P)$. It then follows from the definition that we have $|P(x)| = |P|_{x}$.

\begin{example}
Let $x\in \aank$. If~$x$ is a $k$-rational point, associated to some element $\alpha\in k$, we have $\scrH(x) = k$ and the morphism $\chi_{x}$ is nothing but the usual evaluation morphism $P \in k[T] \mapsto P(\alpha) \in k$.

If~$x$ is a rigid point, associated to some element $\alpha\in k^a$, we have an isomorphism $\scrH(x) \simeq k(\alpha)$. Conversely, if $\scrH(x)$ is a finite extension of~$k$, then we have $\ker(\va_{x}) = (P)$ for some irreducible polynomial~$P$, hence the point~$x$ is rigid (associated to any root of~$P$).
\end{example}

\begin{example}
Let $\alpha\in k$ and $r\in \R_{>0} - |k^\times|^\Q$. Then $\scrH(\eta_{\alpha,r})$ is isomorphic to the field 
\[k_{r} : = \Big\{f = \sum_{i\in \Z} a_{i} (T-\alpha)^i : \lim_{i\to \pm\infty} |a_{i}| r^i = 0\Big\}\]
endowed with the absolute value $|f| = \max_{i\in \Z} ( |a_{i}| r^i )$.
\end{example}

In the previous example, there exists a unique $i_{0} \in \Z$ for which the quantity $|a_{i}| r^i$ is maximal. The fact that~$k_{r}$ is a field follows. This is no longer true for $r\in  |k^\times|^\Q$ and the completed residue field $\scrH(\eta_{\alpha,r})$ is more difficult to describe (see \cite[Theorem~2.1.6]{Christol83} for instance).

\begin{definition}
A \emph{character} of $k[T]$ is a morphism of $k$-algebras $\chi \colon k[T] \to K$, where~$K$ is some complete valued extension of~$k$.

Two characters $\chi' \colon k[T] \to K'$ and $\chi'' \colon k[T] \to K''$ are said to be equivalent if there exists a character $\chi \colon k[T] \to K$ and isometric embeddings $i' \colon K \to K'$ and $i'' \colon K \to K''$ that make the following diagram commutative:
\[\begin{tikzcd}
&&K'\\
k[T] \arrow[r, "\chi" near end]\arrow[urr, "\chi'"]\arrow[drr, "\chi''"'] & K\arrow[ur, "i'"']\arrow[dr, "i''"]\\
&&K''
\end{tikzcd}.\]

\end{definition}

We have already explained how a point $x\in \aank$ gives rise to a character $\chi_{x} \colon k[T] \to \scrH(x)$. Conversely, to each character $\chi \colon k[T] \to K$, where $(K,\va)$ is a complete valued extension of~$(k,\va)$, we may associate the multiplicative seminorm
\[ \va_{\chi} \colon P\in k[T] \to |\chi(P)| \in \R.\] 
Any equivalent character would lead to the same seminorm.

\begin{lemma}
The map $x \mapsto \chi_{x}$ is a bijection from $\aank$ to the set of equivalences classes of characters of~$k[T]$. Its inverse is the map $\chi \mapsto \va_{\chi}$.
\qed
\end{lemma}

We mention the following related standard fact (see \cite[Lemma~2.8 and Remark~2.9]{KedlayaReified} for instance).

\begin{lemma}\label{lem:commonextension}
Any two complete valued extensions of~$k$ may be isometrically embedded in a common one.
\qed
\end{lemma}

Even if we do not have a explicit description of the completed residue fields associated to the points of~$\aank$, we can use them to introduce some invariants.

\begin{notation}
For each valued extension $(\ell,\va)$ of $(k,\va)$, we set 
\[ s(\ell) := \trdeg (\tilde \ell / \tilde k) \]
and
\[ t(\ell) := \dim_{\Q}(|\ell^\times|^\Q/|k^\times|^\Q).\]

For $x\in \aank$, we set 
\[ s(x) := s(\scrH(x)) \textrm{ and } t(x) := t(\scrH(x)). \]
\end{notation}

This invariants are related by the Abhyankar inequality (see \cite[VI, \S 10.3, Cor 1]{BourbakiAC57}).

\begin{theorem}\label{thm:Abhyankar}
Let $\ell$ be a valued extension of~$k$. Then, we have
\[s(\ell) + t(\ell) \le \trdeg (\ell / k).\]
Moreover, if $\ell/k$ is a finitely generated extension for which equality holds, then $|\ell^\times|/|k^\times|$ is a finitely generated abelian group and $\tilde\ell/\tilde k$ is a finitely generated field extension. 
\end{theorem}


For each $x\in \aank$, the fraction field of $k[T]/\ker(\va_{x})$ has degree of transcendence~0 or~1 over~$k$. Since its invariants~$s$ and~$t$ coincide with that of~$\scrH(x)$, it follows from Abhyankar's inequality that we have
\[s(x) + t(x) \le 1.\]

We can now state the classification of the points of the Berkovich affine line.


\begin{definition}\label{def:typesofpoints}
Let $x\in \aank$.

The point~$x$ is said to be of \emph{type~1} if it comes from a point in~$\wka$ in the sense of Example~\ref{ex:type1point}. In this case, we have $s(x) = t(x) = 0$.

The point~$x$ is said to be of \emph{type~2} if we have $s(x) = 1$ and $t(x) = 0$. 

The point~$x$ is said to be of \emph{type~3} if we have $s(x) = 0$ and $t(x) = 1$. 

The point~$x$ is said to be of \emph{type~4} otherwise. In this case, we have $s(x) = t(x) = 0$.
\end{definition}

\begin{example}\label{ex:typeetaalphar}
Let $\alpha\in k$ and $r\in \R_{>0}$. 

Assume that $r \in |k^\times|^\Q$. There exist $n,m \in \N_{\ge 1}$ and $\gamma \in k$ with $r^n = |c|^m$. Consider such an equality with~$n$ minimal. Denote by~$t$ the image of $(T-\alpha)^n/c^m$ in $\widetilde{\scrH(x)}$. It is transcendental over~$\tilde k$ and we have $\tilde k(t) = \widetilde{\scrH(x)}$. We deduce that $\eta_{\alpha,r}$ has type~2.

Assume that $r \notin |k^\times|^\Q$. Then, we have $\widetilde{\scrH(\eta_{\alpha,r})} = \tilde k$, so $\eta_{\alpha,r}$ has type~3.
\end{example}

The classification can be made more explicit when~$k$ is algebraically closed. Note that, in this case, we have $|k^\times|^\Q = |k^\times|$. 

\begin{lemma}\label{lem:type23algclosed}
Assume that~$k$ is algebraically closed. Then~$x$ has type~2 (resp.~3) if, and only if, there exist $\alpha\in k$ and $r\in |k^\times|^\Q$ (resp. $r\notin |k^\times|^\Q$) such that $x = \eta_{\alpha,r}$.
\end{lemma}
\begin{proof}
Assume that~$x$ is of type~2. Since $s(x)=1$, there exists $P \in k[T]$ such that $|P(x)|=1$ and $\tilde P$ is transcendental over~$\tilde k$. Since~$k$ is algebraically closed, we have $|k^\times| = |k^\times|^\Q = |\calH(x)^\times|^\Q$, hence we may write~$P$ as a product of linear polynomials, all of which have absolute value~1. One of these linear polynomials has an image in~$\widetilde{\scrH(x)}$ that is transcendental over~$\tilde k$. Write it as $c(T-\alpha)$, with $c\in k^\times$ and $\alpha \in k$. We then have $x = \eta_{\alpha,|c|^{-1}}$.

Assume that~$x$ is of type~3. Since $t(x)=1$, there exists $P \in k[T]$ such that $r := |P(x)| \notin |k^\times|^\Q$. As before, we may assume that $P = T-\alpha$ with $\alpha \in k$. We then have $x = \eta_{\alpha,r}$.

The converse implications are dealt with in Example~\ref{ex:typeetaalphar}.
\end{proof}

\begin{proposition}\label{prop:family}
Assume that $k$~is algebraically closed. Let $x\in \aank$. There exist a set~$I$, a family $(\alpha_{i})_{i\in I}$ of~$k$ and a family $(r_{i})_{i\in I}$ of~$\R_{\ge0}$ such that, for each $i,j \in I$, we have 
\[\max(|\alpha_{i}- \alpha_{j}|,r_{i}) \le r_{j} \text{ or } \max(|\alpha_{i}- \alpha_{j}|,r_{j}) \le r_{i}\] 
and, for each $P\in k[T]$, 
\[ |P|_{x} = \inf_{i\in I} (|P|_{\alpha_{i},r_{i}}). \]
\end{proposition}
\begin{proof}
Our set~$I$ will be the underlying set of~$k$. For each $a\in k$, we set $\alpha_{a} := a$ and $r_{a} := |T-a|_{x}$. 

Let $a,b \in k$. We have
\[|a-b| = |a-b|_{x} = |a-T + T-b|_{x} \le \max(|a-T|_{x}, |T-b|_{x}), \] 
so the first condition of the statement is satisfied. It implies that we have
\[\forall P \in k[T],\ |P|_{a,r_{a}} \le |P|_{b,r_{b}} \text{ or } \forall P \in k[T],\ |P|_{b,r_{b}} \le |P|_{a,r_{a}}.\]
It follows that the map $v \colon P \in k[T] \mapsto \inf_{a\in k} (|P|_{a,r_{a}})$ is multiplicative, hence a multiplicative seminorm.

Since~$k$ is algebraically closed, every polynomial factors as a product of monomials. As a consequence, to prove that~$v$ and~$\va_{x}$ coincide, it is enough to prove that they coincide on monomials, because of multiplicativity.

Let $\alpha\in k$. We have $|T-\alpha|_{\alpha,r_{\alpha}} = r_{\alpha} = |T-\alpha|_{x}$, hence $v(T-\alpha) \le |T-\alpha|_{x}$. On the other hand, for each $a\in k$, we have 
\[ |T-\alpha|_{x} = |T-a + a- \alpha|_{x} \le \max(|T-a|_{x}, |a-\alpha|_{x}) = \max(r_{a}, |a-\alpha|) = |T-\alpha|_{a,r_{a}},\] 
hence $|T-\alpha|_{x} \le v(T-\alpha)$.

\end{proof}

\begin{remark}\label{rem:intersection}
One should think of the families $(\alpha_{i})_{i\in I}$ and $(r_{i})_{i\in I}$ in the statement of Proposition~\ref{prop:family} as a single family of discs in~$k$ (with center~$\alpha_{i}$ and radius~$r_{i}$). Then, the condition of the statement translates into the fact that, for each pair of discs of the family, one is contained in the other. 

Moreover, it is not difficult to check that, if the intersection of this family of discs contains a point~$\alpha$ of~$k$, then we have $\va_{x} = \va_{\alpha,r}$, where $r = \inf_{i\in I}(r_{i}) \ge 0$.

On the other hand, if the family of discs has empty intersection, we find a new point, necessarily of type~4. Note that we must have $\inf_{i\in I} (r_{i}) >0$ is this case. Otherwise, the completeness of~$k$ would ensure that the intersection of the discs contains an element of~$k$.
\end{remark}

\begin{definition}
Assume that~$k$ is algebraically closed. For each $x\in \aank$, we define the \emph{radius of the point}~$x$ to be
\[r(x) := \inf_{c\in k} (|T-c|_{x}).\]
It can be thought of as the distance from the point to~$k$.
\end{definition}


\begin{example}
Assume that~$k$ is algebraically closed. Let $x\in \aank$. 

If~$x$ has type~1, then $r(x) = 0$.

If~$x$ has type~2 or~3, then, by Lemma~\ref{lem:type23algclosed}, it is of the form $x=\eta_{\alpha,r}$ and we have $r(x) = r$.

If~$x$ has type~4, then, with the notation of Proposition~\ref{prop:family}, we have $r(x) = \inf_{i\in I}(r_{i})$. Indeed, for each $i\in I$, we have $|T-\alpha_{i}|_{x} \le |T-\alpha_{i}|_{\alpha_{i},r_{i}} = r_{i}$. It follows that $r(x) \le \inf_{i\in I}(r_{i})$. On the other hand, let $c\in k$. For $i$ big enough, $c$ is not contained in the disc of center~$\alpha_{i}$ and radius~$r_{i}$, that is to say $|\alpha_{i} - c| > r_{i}$, from which it follows that $|T-c|_{\alpha_{i},r_{i}} = |\alpha_{i} - c|$. We deduce that $|T-c|_{x} = \inf_{i\in I} (|\alpha_{i} - c|) \ge \inf_{i\in I}(r_{i})$.
\end{example}

\begin{figure}[h]
\includegraphics[scale=.5]{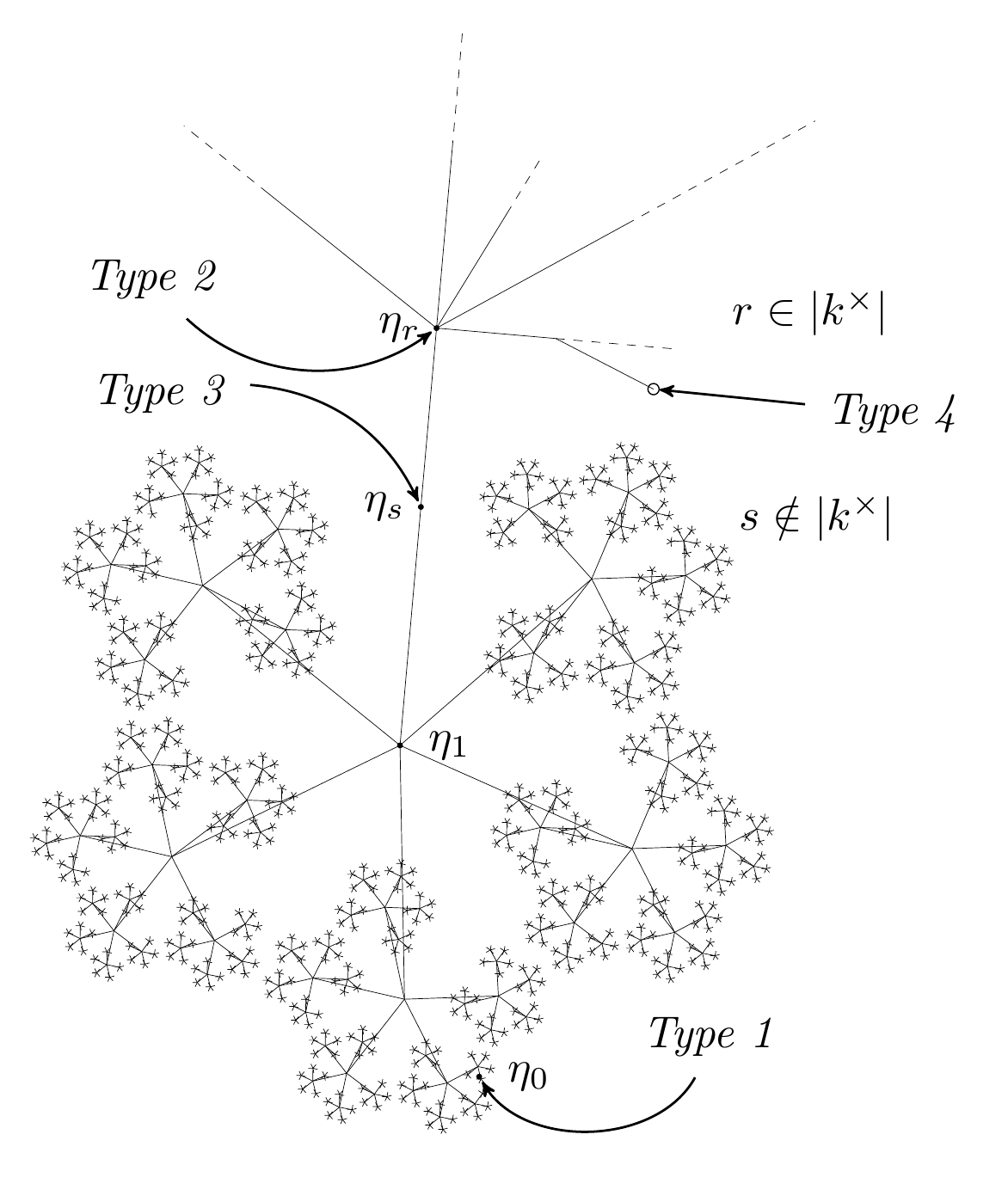}
\caption{The points $\eta_{\alpha,r}$ with $r\in |k^\times|$ are of type 2, the points $\eta_{\alpha,s}$ with $s\notin |k^\times|$ are of type 3, and the points $\eta_{\alpha,0}$ are of type 1. If $k$ is not spherically complete, points of type 4 will occur.}
\label{fig:DroiteTypePoints}
\end{figure}

\begin{remark}\label{rem:ratioradius}
The radius of a point of type different from~1 is not intrinsically attached to the point in the sense that it depends on the chosen coordinate~$T$ on~$\aank$. However, by studying the automorphisms of~$\aank$, one can prove that any change of coordinate will have the effect of multiplying all the radii by the same constant (in $|k^\times|$), see Proposition~\ref{prop:isodisc} and Remark~\ref{rem:isoaffineline}. In particular, the quotient of the radii of two points is well-defined.
\end{remark}

%
%
%

\begin{definition}
The field~$k$ is said \emph{spherically complete} if every family of discs that is totally ordered by inclusion has a non-empty intersection.

The field~$k$ is said \emph{maximally complete} if it has no non-trivial immediate extensions, \emph{i.e.} extensions with the same value group and residue field.
\end{definition}

We refer to the paper~\cite{PoonenMaximallyComplete} by B.~Poonen for more on those topics and in particular the construction of spherical completions, \emph{i.e.} minimal spherically complete extensions. We only quote the following important result.

\begin{theorem}\label{thm:equivalencesphmax}
A valued field is spherically complete if, and only if, it is maximally complete.
\qed
\end{theorem}

\begin{remark}\label{rem:maxsph}
Assume that~$k$ is algebraically closed. Then, the completed residue field of a point of type~4 is an immediate extension of~$k$. Using Remark~\ref{rem:intersection}, we can deduce a proof of Theorem~\ref{thm:equivalencesphmax} in this case.
\end{remark}

To make things more concrete, we would now like to give a rather explicit example of a point of type~4.

\begin{example}
Let $r \in (0,1)$ and consider the field of Laurent series $\C(\!(t)\!)$ endowed with the absolute value defined by $|f| = r^{v_{t}(f)}$. Recall that the $t$-adic valuation $v_{t}(f)$ of~$f$ is the infimum of the indices of the non-zero terms of~$f$ in its Taylor expansion $f = \sum_{n\in \Z} a_{n} t^n$. (Note that, for $f=0$, we have $v_{t}(0) = +\infty$, hence $|f|=0$.)

The algebraic closure of~$\C(\!(t)\!)$ is the field of Puiseux series:
\[\C(\!(t)\!)^a = \bigcup_{m\in\N_{\ge 1}} \C(\!(t^{1/m})\!).\]
In particular, the exponents of~$t$ in the expansion of any given element of $\C(\!(t)\!)^a$ are rational numbers with bounded denominators.

We choose our field~$k$ to be the completion of~$\C(\!(t)\!)^a$. Its elements may still be written as power series with rational exponents. This time, the exponents may have unbounded denominators but they need to tend to~$+\infty$.

Consider a power series of the form $\sum_{n\in \N} t^{q_{n}}$ where $(q_{n})_{n\in \N}$ is a strictly increasing bounded sequence of rational numbers. (For instance, $q_{n} = 1 - 2^{-n}$ would do.) The associated point of~$\aank$ is then a point of type~4. In this case, one can explicitly describe an associated family of discs by taking, for each $m\in \N$, the disc with center $\alpha_{m} := \sum_{n=0}^m t^{q_{n}}$ and radius $r_{m} := r^{q_{m+1}}$.

One can go even further in this case and describe a spherical completion of~$k$. It is the field of Hahn series $\C(\!(t^\Q)\!)$ consisting of the power series $f = \sum_{q\in \Q} a_{q} t^q$, where the $a_{q}$'s are rational numbers and the support $\{q\in \Q \mid a_{q} \ne 0\}$ of~$f$ is well-ordered: each non-empty subset of it has a smallest element.
\end{example}

\section{Topology}\label{sec:topology}

We endow the set $\aank$ with the coarsest topology such that, for each $P\in k[T]$, the map
\[x \in \aank \longmapsto |P(x)| \in \R\]
is continuous. In more concrete terms, a basis of the topology is given by the sets 
\[ \{x \in \aank : r < |P(x)| < s\}, \]
for $P\in k[T]$ and $r,s\in \R$.

\begin{remark}\label{rem:topologyk}
By Example~\ref{ex:rationalpoint}, we can see~$k$ as a subset of~$\aank$. The topology on~$k$ induced by that on~$\aank$ then coincides with that induced by the absolute value~$\va$.
\end{remark}

\begin{lemma}\label{lem:Hausdorff}
The Berkovich affine line~$\aank$ is Hausdorff.
\end{lemma}
\begin{proof}
Let $x \ne y \in \aank$. Then, there exists $P\in k[T]$ such that $|P(x)| \ne |P(y)|$. We may assume that $|P(x)| < |P(y)|$. Let $r \in (|P(x)|,|P(y)|)$. Set
\[U := \{z\in \aank : |P(z)|<r\} \textrm{ and } V := \{z\in \aank : |P(z)|>r\}.\]
The sets~$U$ and~$V$ are disjoint open subsets of~$\aank$ containing respectively~$x$ and~$y$. The result follows.
\end{proof}

\begin{definition}
For $\alpha\in k$ and $r\in \R_{>0}$, the \emph{open disc of center~$\alpha$ and radius~$r$} is
\[ D^-(\alpha,r) = \{ x\in \aank : |(T-\alpha)(x)| < r\}.\]
For $\alpha\in k$ and $r\in \R_{> 0}$, the \emph{closed disc of center~$\alpha$ and radius~$r$} is
\[ D^+(\alpha,r) = \{ x\in \aank : |(T-\alpha)(x)| \le r\}.\]
For $\alpha\in k$ and $r < s\in \R_{> 0}$, the \emph{open annulus of center~$\alpha$ and radii~$r$} and~$s$ is
\[ A^-(\alpha,r,s) = \{ x\in \aank : r < |(T-\alpha)(x)| < s\}.\]
For $\alpha\in k$ and $r \le s\in \R_{> 0}$, the \emph{closed annulus of center~$\alpha$ and radii~$r$} and~$s$ is
\[ A^+(\alpha,r,s) = \{ x\in \aank : r \le |(T-\alpha)(x)| \le s\}.\]
For $\alpha\in k$ and $r\in \R_{> 0}$, the \emph{flat closed annulus of center~$\alpha$ and radius~$r$} is
\[ A^+(\alpha,r,r) = \{ x\in \aank : |(T-\alpha)(x)| =r\}.\]

\end{definition}


\begin{figure}[h]

\includegraphics[scale=.7]{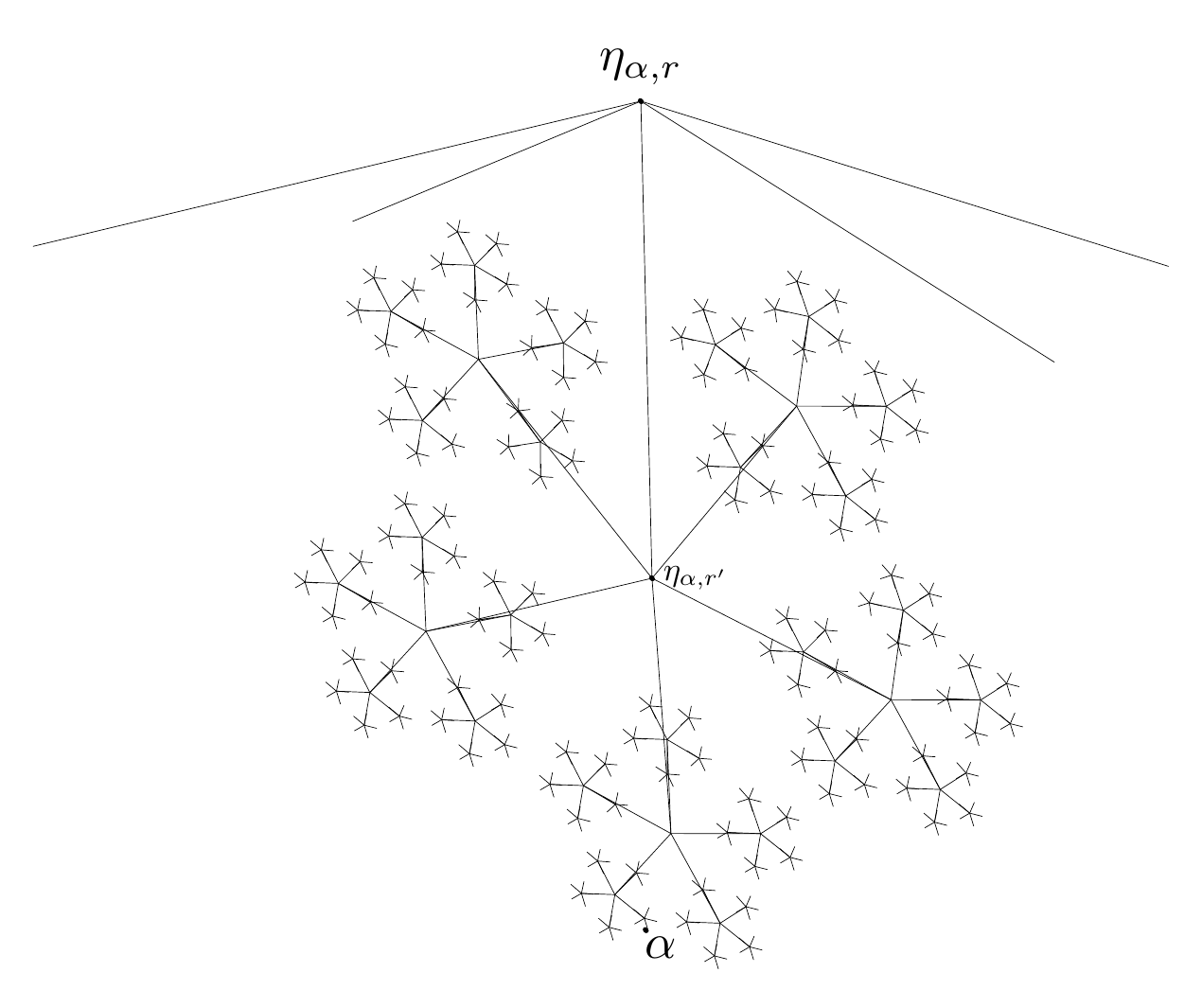}
\includegraphics[scale=.7]{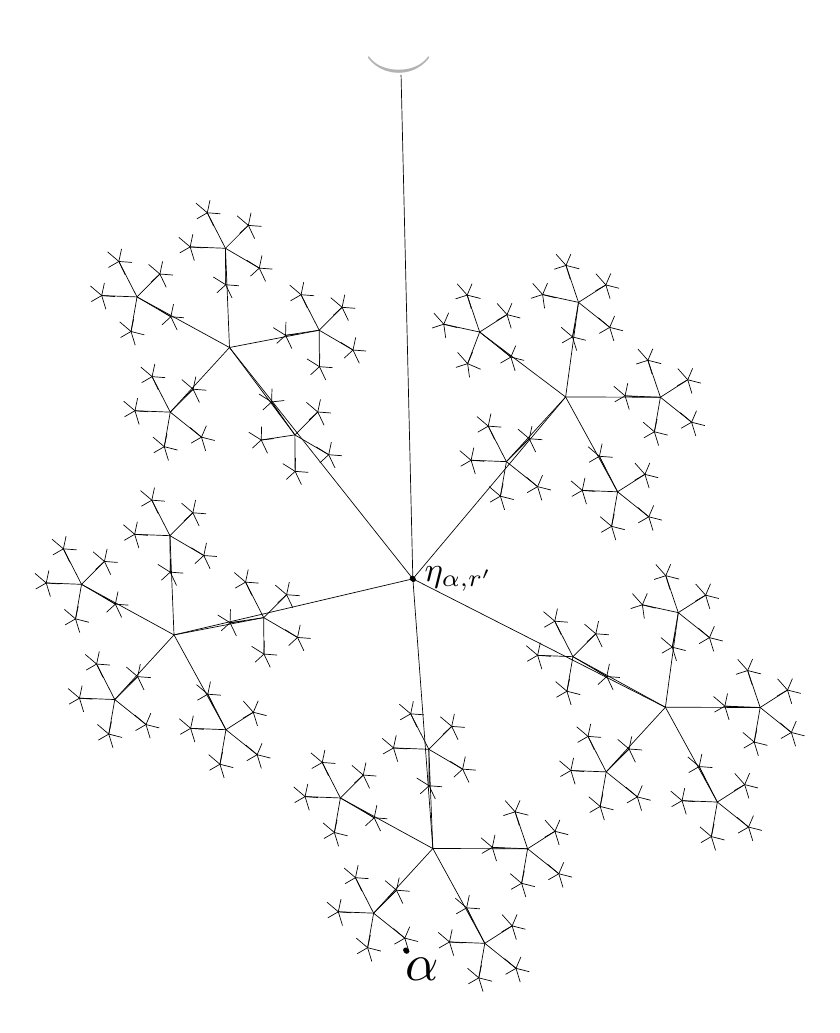}

\caption{On the left, the closed disc $D^+(\alpha, r)$. On the right, the open disc $D^-(\alpha, r)$, which is a maximal open sub-disc of $D^+(\alpha, r)$, but not the only one.}

\end{figure}

In the result that follows, we study the topology of discs and annuli as subsets of~$\aank$. 

\begin{lemma}\label{lem:discs}
Let $\alpha\in k$ and $r\in \R_{>0}$. The closed disc $D^+(\alpha,r)$ is compact and has a unique boundary point: $\eta_{\alpha,r}$. The open disc $D^-(\alpha,r)$ is open and its closure is $D^-(\alpha,r)\cup \{\eta_{\alpha,r}\}$.

Let $\alpha\in k$ and $r < s\in \R_{> 0}$. The closed annulus $A^+(\alpha,r,s)$ is compact and has two boundary points: $\eta_{\alpha,r}$ and $\eta_{\alpha,s}$. The open annulus $A^-(\alpha,r,s)$ is open and its closure is $A^-(\alpha,r,s)\cup \{\eta_{\alpha,r},\eta_{\alpha,s}\}$.

Let $\alpha\in k$ and $r \in \R_{> 0}$. The flat closed annulus $A^+(\alpha,r,r)$ is compact and has a unique boundary point: $\eta_{\alpha,r}$. 
\end{lemma}
\begin{proof}

Let $x\in D^+(\alpha,r)$. We have $|T-\alpha|_{x} \le r$, hence it follows from the non-Archimedean triangle inequality that we have $\va_{x} \le \va_{\alpha,r}$, as seminorms on $k[T]$. 

Consider the product $\prod_{P\in k[T]} [0,|P|_{r}]$ endowed with the product topology and its closed subset~$F$ consisting of the elements $(x_{P})_{P\in k[T]}$ satisfying the conditions
\[\begin{cases}
\forall P,Q \in k[T],\ \lambda_{P+Q} \le \max(\lambda_{P},\lambda_{Q});\\
\forall P,Q \in k[T],\ \lambda_{PQ} = \lambda_{P} \, \lambda_{Q}.
\end{cases}\]
It follows from the previous argument that the map
\[\renewcommand{\arraystretch}{1.2}\begin{array}{cccc}
p \colon &D^+(\alpha,r) & \too & \prod_{P\in k[T]} [0,|P|_{r}]\\
& x & \mapstoo & (|P|_{x})_{P \in k[T]}
\end{array}\]
induces a bijection between $D^+(\alpha,r) $ and~$F$. (The only non-trivial point is to check that the seminorm on~$k[T]$ associated to an element of~$F$ induces the given absolute value~$\va$ on~$k$.) Moreover, it follows from the very definition of the topology that~$p$ is a homeomorphism onto its image. Since $F$ is closed in $\prod_{P\in k[T]} [0,|P|_{r}]$, and the latter is compact by Tychonoff's theorem, $F$ is compact, hence $D^+(\alpha,r)$ is compact too.

\medbreak

Let $x \in D^+(\alpha,r) - \{\eta_{\alpha,r}\}$. Then, there exists $P \in k[T]$ such that $|P|_{x} \ne |P|_{\alpha,r}$, hence $|P|_{x} < |P|_{\alpha,r}$. In other words, the point~$x$ belong to the open subset $\{y\in \aank : |P|_{y} < |P|_{\alpha,r}\}$ of~$\aank$, which is contained in~$D^+(\alpha,r)$. It follows that~$x$ belongs to the interior of~$D^+(\alpha,r)$.

Let~$U$ be an open subset of~$\aank$ containing~$\eta_{\alpha,r}$. By definition of the topology, there exist $P_{1},\dotsc,P_{n} \in k[T]$ and $u_{1},v_{1},\dotsc,u_{n},v_{n} \in \R$ such that
\[\eta_{\alpha,r} \in \{y\in \aank : u_{i} < |P_{i}|_{y} < v_{i}\} \subseteq U.\]
Using the explicit definition of the norms~$\va_{\alpha,s}$, one shows that, for each~$s\in \R_{\ge 0}$ that is close enough to~$r$, we have $\eta_{\alpha,s} \in U$. We deduce that~$\eta_{\alpha,r}$ belongs to the boundary of~$D^+(\alpha,r)$ (because we can choose $s>r$) and to the closure of~$D^-(\alpha,r)$ (because we can choose $s<r$).  
 
This finishes the proof that the boundary of~$D^+(\alpha,r)$ is equal to~$\{\eta_{\alpha,r}\}$.

\medbreak

By definition of the topology, the disc $D^-(\alpha,r)$ is open. Since $D^+(\alpha,r)$ is compact, it contains the closure of~$D^-(\alpha,r)$. We have already proved that its closure contains~$\eta_{\alpha,r}$. 

Let $x\in D^+(\alpha,r) - (D^-(\alpha,r) \cup \{\eta_{\alpha,r}\})$. We have $|T-\alpha|_{x} = r$ and there exists $P\in k[T]$ such that $|P|_{x} < |P|_{\alpha,r}$. Let us choose such a polynomial~$P$ with minimal degree.

Arguing by contradiction, assume that $|P(\alpha)| < |P|_{\alpha,r}$. Write $P = P(\alpha) + (T-\alpha) Q$, with $Q \in k[T]$. We then have 
\[|P|_{\alpha,r} = \max(|P(\alpha)| , r |Q|_{\alpha,r}) = r |Q|_{\alpha,r}.\] 
If $|P(\alpha)| \ne r |Q|_{x}$, we have 
\[r |Q|_{x} \le\max(|P(\alpha)|, r |Q|_{x}) =  |P|_{x} < |P|_{\alpha,r} = r |Q|_{\alpha,r}.\]
If $|P(\alpha)| = r |Q|_{x}$, the same inequality holds. In any case, we have $|Q|_{x} < |Q|_{\alpha,r}$, which contradicts the minimality of the degree of~$P$. 

We have just proved that $|P(\alpha)| = |P|_{\alpha,r}$. It follows that, for each $y \in D^-(\alpha,r)$, we have $|P|_{y} = |P|_{\alpha,r}$, hence the open set  $\{y\in \aank : |P|_{y} < |P|_{\alpha,r}\}$ contains~$x$ and is disjoint from $D^-(\alpha,r)$, so~$x$ does not belong to the boundary of~$D^-(\alpha,r)$. 

We have finally proven that the closure of~$D^-(\alpha,r)$ is $D^-(\alpha,r) \cup \{\eta_{\alpha,r}\}$.

The results for the annuli are proven similarly.
\end{proof}

Since $\aank$ may be exhausted by closed discs, we deduce the following result.

\begin{corollary}\label{cor:loccompact}
The Berkovich affine line~$\aank$ is countable at infinity and locally compact. \qed
\end{corollary}

It is possible to give a characterization of the fields~$k$ for which the space~$\aank$ is metrizable.

\begin{corollary}\label{cor:metrizable}
The following assertions are equivalent:
\begin{enumerate}
\item the Berkovich affine line~$\aank$ is metrizable;
\item the field~$k$ contains a countable dense subset.
\end{enumerate}
\end{corollary}
\begin{proof}
$(i) \implies (ii)$ Assume that~$\aank$ is metrizable. We fix a metric on~$\aank$ and will consider balls with respect to it. Let~$(\eps_{n})_{n\in \N}$ be a sequence of positive real numbers converging to~0. 

Let $r \in \R_{>0}$. By Lemma~\ref{lem:discs}, the closed disc $D^+(0,r)$ is compact. As a consequence, for each $n\in\N$, it is covered by finitely many metric balls of radius~$\eps_{n}$. For each such ball that contains a point of~$k$, pick a point of~$k$ in it. The collection of those points is a finite subset~$k_{r,n}$ of~$k$. The set
$ k_{r} := \bigcup_{n\in \N} k_{r,n} $
is a countable subset of~$k$ that is dense in $k\cap D^+(0,r)$. 

It follows that the set 
$k' := \bigcup_{m\in \N_{\ge 1}} k_{m}$
is a countable dense subset of~$k$.

%
%
%

\medbreak

$(ii) \implies (i)$ Assume that the field~$k$ contains a countable dense subset~$k'$. Then, the family of sets 
\[ \{x\in \aank : r < |P(x)| < s\} \]
with $P \in k'[T]$ and $r,s \in \Q$ is a countable basis of the topology. The result now follows from Corollary~\ref{cor:loccompact} and Urysohn's metrization theorem.
\end{proof}

\begin{figure}[h]

\includegraphics[scale=.3]{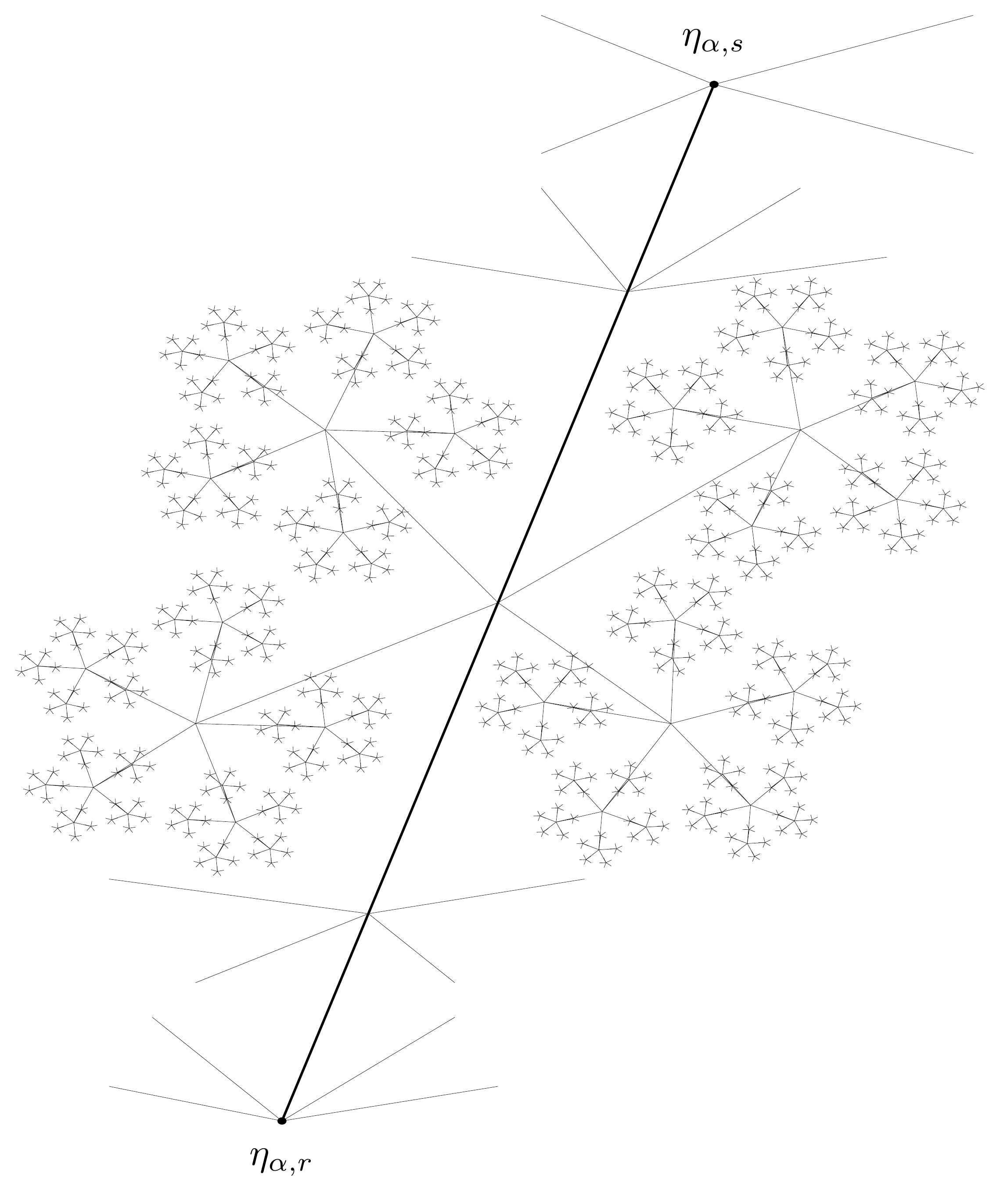}
\includegraphics[scale=.3]{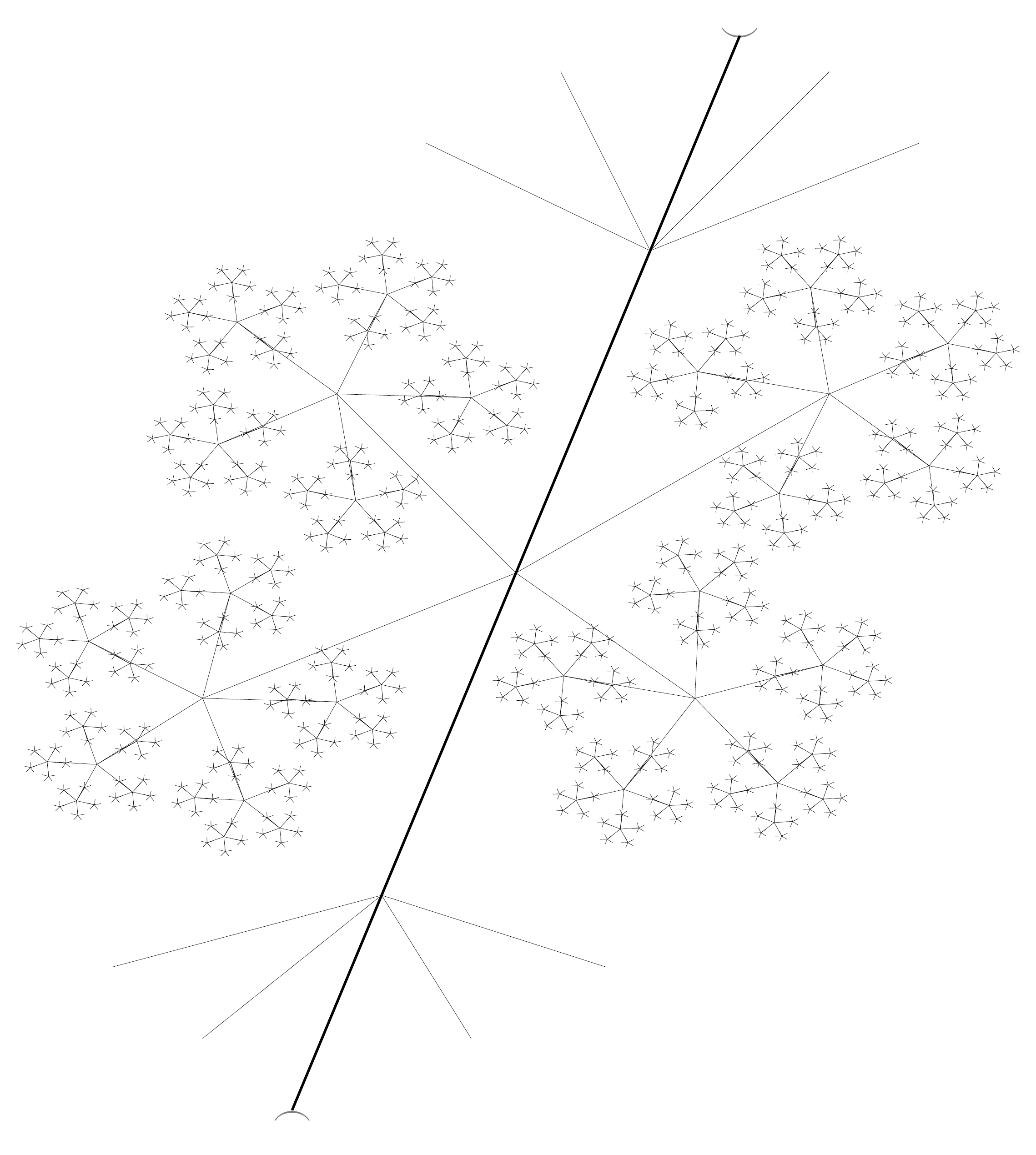}

\caption{On the left, the closed annulus $A^+(\alpha,r,s)$. On the right, the open annulus $A^-(\alpha,r,s)$, which is the unique maximal open sub-annulus of $A^+(\alpha,r,s)$.}

\end{figure}

By removing the boundary point of a closed disc, one may obtain either one or infinitely many discs, depending on the radius.
We will deal with this question assuming that~$k$ is algebraically closed and consider first the case where the radius does not belong to the value group of~$k$.

\begin{lemma}\label{lem:closeddisctype3}
Assume that~$k$ is algebraically closed. For each $\alpha\in k$ and $r \in \R_{>0} - |k^\times|$, we have 
\[D^+(\alpha,r) =  \{\eta_{\alpha,r}\}  \sqcup D^-(\alpha,r).\]
\end{lemma}
\begin{proof}
Let $x\in D^+(\alpha,r) - D^-(\alpha,r)$. We then have $|T-\alpha|_{x} = r$.

Let $P(T) = a_{d} T^d + \dotsb + a_{0} \in k[T]$. Since $k$~is algebraically closed, $|k^\times|$ is divisible and since $r \notin  |k^\times|$, all the terms $|a_{i}| r^i$ are distinct. It follows that
\[ |P|_{x} = \max_{1\le i\le d} (|a_{i}| r^i) = |P|_{\alpha,r}.\]
\end{proof}

We now handle the case of the disc $D^+(0,1)$. When~$k$ is algebraically closed, any disc of the form $D(\alpha,r)$ with $r\in |k^\times|$ may be turned into the latter by a suitable linear change of variable.

\begin{notation}
For each $u \in \tilde{k}$, we set $D^-(u,1) := D^-(\alpha,1)$, where~$\alpha$ is a lift of~$u$ in~$k^\circ$. 

Since any two lifts $\alpha_{1}$ and $\alpha_{2}$ satisfy $|\alpha_{1} - \alpha_{2}| < 1$, the definition does not depend on the choice of~$\alpha$.
\end{notation}

\begin{lemma}\label{lem:closeddisctype2}
Assume that~$k$ is algebraically closed. We have 
\[D^+(0,1) = \{\eta_{0,1}\}  \sqcup \bigsqcup_{u\in \tilde k} D^-(u,1).\] 
\end{lemma}
\begin{proof}
Let $u_{1} \in \tilde k$ and let $\alpha_{1} \in k^\circ$ such that $\tilde \alpha_{1} = u_{1}$. We have $|T-\alpha_{1}|_{0,1} = \max(1,|\alpha_{1}|) =1$, hence $\eta_{0,1} \notin D^-(\alpha_{1},1)$.

Let $u_{2} \ne u_{1} \in \tilde k$ and let $\alpha_{2} \in k^\circ$ such that $\tilde \alpha_{2} = u_{2}$. For each $x\in D^-(u_{2},1)$, we have 
\[|T-\alpha_{1}|_{x} = |(T-\alpha_{2}) + (\alpha_{2}-\alpha_{1})|_{x} =1,\]
since $|T-\alpha_{2}|_{x} < 1$ and $|\alpha_{2} - \alpha_{1}| <1$. It follows that $x \notin D^-(u_{1},1)$.

\medbreak

To finish, it remains to prove that $D^+(0,1) - \{\eta_{0,1}\}$ is covered by the discs $D^-(u,1)$ with $u\in \tilde k$. Let $x\in D^+(0,1) - \{\eta_{0,1}\}$. There exists $P \in k[T]$ such that $|P|_{x} \ne |P|_{0,1}$, hence $|P|_{x} < |P|_{0,1}$. Since~$k$ is algebraically closed, we may find such a~$P$ that is a monomial: $P = T-\alpha$ for some $\alpha\in k$. If $|\alpha| >1$, then we have 
$|T-\alpha|_{x} = |\alpha| = |T-\alpha|_{0,1}$,
which contradicts the assumption. We deduce that $|\alpha| \le 1$, hence 
$|T- \alpha|_{x} < |T-\alpha|_{0,1} = 1$
and $x \in D^-(\tilde \alpha,1)$.
\end{proof}

We now want to describe bases of neighborhoods of the points of~$\aank$, at least in the algebraically closed case. To do this, contrary to the usual complex setting, discs are not enough. We will also need annuli and even more complicated subsets.

\begin{definition}\label{def:Swisscheese}
An open (resp. closed) \emph{Swiss cheese}\footnote{This is called a ``standard set'' in \cite[Section~4.2]{Berkovich90}} over~$k$ is a non-empty subset of~$\aank$ that may be written as the complement of finitely many closed (resp. open) discs over~$k$ in an open (resp. a closed) disc over~$k$.
\end{definition}

\begin{proposition}\label{prop:types}
Assume that~$k$ is algebraically closed. Let $x\in \aank$.

If~$x$ has type~1 or~4, it admits a basis of neighborhoods made of discs.

If~$x$ has type~2, it admits a basis of neighborhoods made of Swiss cheeses.

If~$x$ has type~3, it admits a basis of neighborhoods made of annuli.
\end{proposition}
\begin{proof}
By definition of the topology, every neighborhood of~$x$ contains a finite intersection of sets of the form $\{u < |P| < v\}$ with $P\in k[T]$ and $u,v\in \R$. Since the sets in the statement are stable under finite intersections, it is enough to prove that each set of the form $\{u < |P| < v\}$ that contains~$x$ contains a neighborhood of~$x$ as described in the statement. 

Let $P\in k[T]$ and $u,v\in \R$ such that $|P(x)| \in (u,v)$. Write $P =  c \prod_{j=1}^m (T-\gamma_{j})$ with $c\in k^\times$ and $\gamma_{1},\dotsc,\gamma_{m}\in k$. 

Assume that~$x$ has type~1. Since~$k$ is algebraically closed, it is a rational point, hence associated to some $\alpha\in k$. One checks that $\{u < |P| < v\}$ then contains a disc of the form $D^-(\alpha,r)$ for $r \in \R_{>0}$.

Assume that~$x$ has type~3. By Lemma~\ref{lem:type23algclosed}, there exist $\alpha\in k$ and $r\in \R_{>0}$ such that $x=\eta_{\alpha,r}$. One checks that $\{u < |P| < v\}$ then contains an annulus of the form $A^-(\alpha,r_{1},r_{2})$ for some $r_{1},r_{2} \in \R_{>0}$ with $r \in (r_{1},r_{2})$.

Assume that~$x$ has type 4. By Proposition~\ref{prop:family} and Remark~\ref{rem:intersection}, it is associated to a family of closed discs $(D^+(\alpha_{i},r_{i}))_{i\in I}$ whose intersection contains no rational point. Because of this condition, there exists $i\in I$ such that $D^+(\alpha_{i},r_{i})$ contains none of the~$\gamma_{j}$'s. Then, for each $j \in \{1,\dotsc,m\}$, we have 
\[ |\alpha_{i} - \gamma_{j}| > r_{i}, \]
hence, for each $y \in D^+(\alpha_{i},r_{i})$, we have 
\[|T(y) - \gamma_{j}| = |T(y) - \alpha_{i} + \alpha_{i} - \gamma_{j}| = |\alpha_{i} - \gamma_{j}|.\]
We deduce that, for each $y \in D^+(\alpha_{i},r_{i})$, we have $|P(y)| = |P(x)|$ and the result follows.

Assume that~$x$ has type 2. By Lemma~\ref{lem:type23algclosed}, there exist $\alpha\in k$ and $r\in \R_{>0}$ such that $x = \eta_{\alpha,r}$. We have 
\[ |P(x)| = |c| \prod_{j=1}^m \max(r,|\alpha-\gamma_{j}|) < v, \]
hence there exists $\rho > r $ such that $|c| \prod_{j=1}^m \max(\rho,|\alpha-\gamma_{j}|) < v$. Then, for each $y\in D^-(\alpha,\rho)$, we have $|P(y)| < v$.

There exists $(\rho_{1},\dotsc,\rho_{m}) \in \prod_{j=1}^m (0,|T(x) - \gamma_{j}|)$ such that $|c| \prod_{j=1}^m \rho_{i} > u$. Then, for each $y \in \aank - \bigcup_{j=1}^m D^+(\gamma_{j},\rho_{j})$, we have $|P(y)|>u$. 

It follows that $D^-(\alpha,\rho) - \bigcup_{j=1}^m D^+(\gamma_{j},\rho_{j})$ is a neighborhood of~$x$ contained in $\{u < |P| < v\}$.
\end{proof}

\begin{remark}\label{rem:T2+1}
If $k$ is not algebraically closed, bases of neighborhoods of points may be more complicated. Let us give an example. Let $p$ be a prime number that is congruent to~3 modulo~4, so that $-1$ is not a square in~$\Q_{p}$. Consider the point~$x$ of~$\aank$ associated to a square root of~$-1$ in~$\C_{p}$. Equivalently, the point~$x$ is the unique point of~$\aank$ satisfying $|T^2+1|_{x} = 0$. 

The subset~$U$ of~$\aank$ defined by the inequality $|T^2+1| < 1$ is an open neighborhood of~$x$. It does not contains~0, so the function~$T$ is invertible on it and we may write
\[-1 = T^2 - (T^2+1) = T^2 \left(1 - \frac{T^2+1}{T^2}\right).\]
At each point~$y$ of~$U$, we have $|T^2+1|_{y}<1$, hence $|T^2|_{y} = 1$ and we deduce that~$-1$ has a square root on~$U$. In particular, $U$ contains no $\Q_{p}$-rational points and no discs.  
\end{remark}

Note that the topology of~$\aank$ is quite different from the topology of~$k$. We have already seen that~$\aank$ is always locally compact, whereas~$k$ is if, and only if, $|k^\times|$ is discrete and $\tilde k$~is finite. In another direction, $k$ is always totally disconnected, but $\aank$ contains paths, as the next result shows.

\begin{lemma}\label{lem:path}
Let $\alpha\in k$. The map
\[r \in \R_{\ge 0} \mapstoo \eta_{\alpha,r} \in \aank\]
is a homeomorphism onto its image~$I_{\alpha}$.
\end{lemma}
\begin{proof}
It is clear that the map is injective and open, so to prove the result, it is enough to prove that it is continuous. By definition, it is enough to prove that, for each $P \in k[T]$, the map $r\in \R_{\ge 0} \mapsto |P|_{\alpha,r} \in \R_{\ge 0}$ is continuous. The result then follows from the explicit description of~$\va_{\alpha,r}$ (see Example~\ref{ex:genericpointdisc}). 
\end{proof}

\begin{remark}\label{rem:pathsmeet}
One may use the paths from the previous lemma to connect the points of~$k$. Let $\alpha,\beta \in k$ and consider the paths~$I_{\alpha}$ and~$I_{\beta}$. Example~\ref{ex:genericpointdisc} tells us that they are not disjoint but meet at the point $\eta_{\alpha,|\alpha-\beta|} = \eta_{\beta,|\alpha-\beta|}$ (and actually coincide from this point on). The existence of a path from~$\alpha$ to~$\beta$ inside~$\aank$ follows.


We will use this construction in Section~\ref{sec:connected} to show that~$\aank$ is path-connected.
\end{remark}

\section{Analytic functions}\label{sec:analyticfunctions}

So far, we have described the Berkovich affine line~$\aank$ as a topological space. It may actually be endowed with a richer structure, since we may define analytic functions over it.

\begin{definition}\label{def:analyticfunction}
Let~$U$ be an open subset of~$\aank$. An \emph{analytic function} on~$U$ is a map
\[F \colon U \to \bigsqcup_{x\in U} \scrH(x)\]
such that, for each $x\in U$, the following conditions hold:
\begin{enumerate}
\item $F(x) \in \scrH(x)$;
\item there exist a neighborhood~$V$ of~$x$ and sequences $(P_{n})_{n\in \N}$ and $(Q_{n})_{n\in \N}$ of elements of~$k[T]$ such that the $Q_{n}$'s do not vanish on~$V$ and 
\[\lim_{n\to +\infty} \sup_{y\in V} \Big(\Big| F(y) - \frac{P_{n}(y)}{Q_{n}(y)}\Big|\Big) =0.\]
\end{enumerate}
\end{definition}

\begin{remark}
The last condition can be reformulated by saying that~$F$ is locally a uniform limit of rational functions without poles, which then makes the definition similar to the usual complex one (where analytic functions are locally uniform limits of polynomials).
\end{remark}

The Berkovich affine line~$\aank$ together with its sheaf of analytic functions~$\calO$ now has the structure of a locally ringed space. It satisfies properties that are similar to those of the usual complex analytic line~$\C$. We state a few of them here without proof.

The set of global analytic functions on some simple open subsets of~$\aank$ may be described explicitly.

\begin{proposition}\label{prop:Odisc}
Let $r \in \R_{>0}$. Then $\calO(D^-(0,r))$ is the set of elements
\[ \sum_{i \in \N} a_{i} \, T^i \in k\llbracket T \rrbracket\]
with radius of convergence greater than or equal to~$r$:
\[ \forall s \in (0,r),\ \lim_{i \to +\infty} |a_{i}|\, s^i = 0.\]
\qed
\end{proposition}

\begin{corollary}\label{cor:O0}
The local ring~$\calO_{0}$ at the point~0 of~$\aank$ consists of the power series in $ k\llbracket T \rrbracket$ with positive radius of convergence.
\end{corollary}
\begin{proof}
It follows from Proposition~\ref{prop:Odisc} and the fact that the family of discs~$D^-(0,r)$, with $r>0$, forms a basis of neighborhoods of~0 in~$\aank$ (see the proof of Proposition~\ref{prop:types}).
\end{proof}

\begin{corollary}\label{cor:O0}
Assume that $k$ is algebraically closed. Let~$x \in \aank$ be a point of type 4, associated to a family of closed discs $(D^+(\alpha_{i},r_{i}))_{i\in I}$ as in Proposition~\ref{prop:family}. The local ring~$\calO_{x}$ at the point~$x$ of~$\aank$ consists of the union over $i\in I$ of the sets of power series in $ k\llbracket T-\alpha_{i} \rrbracket$ with radius of convergence bigger than or equal to~$r_{i}$.
\end{corollary}
\begin{proof}
It follows from Proposition~\ref{prop:Odisc} and the fact that the family of discs~$(D^+(\alpha_{i},r_{i}))_{i\in I}$, with $i\in I$, forms a basis of neighborhoods of~$x$ in~$\aank$ (see the proof of Proposition~\ref{prop:types}).
\end{proof}

\begin{proposition}\label{prop:Oannulus}
Let $r,s \in \R_{>0}$ with $r < s$. Then $\calO(A^-(0,r,s))$ is the set of elements
\[ \sum_{i \in \Z} a_{i} \, T^i \in k\llbracket T,T^{-1} \rrbracket\]
satisfying the following condition:
\[ \forall t \in (r,s),\ \lim_{i \to \pm \infty} |a_{i}|\, t^i = 0.\]
\qed
\end{proposition}

\begin{corollary}\label{cor:Oetat}
Let $t \in \R_{>0} - |k^\times|$. The local ring~$\calO_{\eta_{t}}$ at the point~$\eta_{t}$ of~$\aank$ is the set of elements
\[ \sum_{i \in \Z} a_{i} \, T^i \in k\llbracket T,T^{-1} \rrbracket\]
satisfying the following condition:
\[ \exists t_{1},t_{2} \in \R_{>0} \textrm{ with } t_{1} < t < t_{2} \textrm{ such that } \lim_{i \to -\infty} |a_{i}|\, t_{1}^i =  \lim_{i \to +\infty} |a_{i}|\, t_{2}^i  = 0.\]
\qed
\end{corollary}
\begin{proof}
It follows from Proposition~\ref{prop:Oannulus} and the fact that the family of annuli~$A^-(0,t_{1},t_{2})$, with $t_{1}<t<t_{2}$, forms a basis of neighborhoods of~$\eta_{t}$ in~$\aank$ (see the proof of Proposition~\ref{prop:types}).
\end{proof}


\begin{remark}
The local rings at points of type~2 of~$\aank$ do not admit descriptions as simple as that of the other points, due to the fact that they have more complicated bases of neighborhoods (see Proposition~\ref{prop:types}). Over an algebraically closed field~$k$, one may still obtain a rather concrete statement as follows. Let~$C$ be a set of lifts of the elements of~$\tilde k$ in~$k^\circ$. Then, the local ring~$\calO_{\eta_{1}}$ at the point~$\eta_{1}$ of~$\aank$ consists in the functions that may be written as a sum of power series with coefficients in~$k$ of the form
\[ \sum_{i\in\N} a_{i}\, T^i + \sum_{c \in C_{0}} \sum_{i\in \N_{\ge 1}} a_{c,i} \, (T-c)^{-i},\]
where $C_{0}$ is a finite subset of~$C$ and the series all have radius of convergence strictly bigger than~1. We refer to~\cite[Proposition~2.2.6]{FresnelPut04} for details.
\end{remark}

%

%
%

Let us now state some properties of the local rings, \emph{i.e.} germs of analytic functions at one point. They are easily seen to hold for~$\calO_{0}$ using its explicit description as a ring of power series (see Proposition~\ref{prop:Odisc}).


\begin{proposition}
Let $x\in \aank$. The ring~$\calO_{x}$ is a local ring with maximal ideal
\[\mathfrak{m}_{x} = \{F \in \calO_{x} : F(x) = 0\}.\]
The quotient $\calO_{x}/\mathfrak{m}_{x}$ is naturally a dense subfield of~$\scrH(x)$.

If~$x$ is a rigid point (see Example~\ref{ex:rigidpoint}), then $\calO_{x}$ is a discrete valuation ring that is excellent and henselian. Otherwise, $\calO_{x}$ is a field.
\qed
\end{proposition}

\begin{remark}
The existence of the square of~$-1$ in Remark~\ref{rem:T2+1} may be reproved using Henselianity. With the notation of that remark, we have $\calO_{x}/\mathfrak{m}_{x} = \scrH(x) = k[T]/(T^2+1)$, hence the residue field $\calO_{x}/\mathfrak{m}_{x}$ contains a root of~$T^2+1$. Since we are in characteristic~0, this root is simple, hence, by Henselianity, it lifts to a root of~$T^2+1$ in~$\calO_{x}$.
\end{remark}

The next step is to define a notion of morphism between open subsets of~$\aank$. As one should expect, such a morphism $\varphi \colon U \to V$ underlies a morphism of locally ringed spaces (hence a morphism of sheaves $\varphi^\sharp \colon \calO_{V} \to \varphi_{*} \calO_{U}$), but the precise definition is more involved since we want the seminorms associated to the points of~$U$ and~$V$ to be compatible. For instance, for each $x\in U$, we want the map $\calO_{\varphi(x)}/\mathfrak{m}_{\varphi(x)} \to \calO_{x}/\mathfrak{m}_{x}$ induced by~$\varphi^\sharp$ to be an isometry with respect to~$\va_{\varphi(x)}$ and~$\va_{x}$ (so that it induces an isometry $\scrH(\varphi(x)) \to \scrH(x)$). We will not dwell on those questions, which would lead us too far for this survey. Anyway, in the rest of the text, we will actually make only a very limited use of morphisms.

Let us mention that, as in the classical theories, global sections of the structure sheaf correspond to morphisms to the affine line.

\begin{lemma}\label{lem:morphismtoA1}
Let~$U$ be an open subset of~$\aank$. Then, the map
\[ \begin{array}{ccc}
\Hom(U,\aank) & \longrightarrow & \calO(U)\\
\varphi & \longmapsto & \varphi^\sharp(T)
\end{array}\]
is a bijection.
\qed
\end{lemma}

For later use, we record here the description of isomorphisms of discs and annuli. Once the rings of global sections are known (see Proposition~\ref{prop:Odisc} and~\ref{prop:Oannulus}), those results are easily proven using the theory of Newton polygons (or simple considerations on the behaviour of functions as in Section~\ref{sec:variation}). We refer to \cite[3.6.11 and 3.6.12]{DucrosRSS} for complete proofs.

\begin{proposition}\label{prop:isodisc}
Let $r_{1},r_{2} \in \R_{>0}$. Let $\varphi \colon D^-(0,r_{1}) \to D^-(0,r_{2})$ be an isomorphism such that $\varphi(0) = 0$. Write 
\[ \varphi^\sharp(T) = \sum_{i \in \N_{\ge 1}} a_{i} T^i \in k\llbracket T \rrbracket\]
using Proposition~\ref{prop:Odisc}. Then, we have
\[\forall s \in (0,r_{1}),\ |a_{1}|\, s > \sup_{i \ge 2} (|a_{i}| \,s^i).\]
In particular, we have $r_{2} = |a_{1}| \, r_{1}$ and, for each $s \in [0,r_{1})$, $\varphi(\eta_{s}) = \eta_{|a_{1}| s}$.
\qed
\end{proposition}

\begin{remark}\label{rem:isoaffineline}
The previous result still holds if we allow the radii to be infinite, considering the affine line as the disc of infinite radius. 
\end{remark}

\begin{proposition}\label{prop:isoannulus}
Let $r_{1},s_{1},r_{2},s_{2} \in \R_{>0}$ with $r_{1}<s_{1}$ and $r_{2}<s_{2}$. Let $\varphi \colon A^-(0,r_{1},s_{1}) \to A^-(0,r_{2},s_{2})$ be an isomorphism. Write 
\[ \varphi^\sharp(T) = \sum_{i \in \Z} a_{i} T^i \in k\llbracket T,T^{-1} \rrbracket\]
using Proposition~\ref{prop:Oannulus}. Then, there exists $i_{0} \in \{-1,1\}$ such that we have
\[\forall t \in (r_{1},s_{1}),\ |a_{i_{0}}|\, s^{i_{0}} > \sup_{i \ne i_{0}} (|a_{i}| \, s^i).\]
In particular, if $i_{0} = 1$, we have $r_{2} = |a_{1}| \, r_{1}$, $s_{2} = |a_{1}| \, s_{1}$ and, for each $t \in (r_{1},s_{1})$, $\varphi(\eta_{s}) = \eta_{|a_{1}| s}$. If $i_{0} = -1$, we have $r_{2} = |a_{1}|/s_{1}$, $s_{2} = |a_{1}|/r_{1}$ and, for each $t \in (r_{1},s_{1})$, $\varphi(\eta_{s}) = \eta_{|a_{1}|/s}$.
\qed
\end{proposition}

\begin{remark}
The previous result still holds if we allow $r_{1}$ or~$r_{2}$ to be~0 and $s_{1}$ or~$s_{2}$ to be infinite. 
\end{remark}

\begin{definition}\label{def:modulus}
Let $A = A^-(\alpha,r,s)$ be an open annulus. We define the \emph{modulus} of~$A$ to be 
\[ \Mod(A) := \frac s r \in (1,+\infty).\]
\end{definition}

By Proposition~\ref{prop:isoannulus}, the modulus of an annulus only depends on its isomorphism class and not on the coordinate chosen to describe it.

\section{Extension of scalars}\label{sec:extension}

Let $(K,\va)$ be a complete valued extension of~$(k,\va)$. The ring morphism $k[T] \to K[T]$ induces a map
\[\pi_{K/k} \colon \AA^{1,\an}_{K} \too \aank\]
called the \emph{projection map}. In this section, we study this map.

%

\begin{proposition}\label{prop:piKk}
Let~$K$ be a complete valued extension of~$k$. 
The projection map~$\pi_{K/k}$ is continuous, proper and surjective. 
\end{proposition}
\begin{proof}
The map~$\pi_{K/k}$ is continuous as a consequence of the definitions. To prove that it is proper, note that the preimage of a closed disc in~$\aank$ is a closed disc in~$\AA^{1,\an}_{K}$ with the same center and radius, hence a compact set by Lemma~\ref{lem:discs}.

It remains to prove that~$\pi_{K/k}$ is surjective. Let~$x\in \aank$ and consider the associated character $\chi_{x} \colon k[T] \to \scrH(x)$. By Lemma~\ref{lem:commonextension}, there exists a complete valued field~$L$ containing both~$K$ and~$\scrH(x)$. The character~$\chi_{x}$ over~$k$ induces a character over~$K$ given by 
\[K[T] \xrightarrow[]{\chi_{x}\otimes K} \scrH(x)\otimes_{k}K \to L.\]
The associated point of~$\AA^{1,\an}_{K}$ belongs to $\pi_{K/k}^{-1}(x)$.
\end{proof}

The following result shows that the fibers of the projection map may be quite big. 

\begin{lemma}\label{lem:fibertype4}
Let~$K$ be a complete valued extension of~$k$. 
Assume that~$k$ and~$K$ are algebraically closed and that~$K$ is maximally complete. Let $x\in \aank$ be a point of type~4. Then $\pi_{K/k}^{-1}(x)$ is a closed disc of radius~$r(x)$.
\end{lemma}
\begin{proof}
Fix notation as in Proposition~\ref{prop:family}. We have $r(x) = \inf_{i\in I}(r_{i})$. Since~$K$ is maximally complete, the intersection of all the discs $D^+(\alpha_{i},r_{i})$ in~$\AA^{1,\an}_{K}$ contains a point $\gamma \in K$. 

Let us prove that $\pi_{K/k}^{-1}(x) = D^+(\gamma,r(x))$ in $\A^{1,\an}_{K}$.

Let $c\in k$. For $i$ big enough, $c$ is not contained in the disc of center~$\alpha_{i}$ and radius~$r_{i}$, that is to say $|\alpha_{i} - c| > r_{i}$. It follows that, for each $y \in D^+(\alpha_{i},r_{i})$, we have 
\[|T-c|_{y} = |\alpha_{i} - c| = |T-c|_{\alpha_{i},r_{i}}.\] 

We have $D^+(\gamma,r(x)) \subseteq \bigcap_{i\in I} D^+(\alpha_{i},r_{i})$. It then follows from the previous argument that, for each $z \in D^+(\gamma,r(x))$ and each $c\in k$, we have
\[|T-c|_{z} = \inf_{i\in I} (|\alpha_{i}-c|) = |T-c|_{x}.\]
Since~$k$ is algebraically closed, every polynomial is a product of monomials and we deduce that $D^+(\gamma,r(x)) \subseteq \pi_{K/k}^{-1}(x)$.

Let $y\in \AA^{1,\an}_{K} \setminus D^+(\gamma,r(x))$. Then, there exists $i\in I$ such that 
\[|T-\alpha_{i}|_{y} > r_{i} = |T-\alpha_{i}|_{\alpha_{i},r_{i}} \ge |T-\alpha_{i}|_{x}.\]
The result follows.
\end{proof}

\begin{remark}\label{rem:intersectiontype4}
Fix notation as in Proposition~\ref{prop:family}. The preceding proof also shows that we have $\bigcap_{i\in I} D^+(\alpha_{i},r_{i}) = \{x\}$ in~$\aank$.
\end{remark}

\begin{remark}
The preceding lemma shows that the projection map is not open and does not preserve the types of points in general. 
\end{remark}

We now deal more specifically with the case of Galois extensions. Let $\sigma\in \Aut(K/k)$ and assume that it preserves the absolute value on~$K$. (This condition is automatic if $K/k$ is algebraic.) For each $x\in \AA^{1,\an}_{K}$, the map
\[P \in  K[T] \mapstoo |\sigma(P)|_{x} \in \R_{\ge 0}\]
is a multiplicative seminorm. We denote by~$\sigma(x)$ the corresponding point of~$\AA^{1,\an}_{K}$.

\begin{proposition}\label{prop:piGalois}
Let~$K$ be a finite Galois extension of~$k$. The map 
\[(\sigma,x) \in  \Gal(K/k) \times \AA^{1,\an}_{K} \mapstoo \sigma(x) \in \AA^{1,\an}_{K}\]
is continuous and proper. 

The projection map $\pi_{K/k}$ induces a homeomorphism
\[ \AA^{1,\an}_{K}/\Gal(K/k) \simtoo  \aank .\]
In particular, $\pi_{K/k}$ is continuous, proper, open and surjective.
\end{proposition}
\begin{proof}
To prove the first continuity statement, it is enough to prove that, for each $P \in K[T]$, the map 
\[(\sigma,x) \in  \Gal(K/k) \times \AA^{1,\an}_{K} \mapstoo |P(\sigma(x))| = |\sigma(P)|_{x} \in \R_{\ge 0}\]
is continuous.

Let $P \in K[T]$. We may assume that $P\ne 0$. For each $\sigma,\tau \in \Gal(K/k)$ and each $x,y \in \AA^{1,\an}_{K}$, we have
\begin{align*}
\big| |\sigma(P)|_{x} - |\tau(P)|_{y} \big|  & \le  \big| |\sigma(P)|_{x} - |\sigma(P)|_{y} \big| +  \big| |\sigma(P)|_{y} - |\tau(P)|_{y} \big|\\
& \le \big| |\sigma(P)|_{x} - |\sigma(P)|_{y} \big| +  \|(\sigma - \tau)(P)\|_{\infty} \, \max(1,|T|_{y})^{\deg(P)},
\end{align*}
where, for each $R \in K[T]$, we denote by $\|R\|_{\infty}$ the maximum of the absolute values of its coefficients. The continuity now follows from the continuity of the maps $z \in  \AA^{1,\an}_{K} \mapsto |\sigma(P)|_{z}$ and $\sigma \in \Gal(K/k) \mapsto \sigma(c)$, for $c\in K$.

\medbreak

Let us now prove properness. Since any compact subset of $\AA^{1,\an}_{K}$ is contained in a disc of the form $D^+(0,r)$ for some~$r\in \R_{> 0}$, it is enough to prove that the preimage of such a disc is compact. Since this preimage is equal to $\Gal(K/k) \times D^+(0,r)$, the result follows from the compactness of $D^+(0,r)$ (see Lemma~\ref{lem:discs}).

\medbreak

We now study the map~$\pi_{K/k}$. It is continuous and proper, by Proposition~\ref{prop:piKk}.


Let~$x\in \aank$ and consider the associated character $\chi_{x} \colon k[T] \to \scrH(x)$. It induces a morphism of $K$-algebras $\chi_{K,x} \colon K[T] \to \scrH(x)\otimes_{k} K$.

Let~$\alpha \in K$ be a primitive element for~$K/k$. Denote by~$P$ its minimal polynomial over~$k$ and by $P_{1},\dotsc,P_{r}$ the irreducible factors of~$P$ in~$\scrH(x)[T]$. For each $i\in \{1,\dotsc,r\}$, let~$L_{i}$ be an extension of~$\scrH(x)$ generated by a root of~$P_{i}$. We then have an isomorphism of $K$-algebras $\varphi \colon \scrH(x)\otimes_{k}K \xrightarrow[]{\sim} \prod_{i=1}^r L_{i}$.


For each $i\in\{1,\dotsc,r\}$, by composing $\varphi \circ \chi_{K,x}$ with the projection on the $i^\text{th}$ factor, we get a character $\chi_{i} \colon K[T] \to L_{i}$. Denote by~$y_{i}$ the associated point of~$\AA^{1,\an}_{K}$. We have $\pi_{K/k}(y_{i}) = x$. 

Conversely, let $y\in \AA^{1,\an}_{K}$ such that $\pi_{K/k}(y) = x$. The field~$\scrH(y)$ is an extension of both~$\scrH(x)$ and~$K$, so the universal property of the tensor product yields a natural morphism $\scrH(x)\otimes_{k}K \to \scrH(y)$. It follows that there exists $i\in\{1,\dotsc,r\}$ such that~$\scrH(y)$ is an extension of~$L_{i}$, and we then have $y=y_{i}$. We have proven that $\pi^{-1}_{K/k}(x) = \{y_{1},\dotsc,y_{r}\}$. 

Since~$P$ is irreducible in $k[T]$, the group $\Gal(K/k)$ acts transitively on its roots, hence on the~$P_{i}$'s. It follows that, for each $i,j \in \{1,\dotsc,r\}$, there exists $\sigma \in \Gal(K/k)$ such that $\chi_{i} \circ \sigma = \chi_{j}$, hence $\sigma(y_{i}) = y_{j}$. We have proven that $\pi^{-1}_{K/k}(x)$ is a single orbit under $\Gal(K/k)$.

The arguments above show that the projection map $\pi_{K/k} \colon \AA^{1,\an}_{K} \to \aank$ 
factors through a map $\pi_{K/k}' \colon \AA^{1,\an}_{K}/\Gal(K/k) \to \aank$ and that the latter is continuous and bijective. Since~$\pi_{K/k}$ is proper, $\pi_{K/k}'$ is proper too, hence a homeomorphism. 
\end{proof}

Recall that every element of $\Gal(k^s/k)$ preserves the absolute value on~$k^s$. It particular, it extends by continuity to an automorphism of $\widehat{k^s} = \widehat{k^a}$. We endow $\Gal(k^s/k)$ with its usual profinite topology.

\begin{corollary}\label{cor:piwkak}
The map 
\[(\sigma,x) \in  \Gal(k^s/k) \times \AA^{1,\an}_{\wka} \mapsto \sigma(x) \in \AA^{1,\an}_{\wka}\]
is continuous and proper. 

The projection map $\pi_{\wka/k}$ induces a homeomorphism
\[ \AA^{1,\an}_{\wka}/\Gal(k^s/k) \xrightarrow[]{\sim} \aank .\]
In particular, $\pi_{\wka/k}$ is continuous, proper, open and surjective.
\end{corollary}
\begin{proof}
The first part of the statement is proven as in Proposition~\ref{prop:piGalois}, using the fact that the group $\Gal(k^s/k)$ is compact.

Let $x,y \in \AA^{1,\an}_{\wka}$ such that $\pi_{\wka/k}(x) = \pi_{\wka/k}(y)$. By Proposition~\ref{prop:piGalois}, for each finite Galois extension~$K$ of~$k$, there exists~$\sigma_{K}$ in~$\Gal(K/k)$ such that $\sigma_{K}(\pi_{K/k}(x)) = \pi_{K/k}(y)$. By compactness, the family $(\sigma_{K})_{K}$ admits a subfamily converging to some $\sigma \in \Gal(k^s/k)$. We then have
\[\forall P \in k^s[T],\ |\sigma(P)|_{x} = |P|_{y}.\]

Let $Q\in \wka[T]$. We want to prove that we still have $|\sigma(Q)|_{x} = |Q|_{y}$. We may assume that~$Q$ is non-zero. Let~$d$ be its degree. By density of~$k^s$ into~$\wka$, there exists a sequence $(Q_{n})_{n\ge 0}$ of elements of~$k^s[T]$ of degree~$d$ that converge to~$Q$ for the norm~$\nm_{\infty}$ that is the supremum norm of the coefficients. Note that we have
\[|Q - Q_{n}|_{y} \le \|Q-Q_{n}\|_{\infty} \, \max(1,|T|_{y})^d,\]
so that $(|Q_{n}|_{y})_{n\ge 0}$ converges to~$|Q|_{y}$. The same argument shows that $(|\sigma(Q_{n})|_{x})_{n\ge 0}$ converges to~$|\sigma(Q)|_{x}$ and the results follows.

We have just proven that the map $\AA^{1,\an}_{\wka}/\Gal(k^s/k) \to \aank$ induced by~$\pi_{\wka/k}$ is a bijection. The rest of the statement follows as in the proof of Proposition~\ref{prop:piGalois}. 
\end{proof}

\begin{lemma}
Let $y,z\in \aanwka$ such that $\pi_{\wka/k}(y) = \pi_{\wka/k}(z)$. Then $y$ and~$z$ are of the same type and have the same radius.
\end{lemma}
\begin{proof}
By Corollary~\ref{cor:piwkak}, there exists $\sigma \in \Gal(k^s/k)$ such that $z = \sigma(y)$. The result follows easily.
\end{proof}

As a consequence, we may define the type and the radius of a point of the Berkovich affine line over any complete valued field. 

\begin{definition}
Let $x \in \aank$. 

We define the \emph{type of the point~$x$} to be the type of the point~$y$, for any $y \in \pi_{\wka/k}^{-1}(x)$.

We define the \emph{radius of the point~$x$} to be the radius of the point~$y$, for any $y \in \pi_{\wka/k}^{-1}(x)$. We denote it by~$r(x)$.
\end{definition}

We end this section with a finiteness statement that is often useful.

\begin{lemma}\label{lem:finitefiber}
Let $X$ be a subset of~$\aanwka$ that is either a disc, an annulus or a singleton containing a point of type~2 or~3. Then, the orbit of~$X$ under~$\Gal(k^s/k)$ is finite.

In particular, for each $x \in \aank$ of type~2 or~3, the fiber $\pi^{-1}_{\wka/k}(x)$ is finite.
\end{lemma}
\begin{proof}
Let us first assume that~$X$ is a closed disc: there exists $\alpha\in \wka$ and $r\in \R_{>0}$ such that $X = D^+(\alpha,r)$. Since~$k^a$ is dense in~$\wka$, we may assume $\alpha \in k^a$. Then orbit of~$\alpha$ is then finite, hence so is the orbit of~$X$. The case of an open disc is dealt with similarly.

Points of type~2 or~3 are boundary points of closed discs by Lemma~\ref{lem:discs}, hence the orbits of such points are finite too. Since closed and open annuli are determined by their boundary points, which are of type~2 or~3 (see Lemma~\ref{lem:discs} again), their orbits are finite too.

Finally, if $x \in \aank$ is a point of type~2 (resp.~3), then, by Corollary~\ref{cor:piwkak}, the fiber $\pi^{-1}_{\wka/k}(x)$ is an orbit of a point of type~2 (resp.~3). The last part of the result follows.
\end{proof}

%
%
%
%
%

\section{Connectedness}\label{sec:connected}

In this section, we study the connectedness properties of the Berkovich affine line and its subsets. Our main source here is \cite[Section~4.2]{Berkovich90} for the connectedness properties (see also \cite[Section~3.2]{Berkovich90} for higher-dimensional cases) and \cite[Section~1.9]{DucrosRSS} for properties of quotients of graphs.

\begin{proposition}\label{prop:discsconnected}
Open and closed discs, annuli and Swiss cheeses are path-connected. The Berkovich affine line~$\aank$ is path-connected and locally path-connected.
\end{proposition}
\begin{proof}
Let us deal first handle the case of a closed disc, say $D^+(\alpha,r)$ with $\alpha\in k$ and $r\in \R_{>0}$. By Proposition~\ref{prop:piKk}, the projection map~$\pi_{K/k}$ is continuous and surjective for any complete valued extension~$K$ of~$k$. As a result, it is enough to prove the result on some extension of~$k$, hence we may assume that $k$~is algebraically closed and maximally complete.

Let $x\in D^+(\alpha,r)$. By Lemma~\ref{lem:type23algclosed} and Remark~\ref{rem:maxsph}, there exist $\beta\in k$ and $s\in \R_{\ge 0}$ such that $x = \eta_{\beta,s}$. Since $x\in D^+(\alpha,r)$, we have 
\[|T-\alpha|_{\beta,s} = \max(|\alpha-\beta|,s) \le r,\]
hence $s\le r$ and $\eta_{\alpha,r} = \eta_{\beta,r}$. As a consequence, the map 
\[\lambda\in [0,1] \mapsto \eta_{\alpha,s + (r-s)\lambda} \in \aank \]
defines a continuous path from~$x$ to~$\eta_{\alpha,r}$ in $D^+(\alpha,r)$ (see Lemma~\ref{lem:path}). It follows that the disc $D^+(\alpha,r)$ is path-connected.

The same argument may be used in order to prove that closed annuli and Swiss cheeses are connected. Indeed, if such a set~$S$ is written as the complement of some open discs in~$D^+(\alpha,r)$, it is easy to check that, for each $x\in S$, the path joining~$x$ to~$\eta_{\alpha,r}$ that we have just described actually remains in~$S$. 

Since the open figures may be written as increasing unions of the closed ones, the result holds for them too.

\medbreak

The last statement follows from the fact that $\aank$ may be written as an increasing union of discs for the global part, and from Proposition~\ref{prop:types} and Corollary~\ref{cor:piwkak} for the local part.
\end{proof}

The Berkovich affine line actually satisfies a stronger connectedness property: it is uniquely path-connected. We will now prove this result, starting with the case of an algebraically closed field.


\begin{definition}
We define a partial ordering~$\le$ on~$\AA^{1,\an}_{\wka}$ by setting
\[x \le y \text{ if } \forall P\in \wka[T],\ |P|_{x} \le |P|_{y}.\]

For each $x\in \AA^{1,\an}_{\wka}$, we set
\[I_{x} := \{y \in \AA^{1,\an}_{\wka} : y\ge x\}.\]
\end{definition}

\begin{remark}\label{rem:sigmaIx}
Let~$\sigma \in \Gal(k^s/k)$. By definition of the action of~$\sigma$, for every $x,y \in \aanwka$, we have $x\le y$ if, and only if, $\sigma(x) \le \sigma(y)$. It follows that, for each $z\in \aanwka$, we have $\sigma(I_{z}) = I_{\sigma(z)}$.
\end{remark}

The proof of the following lemma is left as an exercise for the reader.
\begin{lemma}\label{lem:orderetaar}
Let $\alpha \in \wka$ and $r\in\R_{> 0}$. 
For each $x\in \aanwka$, we have $x\le \eta_{\alpha,r}$ if, and only if, $x \in D^+(\alpha,r)$.

In particular, for $\beta\in \wka$ and $s\in \R_{\ge 0}$, we have $\eta_{\alpha,r} \le \eta_{\beta,s}$ if, and only if, $\max(|\alpha-\beta|,r) \le s$. Moreover, when those conditions hold, we have $\eta_{\beta,s} = \eta_{\alpha,s}$.
\qed
\end{lemma}

\begin{corollary}\label{cor:ordermin}
The minimal points for the ordering~$\le$ on~$\aanwka$ are exactly the points of type~1 and~4.
\end{corollary}
\begin{proof}
It follows from Lemma~\ref{lem:orderetaar} that points of type~2 and~3 are not minimal and that points of type~1 are. To prove the result, it remains to show that points of type~4 are minimal too. This assertion follows from the fact that any such point is the unique point contained in the intersection of a family of closed discs (see Remark~\ref{rem:intersectiontype4}) by applying Lemma~\ref{lem:orderetaar} again.
\end{proof}


\begin{corollary}\label{cor:Ixexplicit}
Let $\alpha\in \wka$ and $r\in \R_{\ge 0}$. We have
\[I_{\eta_{\alpha,r}} = \{\eta_{\alpha,s} : s\ge r\}.\]

Let~$x \in \aanwka$ be a point of type~4 and fix notation as in Proposition~\ref{prop:family}. Then, for each $i,j \in I$, we have $I_{\eta_{\alpha_{i},r_{i}}} \subseteq I_{\eta_{\alpha_{j},r_{j}}}$ or $I_{\eta_{\alpha_{j},r_{j}}} \subseteq I_{\eta_{\alpha_{i},r_{i}}}$ and we have
\[I_{x} = \{x\} \cup \bigcup_{i\in I} I_{\eta_{\alpha_{i},r_{i}}}.\]
\qed
\end{corollary}
%

The former result shows in particular that our notation is consistent with that of Lemma~\ref{lem:path}.

\begin{corollary}\label{cor:Ix}
Let $x\in \aanwka$. The radius map $r \colon \aanwka \to \R_{\ge 0}$ induces a homeomorphism from~$I_{x}$ onto $[r(x),+\infty)$.

The restriction of the projection map~$\pi_{\wka/k}$ to~$I_{x}$ is injective.
\end{corollary}
\begin{proof}
The fact that the radius map induces is a bijection from~$I_{x}$ onto its image follows from Corollary~\ref{cor:Ixexplicit}. One can then prove directly that its inverse is continuous and open by arguing as in the proof of Lemma~\ref{lem:path}.

Let us now prove the second part of the statement. Let $y, z \in I_{x}$ such that $x < y < z$. Then, there exist $\alpha\in \wka$ and $r< s \in \R_{>0}$ such that $y =\eta_{\alpha,r}$ and $z=\eta_{\alpha,s}$. Since $k^s$ in dense in~$\wka$, we may assume that $\alpha \in k^s$. Let $P_{\alpha} \in k[T]$ be the minimal polynomial of~$\alpha$ over~$k$. Since there exists $Q_{\alpha} \in \wka[T]$ such that $P_{\alpha} = (T-\alpha) \, Q_{\alpha}$, we have
\[|P_{\alpha}|_{y} = r \,|Q_{\alpha}|_{y} < s \, |Q_{\alpha}|_{y} = |P_{\alpha}|_{z}.\]
It follows that $\pi_{\wka/k}(y) \ne \pi_{\wka/k}(z)$.

Finally, assume that there exists $z\in I_{x} - \{x\}$ such that $\pi_{\wka/k}(z) = \pi_{\wka/k}(x)$. For each $y\in I_{x}$ such that $x < y < z$ and each $P \in k[T]$, we have
\[|P|_{x} \le |P|_{y} \le |P|_{z} = |P|_{x},\]
hence $|P|_{y} = |P|_{z}$, which contradicts what we have just proved. 
\end{proof}

\begin{corollary}\label{cor:ppcm}
Let $x,y \in \aanwka$. The set $\{z\in \aanwka : z \ge x \text{ and } z \ge y\}$ admits a smallest element. We will denote it by $x \vee y$.

In particular, if~$x$ and~$y$ are comparable for~$\le$, then $I_{x} \cup I_{y}$ is homeomorphic to a half-open interval and, otherwise, $I_{x} \cup I_{y}$ is homeomorphic to a tripod with one end-point removed, \emph{i.e.} the union of two closed intervals and one half-open interval glued along a common end-point.
\end{corollary}
\begin{proof}
We have $\{z\in \aanwka : z \ge x \text{ and } z \ge y\} = I_{x} \cap I_{y}$. By Corollary~\ref{cor:Ix}, $I_{x}$ and~$I_{y}$ may be sent to  intervals of the form $[\ast,+\infty)$ by order-preserving homeomorphisms. We deduce that, in order to prove the result, it is enough to prove that $I_{x} \cap I_{y}$ is non-empty. Setting $R := \max(|T|_{x},|T|_{y})$, the points $x$ and~$y$ belong to~$D^+(0,R)$, hence $\eta_{R}$ belongs to~$I_{x,y}$, by Lemma~\ref{lem:orderetaar}. 

Denoting by~$x \vee y$ the smallest element of~$I_{x}\cap I_{y}$, we have $I_{x} \cap I_{y} = I_{x\vee y}$. In particular, by Corollary~\ref{cor:Ix}, $I_{x} \cap I_{y}$ is homeomorphic to a half-open interval with end-point~$x\vee y$. Set
\[[x,x\vee y] := \{z\in I_{x} : x \le z\le (x\vee y)\} \text{ and } [y,x\vee y] := \{z\in I_{y} : y \le z\le (x\vee y)\}.\] 
By Corollary~\ref{cor:Ix} again, $[x,x\vee y]$ and $[y,x\vee y]$ are homeomorphic to closed intervals (possibly singletons if~$x$ and~$y$ are comparable). Finally, the result follows by writing 
\[I_{x} \cup I_{y} = [x,x\vee y] \cup [y,x\vee y] \cup I_{x\vee y}.\]
\end{proof}

\begin{corollary}\label{cor:Ixclosed}
Let $x\in \aanwka$. The set $\aanwka -I_{x}$ is a union of open discs. In particular, $I_{x}$~is closed.
\end{corollary}
\begin{proof}
Let $y \in \aanwka -I_{x}$. The point $x \vee y$ defined in Corollary~\ref{cor:ppcm} belongs to~$I_y$ and is not equal to~$y$. It follows that there exists $z\in I_{y} - I_{x}$ such that $y < z$. By Corollary~\ref{cor:Ixexplicit}, $y$ is a point of type~2 or~3, hence, by Lemma~\ref{lem:orderetaar}, the set $D := \{u \in \aanwka : u\le z\}$ is a closed disc of positive radius. It is contained in $\aanwka -I_{x}$. Since~$z$ is not the boundary point of~$D$, $y$ indeed belongs to some open subdisc of~$D$ by Lemmas~\ref{lem:closeddisctype3} and~\ref{lem:closeddisctype2}.
\end{proof}

\begin{proposition}\label{prop:aan-S}
Let~$\Gamma$ be a subset of~$\aanwka$ such that, for each $x\in \Gamma$, $I_{x} \subseteq \Gamma$. Then, $\aanwka - \Gamma$ is a union of discs and points of type~4.

If, moreover, $\Gamma$ is closed, then $\aanwka - \Gamma$ is a union of open discs.
\end{proposition}
\begin{proof}
To prove the first statement, it is enough to show that each point of $\aanwka - \Gamma$ that is not of type~4 is contained in a closed disc. Let~$x$ be such a point. Then, by Lemma~\ref{lem:orderetaar}, $\{y \in \aanwka : y\le x\}$ is a closed disc. By assumption, it is contained in $\aanwka - \Gamma$, and the result follows.

Let us now assume that~$\Gamma$ is closed. Let $x\in \aanwka - \Gamma$. Since~$\Gamma$ is closed, there exists $y\in I_{x} - \Gamma$ such that $y>x$. We then conclude as in the proof of Corollary~\ref{cor:Ixclosed}.
\end{proof}

\begin{corollary}\label{cor:pathwka}
For every $x, y \in \aanwka$, there exists a unique injective path from~$x$ to~$y$.
\end{corollary}
\begin{proof}
Set $I_{x,y} := I_{x} \cup I_{y}$. It follows from Corollaries~\ref{cor:Ix} and~\ref{cor:ppcm} that it is homeomorphic a half-open interval or a tripod with one end-point removed. It follows that there exists a unique injective path from~$x$ to~$y$ inside~$I_{x,y}$. In particular, there exists an injective path from~$x$ to~$y$ in~$\aanwka$.

By Corollary~\ref{cor:Ixclosed} and Proposition~\ref{prop:aan-S}, $\aanwka - I_{x,y}$ is a union of open discs. By Lemma~\ref{lem:discs}, every open disc has a unique boundary point. As a consequence, an injective path going from~$x$ to~$y$ cannot meet any of these open disc, since otherwise it would contain its boundary point twice. It follows that such a path is contained in~$I_{x,y}$.
 \end{proof}

%
%
%
%
%


We want to deduce the result for~$\aank$ by using the projection map $\pi_{\wka/k} \colon \aanwka \to \aank$. Recall that, by Corollary~\ref{cor:piwkak}, it is a quotient map by the group $\Gal(k^s/k)$.

%

\begin{proposition}\label{prop:pathk}
For every $x, y \in \aank$, there exists a unique injective path from~$x$ to~$y$.
\end{proposition}
\begin{proof}


To ease notation, we will write~$\pi$ instead of~$\pi_{\wka/k}$. 

Let $x,y \in \aank$. Choose $x' \in \pi^{-1}(x)$ and $y' \in \pi^{-1}(y)$. Set $J := I_{x'} \cup I_{y'}$. By Lemma~\ref{cor:Ix}, $\pi(I_{x'})$ and $\pi(I_{y'})$ are half-open intervals and, by Corollary~\ref{cor:ppcm}, they meet.

Let $z \in \pi(I_{x'}) \cap \pi(I_{y'})$. There exists $z' \in I_{x'}$ such that $\pi(z')=z$ and $\sigma\in \Gal(k^s/k)$ such that $\sigma(z') \in I_{y'}$. By Remark~\ref{rem:sigmaIx}, we have $\sigma(I_{z'}) = I_{\sigma(z')}$ and we deduce that $\pi(J)$ is a tripod with one end-point removed. In particular, there exists a unique injective path from~$x$ to~$y$ in $\pi(J)$, and at least one such path in~$\aank$.

%

\medbreak

To conclude, let us prove that any injective path from~$x$ to~$y$ in~$\aank$ is contained in $\pi(J)$. Set 
\[\Gamma :=\pi^{-1}(\pi(J)) = \bigcup_{\sigma \in \Gal(k^s/k)} \sigma(J).\]

Since the action of the elements of~$\Gal(k^s/k)$ preserves the norm, for each $R \in \R_{>0}$, we have 
\[\Gamma \cap D^+(0,R) = \pi^{-1}(\pi(J \cap D^+(0,R))),\]
and it follows that $\Gamma \cap D^+(0,R)$ is compact. We deduce that~$\Gamma$ is closed.

Moreover, we have 
\[\Gamma = \bigcup_{z \in \pi^{-1}(x) \cup \pi^{-1}(y)} I_{z},\]
hence, by Proposition~\ref{prop:aan-S}, $\aanwka - \Gamma$ is a union of open discs. Using the fact that it is invariant under the action of $\Gal(k^s/k)$, we may now conclude as in the proof of Corollary~\ref{cor:pathwka}.
\end{proof}

\begin{notation}
For $x, y \in \aank$, we will denote by~$[x,y]$ the unique injective path between~$x$ and~$y$. We set $(x,y) := [x,y] - \{x,y\}$.
\end{notation}

\begin{remark}
The fact that $\aank$ is uniquely path-connected means that it has the structure of a real tree. The type of the points may be easily read off this structure. Indeed the end-points are the points of type~1 and~4 (see Remark~\ref{rem:intersectiontype4} for points of type~4). Among the others, type~2 points are branch-points (with infinitely many edges meeting there, see Lemma~\ref{lem:closeddisctype2}) while type~3 points are not (see Lemma~\ref{lem:closeddisctype3}).
For a graphical representation of this fact, see Figure \ref{fig:DroiteTypePoints}.
\end{remark}

\section{Virtual discs and annuli}\label{sect:retractions}

In this section, we introduce generalizations of discs and annuli that are more suitable when working over arbitrary fields and study them from the topological point of view. 
We explain that they retract onto some simple subsets of the real line, namely singletons and intervals respectively. Here we borrow from \cite[Section~6.1]{Berkovich90}, which also contains a treatment of more general spaces.


\subsection{Definitions}


\begin{definition}
A connected subset~$U$ of~$\aank$ is called a \emph{virtual open (resp. closed) disc} if $\pi_{\wka/k}^{-1}(U)$ is a disjoint union of open (resp. closed) discs. We define similarly virtual open, closed and flat annuli and virtual open and closed Swiss cheeses. \end{definition}
We now introduce particularly interesting subsets of virtual annuli.

\begin{definition}\label{def:skeletonannulus}
Let~$A$ be a virtual open or closed annulus. The \emph{skeleton} of~$A$ is the complement of all the virtual open discs contained in~$A$. We denote it by~$\Sigma_{A}$.
\end{definition}

\begin{example}\label{ex:skeleton}
Consider the open annulus $A := A^-(\gamma,\rho_{1},\rho_{2})$, with $\gamma\in k$ and $\rho_{1} < \rho_{2}\in \R_{> 0}$. Its skeleton is
\[\Sigma_{A} =  \{\eta_{\gamma,s} : \rho_{1} < s < \rho_{2}\}.\]

\begin{figure}[ht]
\includegraphics[scale=.3]{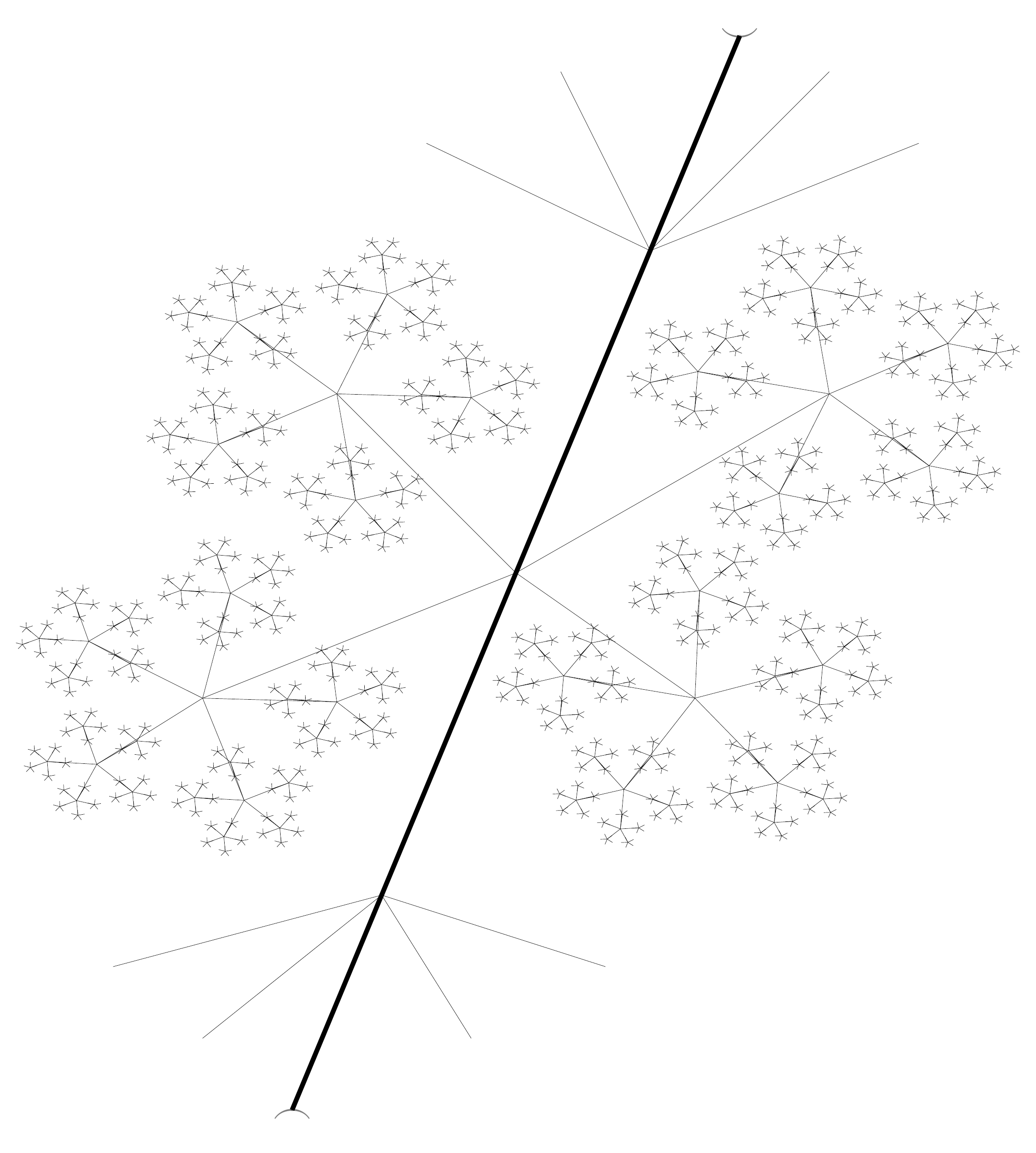}
\caption{The skeleton of an open annulus is the open line segment joining its boundary points}
\end{figure}

%
\end{example}

\begin{lemma}\label{lem:skeletonvirtual}
Let~$A$ be a virtual open (resp. closed) annulus. 
Let $C$ be a connected component of $\pi^{-1}_{\wka/k}(A)$. Then, $C$ is an open (resp. closed) annulus and $\pi_{\wka/k}$ induces a homeomorphism between~$\Sigma_{C}$ and~$\Sigma_{A}$. In particular, $\Sigma_{A}$ is an open (resp. closed) interval.
\end{lemma}
\begin{proof}
Let us handle the case of a virtual open annulus, the other one being dealt with similarly. The connected component~$C$ is an open annulus, by the very definition of virtual open annulus. It also follows from the definitions that we have $\Sigma_{A} = \pi_{\wka/k}(\Sigma_{C})$. 

Denote by~$G_{C}$ the subgroup of~$\Gal(k^s/k)$ consisting of those elements that preserve~$C$. By Corollary~\ref{cor:piwkak}, $\pi_{\wka/k}$ induces a homeomorphism $C/G_{C} \simeq A$, hence a homeomorphism $\Sigma_C/G_{C} \simeq \Sigma_A$.

Write $C = A^-(\gamma,\rho_{1},\rho_{2})$, with $\gamma\in \wka$ and $\rho_{1} < \rho_{2}\in \R_{> 0}$. Its complement in~$\aanwka$ has two connected components, namely $D^+(\gamma,\rho_{1})$ and 
\[D^+_{\infty}(\gamma,\rho_{2}) := \{ x\in \aanwka : |(T-\gamma)(x)| \ge \rho_{2}\}.\]
For $r$ big enough, we have $\eta_{\gamma,r} = \eta_{r}$, which belongs to $D^+_{\infty}(\gamma,\rho_{2})$ and is stable under~$\Gal(k^s/k)$. It follows that $D^+(\gamma,\rho_{1})$ and $D^+_{\infty}(\gamma,\rho_{2})$ are stable under~$G_{C}$. 

Let $\sigma \in G_{C}$. We have 
\[ \sigma(D^+(\gamma,\rho_{1})) = D^+(\sigma(\gamma),\rho_{1}) = D^+(\gamma,\rho_{1}),\]
hence $|\sigma(\gamma)-\gamma| \le \rho_{1}$. It follows that, for each $s\in (\rho_{1},\rho_{2})$, we have 
\[ \sigma(\eta_{\gamma,s}) = \eta_{\sigma(\gamma),s} = \eta_{\gamma,s},\]
that is to say $\sigma$ acts as the identity on~$\Sigma_{C}$. The result follows.
\end{proof}

\begin{remark}\label{rem:virtual}
Virtual discs and annuli are usually defined as arbitrary connected $k$-analytic curves (see Section~\ref{sec:smoothcurves}) whose base change to~$\wka$ is a disjoint union of discs or annuli, without requiring an embedding into~$\aank$. Our definition is \textit{a priori} more restrictive.

With this definition of virtual annulus, an additional difficulty appears. If $A$ is such a virtual annulus and $C$ is a connected component of $\pi^{-1}_{\wka/k}(A)$, then there may exist elements of~$\Gal(k^s/k)$ that preserve~$C$ but swap its two ends. In this case, the skeleton~$\Sigma_{A}$ is a half-open interval. (For an example of such a behaviour, consider a Galois orbit in~$\wka$ consisting of two points, its image~$x$ in~$\aank$ and let~$A$ be the complement in~$\pank$ of a small virtual closed disc containing~$x$.) Some authors (for instance A.~Ducros) explicitly rule out this possibility in the definition of virtual annulus.
\end{remark}

We may also extend the notion of modulus (see Definition~\ref{def:modulus}) from annuli to virtual annuli. Recall that, by Lemma~\ref{lem:finitefiber}, the set of connected components of the preimage of a virtual open annulus by~$\pi^{-1}_{\wka/k}$ is finite and that, by Corollary~\ref{cor:piwkak}, $\Gal(k^s/k)$ acts transitively on it.


\begin{definition}
Let $A$ be a virtual open annulus. Denote by~$\calC$ the set of connected components of $\pi^{-1}_{\wka/k}(A)$ and let~$C_{0}$ be one of them. We define the \emph{modulus} of~$A$ to be 
\[ \Mod(A) := \Mod(C_{0})^{1/\sharp \calC} \in (1,+\infty).\]
It is independent of the choice of~$C_{0}$.
\end{definition}

\begin{remark}
Beware that other normalizations exist in the literature for the modulus of a virtual open annulus. For instance, A.~Ducros sets $\Mod(A) := \Mod(C_{0})$ (see \cite[3.6.15.11]{DucrosRSS}).
\end{remark}
 
We refer the reader to \cite[Section~3.6]{DucrosRSS} for a thorough treatment of classical and virtual discs and annuli.

\subsection{The case of an algebraically closed and maximally complete base field}

In this section, we assume that $k$~is algebraically closed and maximally complete. This will allow us to prove our results through direct computations.

Recall that, by Lemma~\ref{lem:type23algclosed} and Remark~\ref{rem:maxsph}, each point of~$\aank$ (or a disc or an annulus) is of the form~$\eta_{\alpha,r}$, for $\alpha\in k$ and $r\in\R_{\ge0}$. 

\medbreak

Let $\gamma\in k$ and $\rho \in \R_{> 0}$. We consider the closed disc $D^+(\gamma,\rho)$.

\begin{lemma}\label{lem:retractiondisc}
The map
\[[0,1] \times D^+(\gamma,\rho) \ni (t,\eta_{\alpha,r}) \longmapsto \eta_{\alpha,\max(r,t\rho)} \in D^+(\gamma,\rho)\]
is well-defined and continuous. It induces a deformation retraction of~$D^+(\gamma,\rho)$ and $D^-(\gamma,\rho) \cup\{\eta_{\gamma,\rho}\}$ onto their unique boundary point~$\eta_{\gamma,\rho}$.
\qed
\end{lemma}

Let $\gamma\in k$ and $\rho_{1} < \rho_{2}\in \R_{> 0}$. We consider the open annulus $A := A^-(\gamma,\rho_{1},\rho_{2})$. Each rational point of~$A$ is contained in an open disc. Indeed, let $\alpha \in k \cap A$ (which implies that $\rho_{1} < |\gamma-\alpha| < \rho_{2}$). Then, the open disc $D^-(\alpha, |\gamma-\alpha|)$ is the maximal open disc containing~$\alpha$ that lies in~$A$. This open disc is relatively compact in~$A$ and its unique boundary point is $\eta_{\alpha, |\gamma-\alpha|} = \eta_{\gamma,|\gamma-\alpha|}$. 


We have 
\[ A -  \Sigma_{A} = \bigcup_{\alpha \in k \cap A} D^-(\alpha, |\gamma-\alpha|).\]
In particular, each connected component of $A-\Sigma_{A}$ is an open disc. For each such open disc~$D$, we denote by~$\eta_{D}$ its boundary point in~$A$. 

\begin{lemma}\label{lem:retractionannuli}
The map 
\[ [0,1] \times A \ni (t,\eta_{\alpha,r}) \text{ with } r\le |\gamma-\alpha| \longmapsto \eta_{\alpha,\max(r,t|\gamma-\alpha|)} \in A\]
is well-defined and continuous. It induces a deformation retraction of~$A$ onto~$\Sigma_{A}$.

Its restriction to each connected component~$D$ of~$A-\Sigma_{A}$ coincides with the map from Lemma~\ref{lem:retractiondisc} and induces a deformation retraction of~$D$ onto~$\eta_{D}$. 

For each $\eta\in \Sigma_{A}$, the set of points that are sent to~$\eta$ by the retraction map is the union of~$\eta$ and all the connected component~$D$ of~$A-\Sigma_{A}$ such that $\eta_{D} =\eta$. It is a flat closed annulus.
\qed
\end{lemma}

\subsection{The general case}

We now remove the assumption on~$k$. Let~$K$ be an extension of~$k$ that is algebraically closed and maximally complete. If~$X$ is a virtual disc or a virtual annulus, then each connected component of $\pi_{K/k}^{-1}(X)$ is a true disc or annulus over~$K$ and the results of the previous section apply. By continuity and surjectivity of~$\pi_{K/k}$, we deduce retraction results in this setting too.

\begin{proposition}\label{prop:retractionvirtualdisc}
Let~$D$ be a virtual open disc in~$\aank$. Then, $D$ has a unique boundary point~$\eta_{D}$ in~$\aank$ and there exists a canonical deformation retraction $\tau_{D} \colon D \cup\{\eta_{D}\} \to \{\eta_{D}\}$.
\qed
\end{proposition}

\begin{proposition}\label{prop:retractionvirtualannulus}
Let~$A$ be a virtual open annulus. 
Each connected component of $A-\Sigma_{A}$ is a virtual open disc. 

There exists a canonical deformation retraction $\tau_{A} \colon A \to \Sigma_{A}$. Its restriction to any connected component~$D$ of~$A-\Sigma_{A}$ induces the map~$\tau_{D}$ from Proposition~\ref{prop:retractionvirtualdisc}.

Moreover, for each open interval (resp. closed interval, resp. singleton) $I$ in $\Sigma_{A}$, the set $\tau_{A}^{-1}(I)$ is a virtual open annulus (resp. a virtual closed annulus, resp. a virtual flat closed annulus).

\qed
\end{proposition}

\begin{lemma}\label{lem:ccannuli}
Let $A$ be a virtual open annulus. Let~$F$ be a finite subset of~$\Sigma_{A}$ and denote by~$\calI$ the set of connected components of $\Sigma_{A} - F$. The elements of~$\calI$ are open intervals and we have
\[ \Mod(A) = \prod_{I \in \calI} \Mod(\tau_{A}^{-1}(I)).\]
\qed
\end{lemma}

\section{Lengths of intervals}\label{sec:lengths}

In this section, we show that intervals inside~$\aank$ may be endowed with a canonical (exponential) length. To start with, we define the useful notion of branch at a point.

\begin{notation}
For $\alpha \in \wka$ and $r\in \R_{>0}$, we set 
\[ D^-_{\infty}(\alpha,r) := \{x\in \aank : |(T-\alpha)(x)| > r\}.\]
\end{notation}

\begin{lemma}\label{lem:cccomplement}
Let $x \in \aanwka$.

If $x$ is of type~1 or~4, then $\aanwka - \{x\}$ is connected. 

If $x =\eta_{\alpha,r}$, with $\alpha\in \wka$ and $r \in \R_{>0} - |k^\times|$, then the connected components of $\aanwka - \{x\}$ are $D^-(\alpha,r)$ and $D^-_{\infty}(\alpha,r)$.

If $x =\eta_{\alpha,r}$, with $\alpha\in \wka$ and $r \in |k^\times|$, then the connected components of $\aanwka - \{x\}$ are $D^-_{\infty}(\alpha,r)$ and the discs of the form $D^-(\beta,r)$ with $|\beta-\alpha| \le r$. 
\end{lemma}
\begin{proof}
Assume that~$x$ is of type~1 or~4. Let $y,y' \in \aanwka - \{x\}$. By Corollary~\ref{cor:ppcm}, $I_{y} \cup I_{y'}$ is a connected subset of~$\aanwka$ containing~$y$ and~$y'$. By Corollary~\ref{cor:ordermin}, it does not contain~$x$. It follows that $\aanwka - \{x\}$ is connected.

Assume that $x =\eta_{\alpha,r}$, with $\alpha\in \wka$ and $r \in \R_{>0} - |k^\times|$. By Proposition~\ref{prop:discsconnected}, $D^-(\alpha,r)$ is connected. Since $D^-_{\infty}(\alpha,r)$ may be written as an increasing union of open annuli, we deduce from Proposition~\ref{prop:discsconnected}) that it is connected too. The result now follows from Lemma~\ref{lem:closeddisctype3}.

Assume that~$x$ is of type~2. Up to a change of variables, we may assume that $x = \eta_{0,1}$. As before, we deduce from Proposition~\ref{prop:discsconnected} that $D^-_{\infty}(0,1)$ and the discs of the form $D^-(\beta,1)$ with $|\beta| \le 1$ are connected. The result now follows from Lemma~\ref{lem:closeddisctype2}.
\end{proof}

\begin{remark}
We have $D^-(\beta,r) = D^-(\beta',r)$ if, and only if, $|\beta-\beta'| < r$.
\end{remark}

\begin{definition}\label{def:branchatx}
Let $x \in \aank$. Let $y,z \in \aank - \{x\}$. We say that the intervals $(x,y]$ and $(x,z]$ are $x$-equivalent if $(x,y] \cap (x,z] \ne \emptyset$.

An $x$-equivalence class of intervals $(x,y]$, with $x \ne y$, is called a \emph{branch at~$x$}. We denote by~$\calB_{x}$ the set of branches at~$x$.
\end{definition}

\begin{remark}\label{rem:equivalentintervals}
If $(x,y]$ and $(x,z]$ are equivalent, if follows from the unique path-connectedness of~$\aank$ (see Proposition~\ref{prop:pathk}) that there exists $t \in \aank - \{x\}$ such that $(x,t] = (x,y] \cap (x,z]$. 
\end{remark}

\begin{lemma}\label{lem:BC}
Let $x \in \aank$. Denote by~$\calC_{x}$ the set of connected component of~$\aank - \{x\}$. 

For each $y \in \aank - \{x\}$, denote by~$C_{x}(y)$ the connected component of~$\aank - \{x\}$ containing~$y$. The map
\[\begin{array}{cccc} 
C \colon & \calB_{x} & \too & \calC_{x}\\
& (x,y] & \mapstoo & C_{x}(y)  
\end{array}\]
is well-defined and bijective.
\end{lemma}
\begin{proof}
Let $y \in \aank - \{x\}$. For each $t \in (x,y]$, the interval $[t,y]$ is connected and does not contain~$x$, hence $C_{x}(t) = C_{x}(y)$. It then follows from Remark~\ref{rem:equivalentintervals} that $C((x,y))$ only depends on the equivalence class of~$(x,y)$. In other words, $C$~is well-defined.

Let $y,z \in \aank - \{x\}$ such that $(x,y]$ is not equivalent to~$(x,z]$. It follows that $[y,x] \cup [x,z]$ is an injective path from~$y$ to~$z$. Since, by Proposition~\ref{prop:pathk}, $\aank$ is uniquely path-connected, the unique injective path $[y,z]$ from~$y$ to~$z$ contains~$x$. It follows that~$y$ and~$z$ belong to different connected components of $\aank - \{x\}$. This proves the injectivity of~$C$.

Finally, the surjectivity of~$C$ is obvious.
\end{proof}

\begin{lemma}\label{lem:branchskeleton}
Let $x \in \aank$ be a point of type~2 or~3. Let $y\in \aank - \{x\}$. Then, there exists~$z$ in $(x,y)$ such that, for each $t \in (x,z)$, the interval $(x,t)$ is the skeleton of a virtual open annulus.
\end{lemma}
\begin{proof}
Let $x' \in \pi^{-1}_{\wka/k}(x)$. Let $y' \in \pi^{-1}_{\wka/k}(y)$. The image of the path~$[x',y']$ by~$\pi_{\wka/k}$ is a path between~$x$ and~$y$. It follows that, up to changing~$x'$ and~$y'$, we may assume that $\pi_{\wka/k}$ restricts to a bijection between~$[x',y']$ and~$[x,y]$. 

By Lemma~\ref{lem:BC} and the explicit description of the connected components of~$\aanwka - \{x'\}$ from Lemma~\ref{lem:cccomplement}, there exists~$z'$ in $(x,y')$ such that, for each $t' \in (x,z')$, the interval $(x,t')$ is the skeleton of a open annulus. 

By Lemma~\ref{lem:finitefiber}, the orbit of any open annulus under the action of~$\Gal(k^s/k)$ is finite. It follows that, up to choosing~$z'$ closer to~$x$, we may assume that, for each $t' \in (x,z')$, the interval $(x,t')$ is the skeleton of a open annulus, all of whose conjugates either coincide with it or are disjoint from it. The image of such an open annulus by~$\pi_{\wka/k}$ is a virtual open annulus, and the result follows.
\end{proof}

As a consequence, we obtain the following result, which is the key-point to define lengths of intervals.

\begin{lemma}
Let $x,y \in \aank$ be points of type~2 or~3. Then, there exists a finite subset~$F$ of~$(x,y)$ such that each connected component of $(x,y) - F$ is the skeleton of a virtual open annulus.
\qed
\end{lemma}

\begin{definition}
Let $x,y \in \aank$ be points of type~2 or~3. Let~$F$ be a finite subset of~$(x,y)$ such that each connected component of $(x,y) - F$ is the skeleton of a virtual open annulus. Let~$\calI$ be the set of connected components of $(x,y) - F$ and, for each $J \in \calJ$, denote by~$A_{J}$ the virtual open annulus with skeleton~$J$.

We define the \emph{(exponential) length} of~$(x,y)$ to be
\[ \ell((x,y)) := \prod_{J \in \calJ} \Mod(A_{J}) \in [1,+\infty).\]
It is independent of the choices, by Lemma~\ref{lem:ccannuli}.
\end{definition}

\begin{definition}
Let~$I$ be an interval inside~$\aank$ that is not a singleton. We define the \emph{(exponential) length} of $I$ to be 
\[\ell(I) := \sup(\{ \ell((x,y)) : x,y \in I \textrm{ of type } 2 \textrm{ or } 3\})  \in [1,+\infty].\]  
\end{definition}

\begin{example}\label{ex:leta}
Let $\alpha \in k$ and $r\in \R_{>0}$ with $r \le |\alpha|$. Then, we have 
\[ \ell([\eta_{1},\eta_{\alpha,r}]) = \ell((\eta_{1},\eta_{\alpha,r})) =
\begin{cases}
\frac 1 r & \textrm{if } |\alpha| \le 1;\\[3pt]
\frac {|\alpha|} r  &\textrm{if } |\alpha| \ge 1.
\end{cases}\]
In particular, we always have $\ell([\eta_{1},\eta_{\alpha,r}]) \ge 1/r$.
\end{example}

\begin{lemma}\label{lem:length}
Let~$I$ be an interval in~$\aank$. We have $\ell(I) = +\infty$ if, and only if, the closure of~$I$ contains a point of type~1.

Let $I_{1},I_{2}$ be intervals in~$\aank$ such that $I = I_{1} \cup I_{2}$ and $I_{1}\cap I_{2}$ is either empty or a singleton. Then, we have 
\[\ell(I) = \ell(I_{1}) \, \ell(I_{2}).\]
 
\qed
\end{lemma}

\section{Variation of rational functions}\label{sec:variation}

In this section, for every rational function $F \in k(T)$, we study the variation of~$|F|$ on~$\aank$. We will explain that it is controlled by a finite subtree of~$\aank$ and investigate metric properties.

\begin{notation}
Let $x\in \aank$. We set 
\[ I_{x} := \pi_{\wka/k}(I_{x'}),\]
for $x' \in \pi^{-1}_{\wka/k}(x)$. 

By Remark~\ref{rem:sigmaIx}, this does not depend on the choice of~$x'$.
\end{notation}

As in the case of an algebraically closed base field, $I_{x}$ may be thought of as a path from~$x$ to~$\infty$. 

\begin{proposition}
Let $F \in k(T) -\{0\}$. Let~$Z$ be the set rigid points of~$\aank$ that are zeros or poles of~$F$. Set 
\[ I_{Z} := \bigcup_{z\in Z} I_{z}.\]
Then $|F|$ is locally constant on $\aank - I_{Z}$.
\end{proposition}
\begin{proof}
One immediately reduces to the case where the base field is~$\wka$. Since~$\wka$ is algebraically closed, $F$ may be written as a quotient of products of linear polynomials. It follows that is is enough to prove the results for linear polynomials.

Let $\alpha \in \wka$ and let us prove the result for $F = T-\alpha$. Let~$C$ be a connected component of $\aank - I_{\alpha}$. Let~$\eta$ in the closure of~$C$. It belongs to~$I_{\alpha}$, hence is of the form~$\eta_{\alpha,r}$ for $r\in \R_{\ge 0}$.

Recall that discs are connected, by Proposition~\ref{prop:discsconnected}. By Lemma~\ref{lem:closeddisctype3}, the case $r\notin |k^\times|$ leads to a contradiction. It follows that $r\in |k^\times|$. Performing an appropriate change of variables and using Lemma~\ref{lem:closeddisctype2}, we deduce that there exists $\beta \in \wka$ with $|\alpha-\beta|=r$ such that $C = D^-(\beta,r)$. For each $x\in C$, we have
\[|(T-\alpha)(x)| = |(T-\beta)(x) + (\beta-\alpha)| = r,\]
because of the non-Archimedean triangle inequality. The results follows.
\end{proof}

Our next step is to describe the behaviour of~$|F|$ on~$I_{Z}$ using metric data. 
Recall that, for each $x\in \aank$, we denote by~$\calB_{x}$ the set of branches at~$x$ (see Definition~\ref{def:branchatx}).

\begin{definition}
Let $x \in \aank$ and $b\in \calB_{x}$. Let~$E$ be a set. A \emph{map $f \colon b \to E$} is the data of a non-empty subset~$\calI$ of representatives of~$b$ and a family of maps $(f_{I} \colon I \to E)_{I \in \calI}$ such that, for each $I,J \in \calI$, $f_{I}$ and~$f_{J}$ coincide on~$I\cap J$.

Let~$I$ be a representative of~$b$. We say that \emph{$f$ is defined on~$I$} if $I$ belongs to~$\calI$. In this case, we usually write $f \colon I \to E$ instead of $f_{I} \colon I \to E$.
\end{definition}

Note that a map $f\colon I \to E$ defined on some representative~$I$ of~$b$ naturally gives rise to a map $f\colon b \to E$.

%
%

\begin{definition}
Let $x \in \aank$ and $b\in \calB_{x}$. Let $f \colon b \to \R_{\ge 0}$. 
Let $N \in \Z$.

We say that $f$ is \emph{monomial along~$b$ of exponent~$N$}  if there exists a representative~$(x,y]$ of~$b$ such that $f$~is defined on~$(x,y]$ and
\[\forall z \in (x,y], \forall t \in (x,z], \ f(z) = f(t) \, \ell([t,z])^N.\]
We then set 
\[\mu_{b}(f) := N.\] 
We say that $f$ is \emph{constant along~$b$} if it is monomial along~$b$ of exponent~0.
\end{definition}

\begin{remark}\label{rem:loglinear}
Written additively, the last condition becomes
\[\forall z \in (x,y], \forall t \in (x,z], \ \log(f(z)) = \log(f(t)) + N \log(\ell([t,z])).\]
This explains why, in the literature, such maps are often referred to as \emph{log-linear} and $N$ denoted by~$\partial_{b} \log(f)$.
\end{remark}

Let $x \in \aank$ and $x' \in \pi^{-1}_{\wka/k}(x)$. Let $b'\in \calB_{x'}$. It follows from Lemmas~\ref{lem:branchskeleton} and~\ref{lem:skeletonvirtual} that, for each small enough representative $(x',y']$ of~$b'$, $\pi_{\wka/k}$ induces a homeomorphism from $(x',y']$ to $(x,\pi_{\wka/k}(y')]$. This property allows to define the image of the branch~$b'$.

\begin{definition}
Let $x \in \aank$ and $x' \in \pi^{-1}_{\wka/k}(x)$. Let $b'\in \calB_{x'}$. The \emph{image of the branch~$b'$ by~$\pi_{\wka/k}$} is the branch
\[ \pi_{\wka/k}(b') := (x,\pi_{\wka/k}(y')] \in \calB_{x},\]
for a small enough representative $(x',y']$ of~$b'$.
\end{definition}

\begin{lemma}
Let $x \in \aank$ and $x' \in \pi^{-1}_{\wka/k}(x)$. Let $b\in \calB_{x}$. For each $b\in \calB_{x}$, there exists $b' \in \calB_{x'}$ such that $\pi_{\wka/k}(b') = b$. The set of such~$b'$'s is finite and $\Gal(k^s/k)$ acts transitively on it.
\end{lemma}
\begin{proof}
The existence of~$b'$ is proved as in the beginning of the proof of Lemma~\ref{lem:branchskeleton}. The rest of the statement follows from Lemmas~\ref{lem:branchskeleton} and~\ref{lem:finitefiber} and Corollary~\ref{cor:piwkak}. 
\end{proof}

The following result is a direct consequence of the definitions.

\begin{lemma}\label{lem:basechangeexponent}
Let $x \in \aank$ and $b\in \calB_{x}$. Let $f \colon b \to \R_{\ge 0}$. Assume that there exists $N\in \Z$ such that, for each $b' \in \pi^{-1}_{\wka/k}$, $f \circ \pi_{\wka/k}$ is monomial along~$b'$ of exponent~$N$. Then, $f$ is monomial along~$b$ of exponent $N \cdot \sharp \pi^{-1}_{\wka/k}(b)$.
\end{lemma}

\begin{definition}
Let $F \in \wka(T) - \{0\}$. Let~$\alpha \in \wka$. The \emph{order of~$F$ at~$\alpha$} is the unique integer~$v$ such there exists $P \in \wka[T]$ with $P(\alpha) \ne 0$ satisfying
\[ F(T) = (T-\alpha)^v \, P(T).\]
We denote it by~$\ord_{\alpha}(P)$. 
\end{definition}

\begin{theorem}\label{thm:sigmabF}
Let $F \in \wka(T) - \{0\}$. Let $x \in \aanwka$ and $b\in \calB_{x}$. Then the map~$|F|$ is monomial along~$b$. 

If $x$ is of type~1, then $\mu_{b}(|F|) = \ord_{x}(F)$.

If $x$ is of type~2 or 3 and $C(b)$ is bounded, then
\[ \mu_{b}(|F]) = - \sum_{z \in \wka \cap C(b)} \ord_{z}(F).\]

If $x$ is of type~2 or 3 and $C(b)$ is unbounded, then
\[ \mu_{b}(|F]) = \deg(F) - \sum_{z \in \wka \cap C(b)} \ord_{z}(F).\]

If $x$ is of type~4, then $\mu_{b}(|F|) = 0$.
\end{theorem}
\begin{proof}
Let us first remark that if the result holds for~$G$ and~$H$ in~$\wka(T) - \{0\}$, then it also holds for~$GH$ and~$G/H$. As a result, since~$\wka$ is algebraically closed, it is enough to prove the result  for a linear polynomial, so we may assume that $F = T-a$, with $a\in k$.

\medbreak

$\bullet$ Assume that~$x$ is of type~1.

There exists $\alpha \in k$ such that $x=\alpha$. By Lemmas~\ref{lem:cccomplement} and~\ref{lem:BC}, there is a unique branch at~$x$. It is represented by
\[ (\alpha,\eta_{\alpha,s}] = \{ \eta_{\alpha,t} : t \in (0,s]\},\]
for any $s\in \R_{>0}$.

If $\alpha = a$, then for each $t\in \R_{>0}$, we have $|(T-a)(\eta_{a,t})| = t$, hence $|T-a|$ is monomial along~$b$ of exponent~1. Since we have $\ord_{a}(T-a) = 1$, the result holds in this case.

If $\alpha \ne a$, then for each $t\in (0,|a-\alpha|)$, we have $|(T-a)(\eta_{\alpha,t})| = |a_{\alpha}|$, hence $|T-a|$ is monomial along~$b$ of exponent~0. Since we have $\ord_{\alpha}(T-a) = 0$, the result holds in this case too.

\medbreak

$\bullet$ Assume that~$x$ is of type~2 or~3 and that $C(b)$ is bounded.

There exist $\alpha \in k$ and $r\in \R_{>0}$ such that $x=\eta_{\alpha,r}$. By Lemma~\ref{lem:cccomplement}, there exists $\beta \in k$ with $|\beta-\alpha|\le r$ such that $C(b) = D^-(\beta,r)$. Since $\eta_{\alpha,r} = \eta_{\beta,r}$, we may assume that~$\alpha=\beta$. The branch~$b$ is then represented by 
\[(\eta_{\alpha,r},\eta_{\alpha,s}] = \{ \eta_{\alpha,t} : t \in [s,r)\},\]
for any $s\in (0,r]$.

If $a \in C(b)$, then we have $|a-\alpha| < r$, hence, for each $t \in [|a-\alpha|,r)$, we have $|(T-a)(\eta_{\alpha,t})| = t$, hence $|T-a|$ is monomial along~$b$ of exponent~$-1$. It follows that the result holds in this case. 

If $a \notin C(b)$, then we have $|a-\alpha| = r$, hence, for each $t \in [0,r)$, we have $|(T-a)(\eta_{\alpha,t})| = |a-\alpha|$, hence $|T-a|$ is monomial along~$b$ of exponent~$0$. It follows that the result holds in this case too. 

\medbreak

$\bullet$ Assume that~$x$ is of type~2 or~3 and that $C(b)$ is unbounded.

There exist $\alpha \in k$ and $r\in \R_{>0}$ such that $x=\eta_{\alpha,r}$. By Lemma~\ref{lem:cccomplement}, the branch~$b$ is then represented by 
\[(\eta_{\alpha,r},\eta_{\alpha,s}] = \{ \eta_{\alpha,t} : t \in (r,s]\},\]
for any $s\in (r,+\infty)$.

If $a \in C(b)$, then we have $|a-\alpha| > r$, hence, for each $t \in (r,|a-\alpha|)$, we have $|(T-a)(\eta_{\alpha,t})| = |a-\alpha|$, hence $|T-a|$ is monomial along~$b$ of exponent~$0$. We have
\[  \deg(T-a) - \sum_{z \in \wka \cap C(b)} \ord_{z}(T-a) = \deg(T-a) - \ord_{a}(T-a) = 1-1 = 0,\] 
hence the result holds in this case. 

If $a \notin C(b)$, then we have $|a-\alpha| \le r$, hence, for each $t \in (r,+\infty)$, we have $|(T-a)(\eta_{\alpha,t})| = t$, hence $|T-a|$ is monomial along~$b$ of exponent~$1$. We have
\[  \deg(T-a) - \sum_{z \in \wka \cap C(b)} \ord_{z}(T-a) = \deg(T-a) = 1,\] 
hence the result holds in this case too. 

\medbreak

$\bullet$ Assume that~$x$ is of type~4.

By Proposition~\ref{prop:types}, $x$ admits a basis of neighborhood made of discs. It follows that there exist $\alpha \in k$ and $r\in \R_{>0}$ such that $x \in D^-(\alpha,r)$ and $a\notin D^-(\alpha,r)$. For each $y \in D^-(\alpha,r)$, we have $|(T-a)(y)| = |(T-\alpha)(y) + (\alpha-a)| = |a-\alpha|$, hence $|T-a|$ is constant in the neighborhood of~$x$. The result follows.
\end{proof}

\begin{remark}
The term $\deg(R)$ that appears in the formula when~$C(b)$ is unbounded may be identified with the opposite of the order of~$R$ at~$\infty$. If we had worked on~$\panwka$ instead of~$\aanwka$, it would not have been necessary to discuss this case separately.
\end{remark}

\begin{corollary}\label{cor:Kirchhoff}
Let $F \in k(T) - \{0\}$. Let $x \in \aank$ be a point of type~2 or~3. 
Then, there exists a finite subset~$B_{x,F}$ of $\calB_{x}$ such that, for each $b \in \calB_{x} \setminus B_{x,F}$, $|F|$ is constant along~$b$ and we have 
\[ \sum_{b\in \calB_{x}} \mu_{b}(|F|) =  \sum_{b\in B_{x,F}} \mu_{b}(|F|) = 0.\]
\end{corollary}
\begin{proof}
Using Lemma~\ref{lem:basechangeexponent}, one reduces to the case where the base field is~$\wka$. The result then follows from Theorem~\ref{thm:sigmabF}, since we have $\ord_{z}(F) = 0$ for almost all $z\in \wka$ and
\[ \sum_{z \in \wka} \ord_{z}(F) = \deg(F).\] 
\end{proof}

\begin{remark}
The statement of Corollary~\ref{cor:Kirchhoff} corresponds to a harmonicity property. This is more visible written in the additive form (see Remark~\ref{rem:loglinear}):
\[ \sum_{b\in \calB_{x}} \partial_{b} \log(|F|) = 0.\]
A full-fledged potential theory actually exists over Berkovich analytic curves. We refer to A.~Thuillier's thesis~\cite{ThuillierPhD} for the details (see also~\cite{BakerRumely} for the more explicit case of the Berkovich line over an algebraically closed field).
\end{remark}

Since analytic functions are, by definition, locally uniform limits of rational functions, the results on variations of functions extend readily. 

\begin{theorem}\label{thm:harmonicityanalytic}
Let $x \in \aank$ be a point of type~2 or~3 and let $F \in \calO_{x} - \{0\}$. Then, for $b \in \calB_{x}$, $|F|$ is monomial along~$b$ with integer slope. Moreover, there exists a finite subset~$B_{x,F}$ of $\calB_{x}$ such that, for each $b \in \calB_{x} \setminus B_{x,F}$, $|F|$ is contant along~$b$ and we have
\[ \sum_{b\in \calB_{x}} \mu_{b}(|F|) = 0.\]
\qed
\end{theorem}

\begin{corollary}
Let $x \in \aank$ be a point of type~2 or~3 and let $F \in \calO_{x} - \{0\}$. 
If~$|F|$ has a local maximum at~$x$, then it is locally constant at~$x$.
\qed
\end{corollary}

\begin{corollary}\label{cor:maximumprinciple}
Let $U$ be a connected open subset of~$\aank$ and let $F \in \calO(U)$. If~$|F|$ is not constant on~$U$, then there exists $y \in \partial U$ and $b\in \calB_{y}$ such that 
\[\lim_{z \xrightarrow[b]{} y} |F(z)| = \sup_{t\in U}(|F(t)|),\]
where the limit is taken on points~$z$ converging to~$y$ along~$b$, and $|F|$ has a negative exponent along~$b$.
\qed
\end{corollary}

We conclude with a result of a different nature, showing that, if~$\varphi$ is a finite morphism of curves, the relationship between the length of an interval at the source and the length of its image is controlled by the degree of the morphism. We state a simplified version of the result and refer to \cite[Proposition~3.6.40]{DucrosRSS} for a more general statement.

\begin{theorem}\label{th:degreelength}
Let~$A_{1}$ and~$A_{2}$ be two virtual annuli over~$k$ with skeleta~$\Sigma_{1}$ and~$\Sigma_{2}$. Let $\varphi \colon A_{1} \to A_{2}$ be a finite morphism such that $\varphi(\Sigma_{1}) = \Sigma_{2}$. Then, for each $x,y \in \Sigma_{1}$, we have
\[ \ell(\varphi([x,y])) = \ell([x,y])^{\deg(\varphi)}.\]
\qed
\end{theorem}

\begin{example}
Let $n\in \N_{\ge 1}$ and consider the morphism $\varphi \colon \aank \to \aank$ given by $T \mapsto T^n$. For each $r \in \R_{>0}$, we have $\varphi(\eta_{r}) = \eta_{r^n}$. In particular, for $r < s \in \R_{>0}$, we have 
\[ \ell(\varphi([\eta_{r},\eta_{s}])) = \ell([\eta_{r^n},\eta_{s^n}]) = \frac{s^n}{r^n} =  \ell([\eta_{r},\eta_{s}])^n.\] 
\end{example}

\vskip1cm

\begin{center}
\textbf{\Large{Part II: Berkovich curves and Schottky uniformization}}
\end{center}
\addcontentsline{toc}{chapter}{\textbf{II. Berkovich curves and Schottky uniformization}}
\refstepcounter{part}\label{part:curves}

\medbreak


\section{The Berkovich projective line and M\"obius transformations}\label{sec:BerkprojMobius}

\subsection{Affine Berkovich spaces}\label{sec:affineBerkovichspaces}

We generalize the constructions of Part~\ref{part:A1k}, replacing $k[T]$ by an arbitrary $k$-algebra of finite type. Our reference here is \cite[Section~1.5]{Berkovich90}.

\begin{definition}
Let $A$ be $k$-algebra of finite type. The \emph{Berkovich spectrum} $\Spec^\an(A)$ of~$A$ is the set of multiplicative seminorms on~$A$ that induce the given absolute value~$\va$ on~$k$.
\end{definition}

As in Definition~\ref{def:Hx}, we can associate a completed residue field~$\scrH(x)$ to each point~$x$ of~$\Spec^\an(A)$. As in Section~\ref{sec:topology}, we endow $\Spec^\an(A)$ with the coarsest topology that makes continuous the maps of the form
\[x \in \Spec^\an(A) \mapstoo |f(x)| \in \R\]
for $f \in A$. Properties similar to that of the Berkovich affine line still hold in this setting: the space $\Spec^\an(A)$ is countable at infinity, locally compact and locally path-connected.

We could also define a sheaf of function on $\Spec^\an(A)$ as in Definition~\ref{def:analyticfunction}\footnote{Note however that the ring of global sections is always reduced, so that we only get the right notion when~$A$ is reduced. The proper construction involves defining first the space $\AA^{n,\an}_{k} := \Spec^\an(k[T_{1},\dotsc,T_{n}])$, then open subsets of it, and then closed analytic subsets of the latter, as we usually proceed for analytifications in the complex setting.}
 with properties similar to that of the usual complex analytic spaces.
 
\begin{lemma}\label{lem:fonctorialitySpec}
Each morphism of $k$-algebras $\varphi \colon A \to B$ induces a continuous map of Berkovich spectra
\[\begin{array}{cccc}
\Spec^\an(\varphi) \colon & \Spec^\an(B) & \too & \Spec^\an(A)\\
& \va_{x} & \mapstoo &  |\varphi(\wc)|_{x} 
\end{array}.
\]
\end{lemma}

Let us do the example of a localisation morphism.

\begin{notation}
Let~$A$ be a $k$-algebra of finite type and let $f\in A$. We set
\[D(f) := \{x \in \Spec^\an(A) \mid f(x) \ne 0\}.\]
It is an open subset of $\Spec^\an(A)$.
\end{notation}

\begin{lemma}\label{lem:localisation}
Let~$A$ be a $k$-algebra of finite type and let $f\in A$. The map $\Spec^\an(A[1/f]) \to \Spec^\an(A)$ induced by the localisation morphism $A \to A[1/f]$  induces a homeomorphism onto its image $D(f)$.
\qed
\end{lemma}

\subsection{The Berkovich projective line}\label{sec:defP1}

In this section, we explain how to construct the Berkovich projective line over~$k$. It can be done, as usual, by gluing upside-down two copies of the affine line $\aank$ along $\aank-\{0\}$. We refer to \cite[Section~2.2]{BakerRumely}
for a definition in one step reminiscent of the ``$\textrm{Proj}$'' construction from algebraic geometry.

To carry out the construction of the Berkovich projective line more precisely, let us introduce some notation. We consider, as before, the Berkovich affine line $X := \aank$ with coordinate~$T$, \emph{i.e.} $\Spec^\an(k[T])$. By Lemma~\ref{lem:localisation}, its subset $U := \aank - \{0\} = D(T)$ may be identified with $\Spec^\an(k[T,1/T])$. 

We also consider another Berkovich affine line~$X'$ with coordinate~$T'$ and identify its subset $U' := X'-\{0\}$ with $\Spec^\an(k[T',1/T'])$. 

By Lemma~\ref{lem:fonctorialitySpec}, the isomorphism $k[T',1/T'] \xrightarrow[]{\sim} k[T,1/T]$ sending~$T'$ to~$1/T$ induces an isomorphism $\iota \colon U \xrightarrow[]{\sim} U'$.

\begin{definition}
The Berkovich projective line~$\pank$ is the space obtained by gluing the Berkovich affine lines~$X$ and~$X'$ along their open subsets~$U$ and~$U'$ via the isomorphim~$\iota$.

We denote by~$\infty$ the image in~$\pank$ of the point~0 in~$X'$.
\end{definition}

The basic topological properties of~$\pank$ follow from that of~$\aank$.

\begin{proposition}
We have $\pank = \aank \cup\{\infty\}$.

The space~$\pank$ is Hausdorff, compact, uniquely path-connected and locally path-connected.
\qed
\end{proposition}

For $x,y \in \pank$, we denote by $[x,y]$ the unique injectif path between~$x$ and~$y$.

\subsection{M\"obius transformations}


Let us recall that, in the complex setting, the group $\PGL_{2}(\C)$ acts on~$\PP^1(\C)$ \emph{via} M\"obius transformations. More precisely, to an invertible matrix $A = \begin{pmatrix} a & b\\ c & d \end{pmatrix}$, one associates the automorphism
\[\gamma_{A} \colon z \in \PP^1(\C) \mapstoo \frac{az+b}{cz+d} \in \PP^1(\C)\]
with the usual convention that, if $c\ne 0$, then $\gamma_{A}(\infty) = a/c$ and $\gamma_{A}(-d/c) = \infty$, and, if $c=0$, then $\gamma_{A}(\infty) = \infty$.

\medbreak

We would like to define an action of $\PGL_{2}(k)$ on~$\pank$ similar to the complex one.  Let $A := \begin{pmatrix} a & b\\ c & d \end{pmatrix} \in \GL_{2}(k)$.

First note that we can use the same formula as above to associate to~$A$ an automorphism~$\gamma_{A}$ of the set of rational points $\pank(k)$. 

It is actually possible to deal with all the points this way. Indeed, let $x\in \pank - \pank(k)$. In Section~\ref{sec:classification}, we have associated to~$x$ a character $\chi_{x} \colon k[T] \to \scrH(x)$. Since~$x$ is not a rational point, $\chi_{x}(T)$ does not belong to~$k$, hence the quotient $(a \chi_{x}(T) + b)/(c \chi_{x}(T) +d)$
makes sense. 
We can then define $\gamma_{A}(x)$ as the element of~$\aank$ associated to the character 
\[P(T) \in k[T] \mapsto P\Big(\frac{a \chi_{x}(T) + b}{c \chi_{x}(T) +d}\Big) \in \scrH(x).\] 

\medbreak

This construction can also be made in a more algebraic way. By Lemmas~\ref{lem:fonctorialitySpec} and~\ref{lem:localisation}, the morphism of $k$-algebras
\[P(T) \in k[T] \mapsto P\Big(\dfrac{aT+b}{cT+d}\Big) \in k\Big[T,\dfrac1{cT+d}\Big]\]
induces a map $\gamma_{A,1} \colon \aank - \{-\frac d c\} \to \aank \subseteq \pank$ (with the convention that $-d/c = \infty$ if $d=0$).

Similarly, the morphism of $k$-algebras
\[Q(U) \in k[T'] \mapsto Q\Big(\dfrac{c+dT'}{a+bT'}\Big) \in k\Big[T',\dfrac1{a+bT'}\Big]\]
induces a map $\gamma_{A,2} \colon \pank - \{0,-\frac b a\} \to \pank$ (with the convention that $-b/a = \infty$ if $a=0$).

A simple computation shows that the maps $\gamma_{A,1}$ and~$\gamma_{A,2}$ are compatible with the isomorphism~$\iota$ from Section~\ref{sec:defP1}. Note that we always have $- \frac d c \ne - \frac b a$. If $ad\ne 0$, it follows that we have $\big(\aank - \{-\frac c d\} \big) \cup \big(\pank - \{0,-\frac b a\} \big) = \pank$, so the two maps glue to give a global map 
\[\gamma_{A} \colon \pank\to\pank.\]

We let the reader handle the remaining cases by using appropriate changes of variables.

\begin{notation}
For $a,b,c,d \in k$ with $ad-bc\ne 0$, we denote by $\begin{bmatrix} a & b\\ c & d \end{bmatrix}$ the image in $\PGL_{2}(k)$ of the matrix  $\begin{pmatrix} a & b\\ c & d \end{pmatrix}$.

From now on, we will identify each element $A$ of~$\PGL_{2}(k)$ with the associated automorphism~$\gamma_{A}$ of~$\pank$.
\end{notation}

\begin{lemma}\label{lem:disctodisc}
The image of a closed (resp. open) disc of~$\pank$ by a M\"obius transformation is a closed (resp. open) disc.
\end{lemma}
\begin{proof}
Let $A \in \GL_{2}(k)$. We may extend the scalars, hence assume that $k$ is algebraically closed. In this case, $A$~is similar to an upper triangular matrix. In other words, up to changing coordinates of~$\pank$, we may assume that $A$~is upper triangular. The transformation~$\gamma_{A}$ is then of the form
\[\gamma_{A} \colon z \in \pank \mapsto \alpha z \in \pank\]
or 
\[\gamma_{A} \colon z \in \pank \mapsto  z + \alpha\in \pank\]
for some $\alpha\in k$. In both cases, the result is clear.
\end{proof}

\subsection{Loxodromic transformations and Koebe coordinates}\label{sec:Koebe}

\begin{definition}
A matrix in $\GL_{2}(k)$ is said to be \emph{loxodromic} if its eigenvalues in~$k^a$ have distinct absolute values.

A M\"obius transformation is said to be \emph{loxodromic} if some (or equivalently every) representative is.
\end{definition}

\begin{lemma}\label{lem:abcdloxodromic}
Let $a,b,c,d \in k$ with $ad-bc\ne 0$ and set $A := \begin{pmatrix} a & b\\ c & d \end{pmatrix} \in \GL_{2}(k)$. Then $A$ is loxodromic if, and only if, we have $|ad-bc| < |a+d|^2$.
\end{lemma}
\begin{proof}
Let~$\lambda$ and~$\lambda'$ be the eigenvalues of~$A$ in~$k^a$. We may assume that $|\lambda| \le |\lambda'|$. 

If we have $|\lambda| = |\lambda'|$, then we have 
\[|a+d|^2 = |\lambda+\lambda'|^2 \le |\lambda'|^2 = |\lambda| \, |\lambda'| = |ad-bc|.\]
Conversely, if we have $|\lambda| < |\lambda'|$, then we have 
\[|a+d|^2 = |\lambda+\lambda'|^2 = |\lambda'|^2 > |\lambda| \, |\lambda'| = |ad-bc|.\]
\end{proof}

%
%
%
%

Let $A\in \PGL_{2}(k)$ be a loxodromic M\"obius transformation. 

Fix some representative~$B$ of~$A$ in~$\GL_{2}(k)$. Denote by~$\lambda$ and~$\lambda'$ its eigenvalues in~$k^a$. We may assume that $|\lambda| < |\lambda'|$. The characteristic polynomial~$\chi_{B}$ of~$B$ cannot be irreducible over~$k$, since otherwise its roots in~$k^a$ would have the same absolute values. It follows that~$\lambda$ and~$\lambda'$ belong to~$k$. Set $\beta := \lambda/\lambda' \in k^{\circ\circ}$.

The eigenspace of~$B$ associated to the eigenvalue~$\lambda$ (resp.~$\lambda'$) is a line in~$k^2$. Denote by~$\alpha$ (resp.~$\alpha'$) the corresponding element in~$\PP^1(k)$.

\begin{definition}
The elements $\alpha, \alpha' \in \PP^1(k)$ and $\beta \in k^{\circ\circ}$ depend only on~$A$ and not on the chosen representative. They are called the \emph{Koebe coordinates} of~$A$. 
\end{definition}

There exists a M\"obius transformation $\varepsilon \in \mathrm{PGL}_2(k)$ such that $\varepsilon(0)=\alpha$ and $\varepsilon(\infty)=\alpha'$. The M\"obius transformation $\varepsilon^{-1}A \varepsilon$ now has eigenspaces corresponding to 0 and~$\infty$ in~$\PP^1(k)$ and the associated automorphism of~$\pank$ is
\[\gamma_{\varepsilon^{-1}A \varepsilon} \colon z\in \pank \mapsto \beta z \in \pank.\]
We deduce that~0 and~$\infty$ are respectively the attracting and repelling fixed points of~$\gamma_{\varepsilon^{-1}A \varepsilon}$ in~$\pank$. It follows that~$\alpha$ and~$\alpha'$ are respectively the attracting and repelling fixed points of~$\gamma_{A}$ in~$\pank$.

%
%

\medbreak

It follows from the same argument that the Koebe coordinates determine uniquely the M\"obius transformation $A$. In fact, given $\alpha,\alpha', \beta \in k$ with $\alpha\ne \alpha'$ and $0<|\beta|<1$, the M\"obius transformation that has these elements as Koebe coordinates is given explicitly by 
\begin{equation}\label{eq:LoxodromicMatrix}
M(\alpha, \alpha', \beta) =  \left[
\begin{matrix} 
\alpha-\beta \alpha' & (\beta-1)\alpha\alpha' \\
1-\beta & \beta\alpha-\alpha'
\end{matrix}\right],
 \end{equation}
 
In an analogous way, whenever $\infty \in \pank$ is an attracting or repelling point of a loxodromic M\"obius transformation, we can recover the latter as: 
\begin{equation}\label{eq:LoxodromicMatrix2}
M(\alpha, \infty, \beta) =  \left[
\begin{matrix} 
\beta & (1-\beta)\alpha \\
0 & 1
\end{matrix}\right] \; \mbox{ or } \;
M(\infty, \alpha', \beta) = \left[
\begin{matrix} 
1 & (\beta-1)\alpha' \\
0 & \beta
\end{matrix}\right].
\end{equation}

\begin{remark}\label{rem:nonloxodromic}
Let $A\in \PGL_{2}(k)$ be a M\"obius transformation that is not loxodromic. Then, extending the scalars to~$\wka$ and possibly changing the coordinates, the associated automorphism of~$\panwka$ is a homothety
\[z\in \panwka \mapsto \beta z \in \panwka \textrm{ with } |\beta|=1\] 
or a translation
\[z\in \panwka \mapsto z + b \in \panwka.\]
Note that those automorphisms have several fixed points in $\panwka$ ($\eta_{r}$ with $r\ge 0$ in the first case and $r\ge |b|$ in the second). It follows that~$A$ itself also has infinitely many fixed points in~$\pank$.

\end{remark}

%
%
%

%

\section{Berkovich $k$-analytic curves}\label{sec:kanal}

\subsection{Berkovich $\AA^1$-like curves}\label{sec:A1like}

In this section we go one step further the study of affine and projective lines, by introducing a class of curves that ``locally look like the affine line'', and see that there are interesting examples of curves belonging to this class. 


A much more general theory of $k$-analytic curves exists but it will be discussed only briefly in this text in Section \ref{sec:smoothcurves}, in the case of smooth curves.
For more on this topic, the standard reference is \cite[Chapter~4]{Berkovich90}. The most comprehensive account to-date can be found in A.~Ducros' book project \cite{DucrosRSS}, while deeper discussions of specific aspects are contained in the references in the Appendix \ref{app:Berkurves} of the present text.


\begin{definition}
A \emph{$k$-analytic $\AA^1$-like curve} is a locally ringed space in which every point admits an open neighborhood isomorphic to an open subset of~$\aank$.
\end{definition}

It follows from the explicit description of bases of neighborhoods of points of~$\aank$ (see Proposition~\ref{prop:types}) that each $k$-analytic $\AA^1$-like curve admits a covering by virtual open Swiss cheeses.
By local compactness, such a covering can always be found locally finite. 
It can be refined into a partition (no longer locally finite) consisting of simpler pieces.

\begin{theorem}\label{thm:triangulation}
Let~$X$ be a $k$-analytic $\AA^1$-like curve. Then, there exist
\begin{enumerate}
\item a locally finite set~$S$ of type~2 points of~$X$;
\item a locally finite set~$\calA$ of virtual open annuli of~$X$;
\item a set~$\calD$ of virtual open discs of~$X$
\end{enumerate}
such that $S \cup \calA \cup \calD$ is a partition of~$X$.
\end{theorem}
\begin{proof}
Each virtual open Swiss cheese may be written as a union of finitely many points of type~2, finitely many virtual open annuli and some virtual open discs (as in Example~\ref{ex:triangulationSwisscheese} below). By a combinatorial argument that is not difficult but quite lengthy, the covering so obtained can be turned into a partition.
\end{proof}

\begin{definition}\label{def:skeleton}
Let~$X$ be a $k$-analytic $\AA^1$-like curve. A partition $\calT = (S,\calA,\calD)$ of~$X$ satisfying the properties (i), (ii), (iii) of Theorem~\ref{thm:triangulation} is called a \emph{triangulation} of~$X$. The locally finite graph naturally arising from the set
\[\Sigma_{\calT} := S \cup \bigcup_{A\in \calA} \Sigma_{A}\] 
is called the \emph{skeleton} of~$\calT$. It is such that $X - \Sigma_{\calT}$ is a disjoint union of virtual open discs.

A triangulation $\calT$ is said to be \emph{finite} if the associated set $S$ is finite.
If this is the case, then $\Sigma_{\calT}$ is a finite graph.
By the results of Section~\ref{sec:lengths}, for each triangulation~$\calT$, $\Sigma_{\calT}$ may be naturally endowed with a metric structure. 
\end{definition}

%
%
%

\begin{remark}
It is more usual to define a triangulation as the datum of the set~$S$ only. 
Note that~$S$ determines uniquely~$\calA$ and~$\calD$ since their elements are exactly the connected components of~$X-S$, so our change of convention is harmless. 
\end{remark}

\begin{example}\label{ex:triangulationSwisscheese}
Consider the curve 
\[X := D^-(0,1) - (D^+(a,r) \cup D^+(b,r))\] 
for $r \in (0,1)$ and $a,b\in k$ with $|a|,|b| <1$, $|a-b|>r$. Set 
\[S := \{\eta_{a,|a-b|}\},\] 
\[\calA := \{A^-(a,|a-b|,1), A^-(a,r,|a-b|), A^-(b,r,|a-b|)\}\]
and 
\[\calD := \{D^-(u,|a-b|) : u \in k, |u-a|=|u-b| = |a-b|\}.\] 
Then, the triple $\calT := (S,\calA,\calD)$ is a triangulation of~$X$. The associated skeleton is a finite tree with three (half open) edges.
\end{example}

\begin{figure}[h]

\includegraphics[scale=.45]{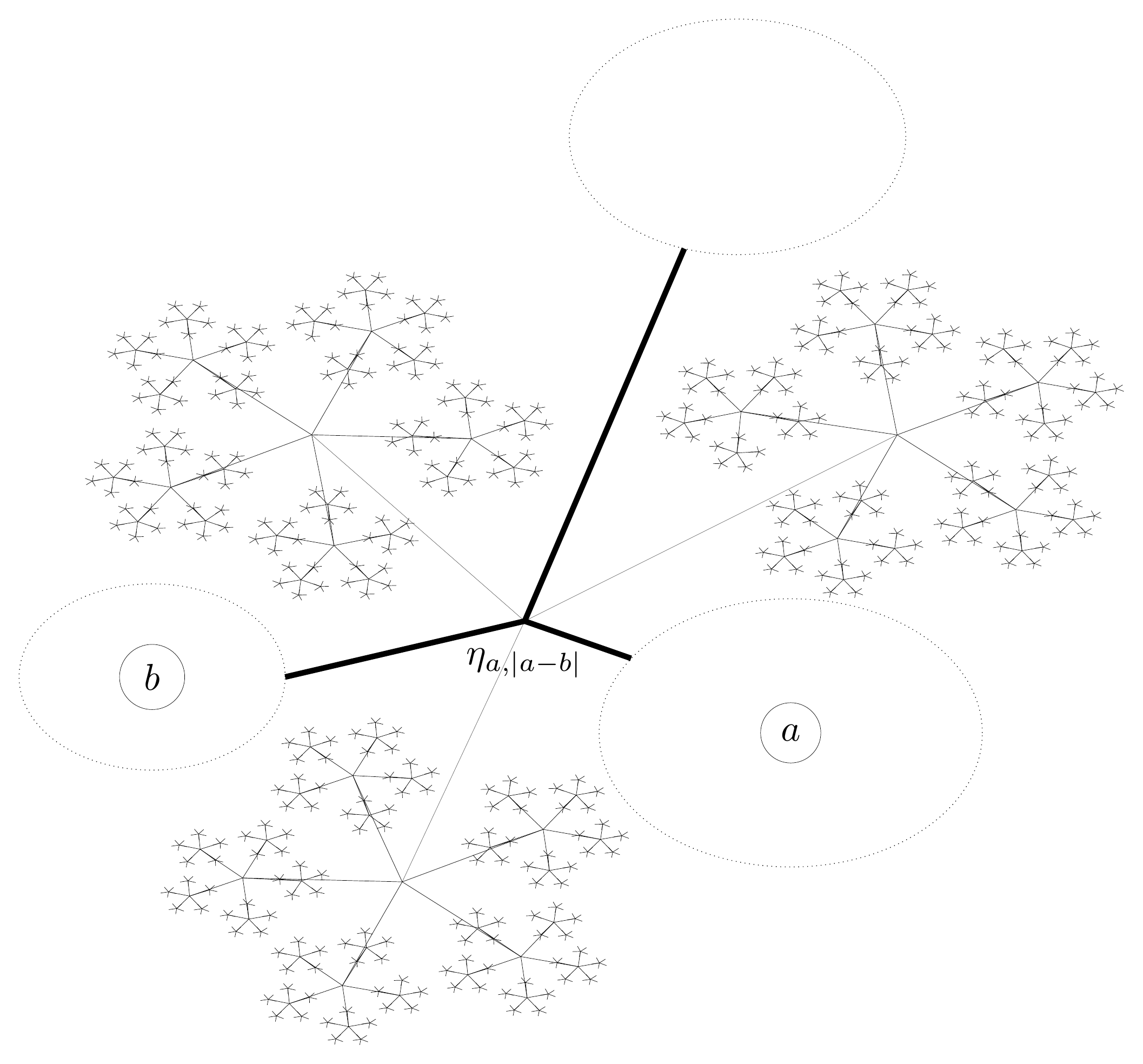}

\caption{The Swiss cheese $X$ described in Example \ref{ex:triangulationSwisscheese}. Its skeleton $\Sigma_X$ is the union of the three edges in evidence.}
\end{figure}

\begin{proposition}\label{prop:retraction}
Let $X$ be a connected $\AA^1$-like curve.
Let $\calT = (S,\calA,\calD)$ be a triangulation of~$X$ such that $S \neq \emptyset$ or $\calA \neq \emptyset$. 

There exists a canonical deformation retraction $\tau_{\calT} \colon X \to \Sigma_{\calT}$. Its restriction to any virtual open annulus $A\in \calA$ induces the map~$\tau_{A}$ from Proposition~\ref{prop:retractionvirtualannulus} and its restriction to any connected component~$D$ of~$A-\Sigma_{A}$ (which is a virtual open disc) induces the map~$\tau_{D}$ from Proposition~\ref{prop:retractionvirtualdisc}.

In particular, for each $\eta \in \Sigma_{A}$, the set $\tau_{\calT}^{-1}(\eta)$ is a virtual flat closed annulus.
\qed
\end{proposition}

\begin{definition}\label{def:skeletonX}
Let~$X$ be a $k$-analytic $\AA^1$-like curve. The \emph{skeleton} of~$X$ is the complement of all the virtual open discs contained in~$X$. We denote it by~$\Sigma_{X}$.
\end{definition}

\begin{remark}\label{rem:skeletonX}
Let~$X$ be a $k$-analytic $\AA^1$-like curve. It is not difficult to check that we have
\[\Sigma_X = \bigcap_{\calT } \Sigma_{\calT},\] 
for $\calT$ ranging over all triangulations of~$X$. In particular, $\Sigma_{X}$ is a locally finite metric graph (possibly empty).

Assume that~$X$ is connected and that $\Sigma_{X}$ is non-empty. Then there exists a triangulation~$\calT_{0}$ of~$X$ such that $\Sigma_{X} = \Sigma_{\calT_{0}}$. In particular, there is a canonical deformation retraction $\tau_{X} \colon X \to \Sigma_{X}$.
\end{remark}

\subsection{Arbitrary smooth curves}\label{sec:smoothcurves}

It goes beyond the scope of this survey to develop the full theory of Berkovich analytic curves. 
We only state in this section a few definitions and general facts, to which we would like to refer later.

\begin{definition}\label{def:smoothcurve}
A \emph{smooth $k$-analytic curve} is a locally ringed space $X$ that is locally isomorphic to an open subset of a Berkovich spectrum of the form~$\Spec^\an(A)$, where~$A$ is the ring of functions on a smooth affine algebraic curve over~$k$.
\end{definition}

For each smooth $k$-analytic curve~$X$ and each complete valued extension~$K$ of~$k$, one may define the \emph{base-change}~$X_{K}$ of~$X$ to~$K$, by replacing each $\Spec^\an(A)$ by $\Spec^\an(A\otimes_{k}K)$ in its definition. It is a smooth $K$-analytic curve and there is a canonical projection morphism $\pi_{K/k} \colon X_{K} \to X$. The analogues of Proposition~\ref{prop:piGalois} and Corollary~\ref{cor:piwkak} hold in this more general setting.

\begin{example}
For each complete valued extension~$K$ of~$k$, the base-change of~$\aank$ to~$K$ is $\AA^{1,\an}_{K}$.
\end{example}

If one starts with a smooth algebraic curve~$\calX$ over~$k$, one may cover it by curves of the form $\Spec(A)$, with~$A$ as in Definition~\ref{def:smoothcurve} above, and then glue the corresponding analytic spaces~$\Spec^\an(A)$ to get a smooth $k$-analytic curve, called the \emph{analytification} of~$\calX$, and denoted by~$\calX^\an$.

\begin{example}
The analytification of~$\AA^1_{k}$ is~$\aank$.
\end{example}

As in the complex case, smooth compact $k$-analytic curves are automatically algebraic.

\begin{theorem}\label{thm:smoothcompact}
Let~$X$ be a smooth compact $k$-analytic curve. Then, there exists a projective smooth algebraic curve over~$k$ such that $X = \calX^\an$.
\end{theorem}

The invariants we have defined so far for the Berkovich affine line~$\aank$ have natural counterparts for smooth $k$-analytic curves. Let $X$ be a smooth $k$-analytic curve. For each point~$x\in X$, the completed residue field~$\scrH(x)$ is the completion of a finitely generated extension of~$k$ of transcendence degree less than or equal to~1. We may then define integers $s(x)$ and $t(x)$ such that $s(x)+t(x) \le 1$ and the type of~$x$, as we did in the case of~$\aank$ (see Definition~\ref{def:typesofpoints}). 

If~$x$ is of type~2, then, by the equality case in Abhyankar's inequality (see Theorem~\ref{thm:Abhyankar}), the group $|\scrH(x)^\times| / |k^\times|$ is finitely generated, hence finite, and the field extension $\widetilde{\scrH(x)}/\tilde k$ is finitely generated.

\medbreak

Let us fix the definition of genus of an algebraic curve.

\begin{definition}
Let $F$ be a field and let $C$ be a projective curve over~$F$, \textit{i.e.} a connected normal projective scheme of finite type over~$F$ of dimension~1. 

If $F$ is algebraically closed, then $C$ is smooth, and we define the \emph{geometric genus of $C$} to be 
\[g(C) := \dim_{F} H^0(C,\Omega_{C}).\]

In general, let $\bar F$ be an algebraic closure of~$F$. Let~$C'$ be the normalization of a connected component of $C\times_{F} \bar F$. It is a projective curve over~$\bar F$ and we define the \emph{geometric genus of $C$} to be 
\[g(C) := g(C').\]
It does not depend on the choice of~$C'$.
\end{definition}

\begin{definition}\label{def:residuecurve}
Let~$X$ be a smooth $k$-analytic curve and let $x\in X$ be a point of type~2. 

The \emph{residue curve} at~$x$ is the unique (up to isomorphism) projective curve~$\scrC_{x}$ over $\tilde{k}$ with function field~$\widetilde{\scrH(x)}$. The \emph{genus of~$x$} is the geometric genus of~$\scrC_{x}$. We denote it by~$g(x)$.

The \emph{stable genus of~$x$}, is the genus of any point~$x'$ over~$x$ in~$X_{\wka}$. We denote it by~$g_\st(x)$. It does not depend on the choice of~$x'$.
\end{definition}

\begin{example}\label{ex:residuecurve}
Let $\alpha\in k$ and $r\in |k^\times|^\Q$. By Example~\ref{ex:typeetaalphar}, the residue curve at the point~$\eta_{\alpha,r}$ in~$\aank$ is the projective line~$\PP^1_{\tilde k}$ over~$\tilde k$. In particular, we have $g(\eta_{\alpha,r}) = 0$. 

By Lemma~\ref{lem:type23algclosed}, any point of type~2 in~$\aank$ (hence in any $k$-analytic $\AA^1$-like curve) has stable genus~0.
\end{example}

The fact that the stable genus does not need to coincide with the genus is what motivates our definition. Let us give an example of this phenomenon.

\begin{remark}
Let $p\ge 5$ be a prime number. 
Consider the affine analytic plane~$\AA^{2,\an}_{\Q_{p}}$ with coordinates $x,y$. Let~$X$ be the smooth $\Q_{p}$-analytic curve inside~$\AA^{2,\an}_{\Q_{p}}$ given by the equation $y^2 = x^3 + p$ and let $\pi \colon X \to \AA^{1,\an}_{\Q_{p}}$ be the projection onto the first coordinate~$x$. 

The fiber $\pi^{-1}(\eta_{0,|p|^{-1/3}})$ contains a unique point, that we will denote by~$a$. One may check that~$\widetilde{\scrH(a)}$ is a purely transcendental extension of~$\F_{p}$ generated by the class~$u$ of~$p x^3$ (which coincides with the class of $p y^2$):
\[ \widetilde{\scrH(a)} \simeq \F_{p}(u).\]
In particular, we have $\scrC_{a} = \PP^1_{\F_{p}}$ and $g(a)=0$.

Let us now extend the scalars to the field~$\C_{p}$, whose residue field is an algebraic closure~$\overline{\F}_{p}$ of~$\F_{p}$. Let~$b$ be the unique point of~$X_{\C_{p}}$ over~$a$. The field $\widetilde{\scrH(b)}$ is now generated by the class~$v$ of~$p^{-1/3} x$ and the class~$w$ of~$p^{-1/2} y$:
\[ \widetilde{\scrH(b)} \simeq \overline{\F}_{p}(v)[w]/(v^3 - w^2 +1).\]
In particular, $\scrC_{b}$ is an elliptic curve over~$\overline{\F_{p}}$, and we have $g_{\st}(a) = g(b) =1$.
\end{remark}

We always have an inequality between genus and stable genus.

\begin{lemma}
Let~$X$ be a smooth $k$-analytic curve and let $x\in X$ be a point of type~2. Then, we have $g(x) \le g_{\st}(x)$.
\end{lemma}
\begin{proof}
Let~$x'$ be a point of~$X_{\wka}$ over~$x$. By definition, the residue curve~$\scrC_{x}$ at~$x$ is defined over~$\tilde k$ and the residue curve~$\scrC_{x'}$ at~$x'$ is defined over an algebraic closure~$\bar{\tilde{k}}$ of~$\tilde k$.

The projection morphism $\pi_{\wka/k} \colon X_{\wka} \to X$ induces an isometric embedding $\scrH(x) \to \scrH(x')$, hence an embedding $\widetilde{\scrH(x)} \to \widetilde{\scrH(x')}$. It follows that we have a morphism $\scrC_{x'} \to \scrC_{x}$, hence a morphism $\varphi \colon \scrC_{x'} \to \scrC_{x} \times_{\tilde{k}} \bar{\tilde{k}}$. Its image is a connected component~$C$ of $\scrC_{x} \times_{\tilde{k}} \bar{\tilde{k}}$. The morphism $\varphi$ factors through~$C$, and even through the normalization~$\widetilde C$ of~$C$. By definition, we have $g(\widetilde C) = g(x)$ and $g(\scrC_{x'}) = g_{\st}(x)$. The result now follows from the Riemann-Hurwitz formula. 
\end{proof}

\begin{proposition}\label{prop:bijrescurve}
Let~$X$ be a smooth $k$-analytic curve and let $x\in X$ be a point of type~2. 
There is a natural bijection between the closed points of the residue curve~$\scrC_{x}$ at~$x$ and the set of directions emanating from~$x$ in~$X$. 
\qed
\end{proposition}

\begin{example}
Assume that $k$ is algebraically closed. For $X = \aank$ and~$x=\eta_{1}$, the result of Proposition~\ref{prop:bijrescurve} follows from Lemma~\ref{lem:closeddisctype2}.
\end{example}


The structure of smooth $k$-analytic curves is well understood. 

\begin{theorem}\label{thm:triangulationsmooth}
Every smooth $k$-analytic curve admits a triangulation in the sense of Theorem~\ref{thm:triangulation}. 
\end{theorem}

%

The result of Proposition~\ref{prop:retraction} also extends. If~$\calT$ is a non-empty triangulation of a smooth connected $k$-analytic curve~$X$, then there is a canonical deformation retraction of~$X$ onto the skeleton~$\Sigma_{\calT}$ of~$\calT$, which is a locally finite metric graph. We may also define the skeleton of~$X$ as in Definition~\ref{def:skeletonX}, and it satisfies the properties of Remark~\ref{rem:skeletonX}.


\begin{remark}\label{rem:triangulationmodel}
With this purely analytic formulation, Theorem~\ref{thm:triangulationsmooth} is due to A.~Ducros, who provided a purely analytic proof in~\cite{DucrosRSS}. It is very closely related to the semi-stable reduction theorem of S. Bosch and W. L\"utkebohmert (see~\cite{BoschLuetkebohmert85}): for each smooth $k$-analytic curve~$X$, there exists a finite extension~$\ell/k$ such that~$X_{\ell}$ admits a model over~$\ell^\circ$ whose special fiber is a semi-stable curve over~$\tilde \ell$, that is, it is reduced and its singularities are at worst double nodes. 

If a smooth $k$-analytic curve~$X$ admits a semi-stable model over~$k^\circ$, then we may associate to it a triangulation of~$X$. The points of~$S$, $\calA$ and~$\calD$ then correspond respectively to the irreducible components, the singular points and the smooth points of the special fiber of the model. Moreover, the genus of a point of~$S$ (which, in this case, coincides with its stable genus) is equal to the genus of the corresponding component. We refer to \cite[Theorem~4.3.1]{Berkovich90} for more details.

In the other direction, it is always possible to associate a model over~$k^\circ$ to a triangulation of~$X$, but it may fail to be semi-stable in general. The reader may consult \cite[\S 6.3 and \S 6.4]{DucrosRSS} for general results.

\end{remark}

%


\begin{definition}\label{def:genus}
Assume that~$k$ is algebraically closed. Let~$X$ be a smooth connected $k$-analytic curve. 
We define the \emph{genus of~$X$} to be
\[ g(X) :=b_{1}(X) + \sum_{x\in X^{(2)}} g(x),\]
where~$b_{1}(X)$ is the first Betti number of~$X$ and~$X^{(2)}$ the set of type~2 points of~$X$.


If $k$ is arbitrary, we define the genus of a smooth geometrically connected $k$-analytic curve~$X$ to be the genus of~$X_{\wka}$. 
\end{definition}

This notion of genus is compatible with the one defined in the algebraic setting.

\begin{theorem}\label{thm:genus}
For each smooth geometrically connected projective algebraic curve $\calX$ over~$k$, we have 
\[ g(\calX) = g(\calX^\an).\]
\end{theorem}

Let us finally comment that, among the results that are presented here, Theorem~\ref{thm:triangulationsmooth} is deep and difficult, but we will not need to use it since an easier direct proof is available for $k$-analytic $\AA^1$-like curves (see Theorem~\ref{thm:triangulation}). The others are rather standard applications of the general theory of curves.

\subsection{Mumford curves}

Let us now return to~$\AA^1$-like curves over~$k$. A special kind of such curves is obtained by asking for the existence of an open covering made of actual open Swiss cheeses over~$k$ rather than virtual ones.
Recall that open Swiss cheeses over~$k$ are defined as the complement of closed discs in an open disc over $k$.

\begin{definition}\label{def:Mumfordcurve}
A connected, compact $k$-analytic ($\AA^1$-like) curve~$X$ is called a \emph{$k$-analytic Mumford curve} if every point $x \in X$ has a neighborhood that is isomorphic to an open Swiss cheese over~$k$.
\end{definition}

\begin{remark}
Such a curve is automatically projective algebraic by Theorem~\ref{thm:smoothcompact}.
\end{remark}



The following proposition relates the definition of a $k$-analytic Mumford curve with the existence of a triangulation of a certain type, and therefore with the original algebraic definition given by Mumford in \cite{Mumford72}.
Its proof uses some technical notions that were not fully presented in the first sections of this text, but we believe that the result of the proposition is important enough to deserve to be fully included for completeness.

\begin{proposition}\label{prop:Mumfordtriangulation}
Let $X$ be a compact $k$-analytic curve.

If $g(X)=0$, then $X$ is a $k$-analytic Mumford curve if and only if $X$ is isomorphic to~$\pank$.

If $g(X) \ge 1$, then $X$ is a $k$-analytic Mumford curve if and only if there exists a triangulation $(S,\calA,\calD)$ of $X$ such that the points of $S$ are of stable genus~0 and the elements of~$\calA$ are open annuli.
\end{proposition}
\begin{proof}
$\bullet$ Assume that $g(X)=0$. If $X$ is isomorphic to~$\pank$, then it is obviously a Mumford curve. 

Conversely, assume that $X$ is a $k$-analytic Mumford curve. By Theorems~\ref{thm:smoothcompact} and~\ref{thm:genus}, it is isomorphic to the analytification of a projective smooth algebraic curve over~$k$. Therefore, to prove that it is isomorphic to~$\pank$, it is enough to prove that it has a $k$-rational point.

By assumption, $X$ contains an open Swiss cheese over~$k$. In particular, it contains an open annulus~$A$ over~$k$. Let~$x$ be a boundary point of the skeleton of~$A$. By assumption, $x$ has a neighborhood that is isomorphic to an open Swiss cheese over~$k$. It follows that~$A$ is contained in a strictly bigger annulus~$A'$ whose skeleton contains strictly contains that of~$A$. Arguing this way (possibly considering the union of all the annuli and applying the argument again), we show that~$X$ contains an open annulus over~$k$ of infinite modulus. At least one of its boundary points is a $k$-rational point, and the result follows. 

\medbreak

$\bullet$ Assume that $g(X) \ge 1$. If $X$ is a $k$-analytic Mumford curve, then it may be covered by finitely many Swiss cheeses over~$k$. The result follows from the fact that every Swiss cheese over~$k$ admits a triangulation $(S,\calA,\calD)$ such that the points of $S$ are of stable genus~0 and the elements of~$\calA$ are annuli.

Conversely, assume that there exists a triangulation $(S,\calA,\calD)$ of~$X$ satisfying the properties of the statement. Since $g(X) \ge 1$, we have $\calA \ne \emptyset$. Up to adding a point of~$S$ in the skeleton of each element of~$\calA$,
we may assume that all the elements of~$\calA$ have two distinct endpoints in~$X$.

Let $x\in S$. Denote by~$\calD_{x}$ (resp.~$\calA_{x}$) the set of elements of~$\calD$ (resp. $\calA$) that have~$x$ as an endpoint and set
\[ U_{x} := \{x\} \cup \bigcup_{D \in \calD_{x}} D \cup  \bigcup_{A \in \calA_{x}} A.\]
It is an open neighborhood of~$x$ in~$X$. Let us now enlarge~$U_{x}$ in the following way: for each $A \in \calA_{x}$, we paste a closed disc at the extremity of~$A$ that is different from~$x$. The resulting curve~$V_{x}$ is compact, hence the analytification of a projective smooth algebraic curve over~$k$, by Theorem~\ref{thm:smoothcompact}. Since~$x$ is of stable genus~0, the genus of the base-change $(V_{x})_{\wka}$ of~$V_{x}$ to~$\wka$ is~0. By Theorem~\ref{thm:genus}, we deduce that $(V_{x})_{\wka}$ is isomorphic to~$\panwka$. Since~$V_{x}$ contains $k$-rational points (inside the pasted discs, for instance), $V_{x}$ itself is isomorphic to~$\pank$. We deduce that~$U_{x}$ is a Swiss cheese over~$k$. 

Since any point of~$X$ has a neighborhood that is of the form~$U_{x}$ for some $x\in S$, it follows that~$X$ is a Mumford curve.
\end{proof}

\begin{remark}
If $X$ is a compact $k$-analytic curve and $k$ is algebraically closed, then Proposition~\ref{prop:Mumfordtriangulation} shows that the following properties are equivalent:
\begin{enumerate}
\item $X$ is a Mumford curve;
\item $X$ is an $\AA^1$-like curve;
\item the points of type~2 of~$X$ are all of genus~0.
\end{enumerate}
\end{remark}

\begin{remark}\label{rmk:algMumford}
Using the correspondence between triangulations and semi-stable models (see Remark~\ref{rem:triangulationmodel}), the result of Proposition \ref{prop:Mumfordtriangulation} says that $k$-analytic Mumford curves are exactly those for which there exists a semi-stable model over~$k^\circ$ whose special fiber consists of projective lines over $\tilde k$, intersecting transversally in $\tilde k$-rational points. 
This last condition is how algebraic Mumford curves are defined in Mumford's paper \cite{Mumford72}.

\end{remark}

\begin{corollary}\label{cor:Mumfordgenus}
Let $X$ be a $k$-analytic Mumford curve and $\calT$ be a triangulation of $X$.
Then the following quantities are equal:
\begin{enumerate} 
\item the genus of~$X$;
\item the cyclomatic number of the skeleton~$\Sigma_{\calT}$;
\item the first Betti number of~$X$.
\end{enumerate}
\end{corollary}
\begin{proof}
We may assume that~$\calT = (S,\calA,\calD)$ satisfies the conclusions of Proposition~\ref{prop:Mumfordtriangulation}. We will assume that~$\calA \ne \emptyset$, the other case being dealt with similarly. Consider the base-change morphism $\pi_{\wka/k} \colon X_{\wka} \to X$. By assumption, every element~$A$ of~$\calA$ is an annulus over~$k$, hence its preimage $\pi_{\wka/k}^{-1}(A)$ is an annulus over~$\wka$. In particular, $\pi_{\wka/k}$ induces a homeomorphism between the skeleton of $\pi_{\wka/k}^{-1}(A)$ and that of~$A$. Since each point of~$S$ lies at the boundary of the skeleton of an element of~$\calA$, we deduce that each point of~$S$ has exactly one preimage by~$\pi_{\wka}$.


It follows that the set $\calT' = (S',\calA',\calD')$ of~$X_{\wka}$, where 
\begin{itemize}
\item $S'$ is the set of preimages of the elements of~$S$ by~$\pi_{\wka/k}$; 
\item $\calA'$ is the set of preimages of the elements of~$\calA$ by~$\pi_{\wka/k}$;
\item $\calD'$ is the set of connected components of the preimages of the elements of~$\calD$ by~$\pi_{\wka/k}$
\end{itemize}
is a triangulation of~$X_{\wka}$ and, moreover, that $\pi_{\wka/k}$ induces a homeomorphism between the skeleta~$\Sigma_{\calT'}$ and~$\Sigma_{\calT}$. In particular, their cyclomatic numbers are equal.

Since~$X$ is a Mumford curve, all the points of type~2 of the curve~$X_{\wka}$ are of genus~0, hence the genus of~$X_{\wka}$ coincides with its first Betti number, hence with the cyclomatic number of~$\Sigma_{\calT'}$, by Proposition~\ref{prop:retraction}. The equality between (i) and (ii) follows. 

The equality between (ii) and (iii) follows from Proposition~\ref{prop:retraction} again.
\end{proof}

%
%
%
%
%


\section{Schottky groups}\label{sec:Schottky}

Let $(k,\va)$ be a complete valued field. Some of the material of this section is adapted from Mumford \cite{Mumford72}, Gerritzen and van der Put \cite{GerritzenPut80} and Berkovich \cite[Section~4.4]{Berkovich90}. 

\subsection{Schottky figures}\label{sec:Schottkyfigures}

Let $g\in \N_{\ge 1}$.

\begin{definition}\label{def:Schottkyfigure}
Let $\gamma_{1},\dotsc,\gamma_{g} \in \PGL_{2}(k)$. 
Let $\calB = (D^+(\gamma_{i}^\eps), 1\le i\le g, \eps \in \{\pm1\})$ be a family of pairwise disjoint closed discs in~$\pank$.
For each $i\in\{1,\dotsc,g\}$ and $\eps \in \{-1,1\}$, set 
\[D^-(\gamma_{i}^\eps) := \gamma_{i}^\eps (\PP^1_{k} - D^+(\gamma_{i}^{-\eps})).\]

We say that~$\calB$ is a \emph{Schottky figure} adapted to $(\gamma_{1},\dotsc,\gamma_{g})$ if, for each $i\in\{1,\dotsc,g\}$ and $\eps \in \{-1,1\}$, $D^-(\gamma_{i}^\eps)$ is a maximal open disc inside $D^+(\gamma_{i}^\eps)$.
\end{definition}

\begin{figure}
\includegraphics[scale=.5]{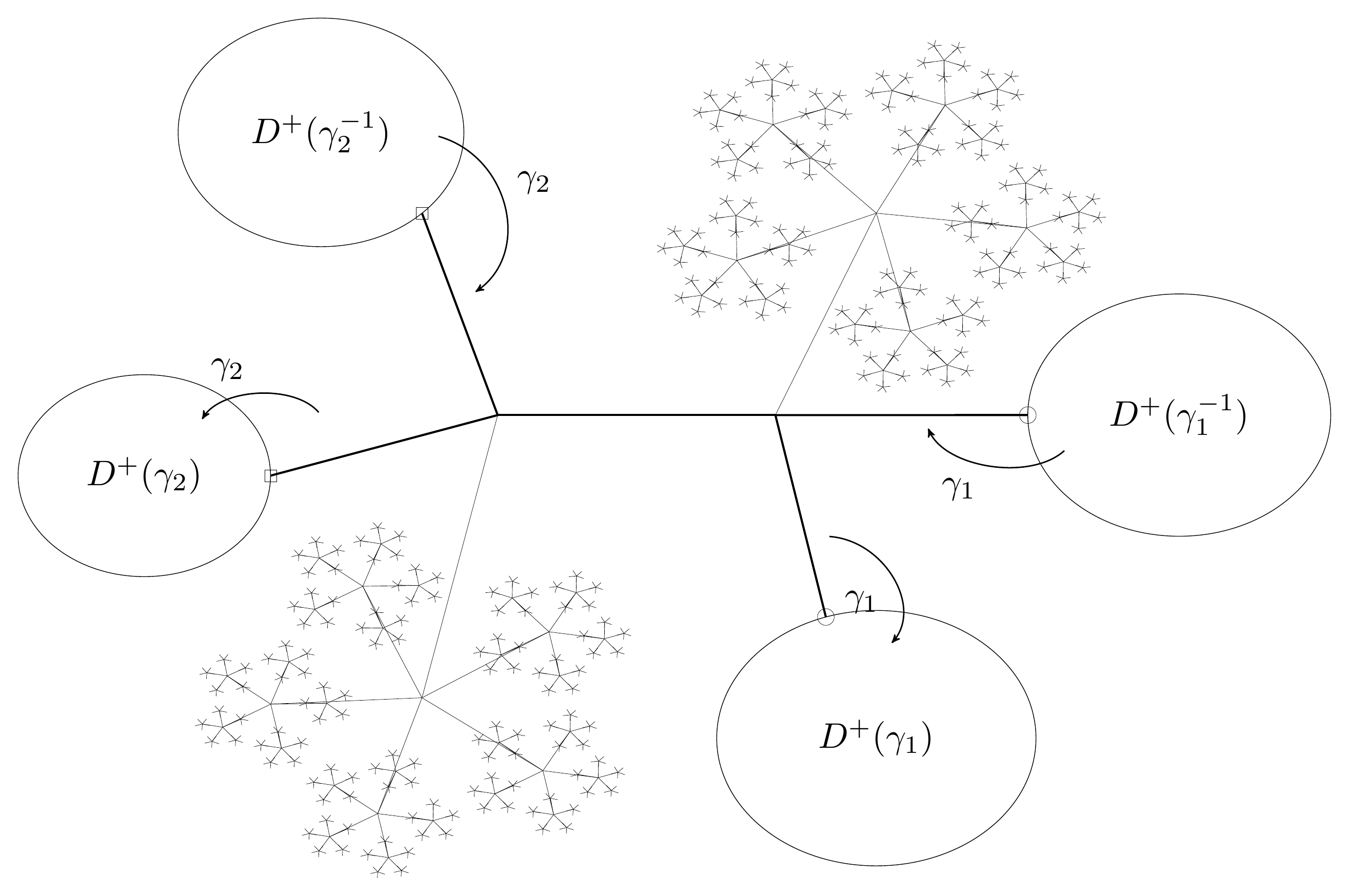}

\caption{A Schottky figure adapted to a pair $(\gamma_1, \gamma_2)$.}
\end{figure}

\begin{remark}\label{rem:Schottkyfigure}
Let $i\in \{1,\dotsc,g\}$. It follows from Remark~\ref{rem:nonloxodromic} that $\gamma_{i}$ is loxodromic. Moreover, denoting by~$\alpha_{i}$ and~$\alpha'_{i}$ the attracting and repelling fixed points of~$\gamma_{i}$ respectively, we have
\[\alpha'_{i} \in D^-(\gamma_{i}^{-1}) \textrm{ and } \alpha_{i} \in D^-(\gamma_{i}).\] 
The result is easily proven for $\gamma =  \left[
\begin{matrix} 
1 & 0\\
0 & q
\end{matrix}\right]$ and one may reduce to this case by choosing a suitable coordinate on~$\pank$.
\end{remark}

For the rest of the section, we fix $\gamma_{1},\dotsc,\gamma_{g} \in \PGL_{2}(k)$ and a Schottky figure adapted to $(\gamma_{1},\dotsc,\gamma_{g})$, with the notation of Definition~\ref{def:Schottkyfigure}.

\begin{notation}
For $\sigma\in \{-,+\}$, we set 
\[F^\sigma := \PP^1_{k} - \bigcup_{\substack{1\le i\le g \\ \eps=\pm1}} D^{-\sigma}(\gamma_{i}^\eps).\]
\end{notation}

Note that, for $\gamma \in\{\gamma_{1}^{\pm1},\dotsc,\gamma_{g}^{\pm1}\}$, $D^+(\gamma)$ is the unique disc that contains~$\gamma (F^+)$ among those defining the Schottky figure.


\begin{remark}\label{rem:skeletonF+}
The sets $F^-$ and $F^+$ are open and closed Swiss cheeses respectively.

Denote by $\partial F^+$ the boundary of~$F^+$ in~$\pank$. It is equal to the set of boundary points of the $D^+(\gamma_{i}^{\pm 1})$'s, for $i \in \{1,\dotsc,g\}$. The skeleton~$\Sigma_{F^+}$ of~$F^+$ is the convex envelope of~$\partial F^+$, that is to say the minimal connected graph containing~$\partial F^+$, or 
\[ \Sigma_{F^+} = \bigcup_{x,y \in \partial F^+} [x,y].\]
The skeleton~$\Sigma_{F^-}$ of~$F^-$ satisfies
\[\Sigma_{F^-} = \Sigma_{F^+} \cap F^- = \Sigma_{F^+} - \partial F^+.\]
\end{remark}

Set $\Delta := \{\gamma_{1},\dotsc,\gamma_{g}\}$. Denote by~$F_{g}$ the abstract free group with set of generators~$\Delta$ and by~$\Gamma$ the subgroup of $\PGL_{2}(k)$ generated by~$\Delta$. 
The existence of a Schottky figure for the $g$-tuple $(\gamma_{1},\dotsc,\gamma_{g})$ determines important properties of the group $\Gamma$.
In fact, we have a natural morphism $\varphi\colon F_{g} \to \Gamma$ inducing an action of~$F_{g}$ on~$\pank$.
We now define a disc in~$\PP^1_{k}$ associated with every elements of~$F_g$. As usual, we will identify these elements with the words over the alphabet $\Delta^{\pm} := \{\gamma^{\pm1}_{1},\dotsc,\gamma_{g}^{\pm1}\}$.

\begin{notation}
For a non-empty reduced word $w=w'\gamma$ over~$\Delta$ and $\sigma \in \{-,+\}$, we set
\[D^\sigma(w) := w' \,D^\sigma(\gamma).\]
\end{notation}

\begin{lemma}\label{lem:wordF}
Let $u$ be a non-empty reduced word over~$\Delta^\pm$. Then we have $u F^+ \subseteq D^+(u)$.

Let $v$ be a non-empty reduced word over~$\Delta^\pm$. If there exists a word~$w$ over~$\Delta^\pm$ such that $u = vw$, then we have $u F^+ \subseteq D^+(u) \subseteq D^+(v)$. If, moreover, $u\ne v$, then we have $D^+(u) \subseteq D^-(v)$.

Conversely, if we have $D^+(u) \subseteq D^+(v)$, then there exists a word~$w$ over~$\Delta^\pm$ such that $u = vw$.
\end{lemma}
\begin{proof}
Write in a reduced form $u = u'\gamma$ with $\gamma\in \Delta^\pm$. We have $\gamma F^+ \subseteq D^+(\gamma)$, by definition. Applying~$u'$, it follows that $u F^+ \subseteq D^+(u)$. 

\medbreak

Assume that there exists a word~$w$ such that $u = vw$ and let us prove that $D^+(u) \subseteq D^+(v)$. We first assume that~$v$ is a single letter. We will argue by induction on the length~$|u|$ of~$u$. If $|u|=1$, then $u=v$ and the result is trivial. If~$|u|\ge 2$, denote by~$\delta$ the first letter of~$w$. By induction, we have $D^+(w) \subseteq D^+(\delta)$. Since $\delta \ne v^{-1}$, we also have $D^+(\delta)\subseteq \PP^1_{k} - D^+(v^{-1})$. The result follows by applying~$v$.

Let us now handle the general case. Write in a reduced form $v=v' \gamma$ with $\gamma\in \Delta^\pm$. By the former case, we have $D^+(\gamma w) \subseteq D^+(\gamma)$ and $D^+(\gamma w) \subseteq D^-(\gamma)$ if $w$ is non-empty. The result follows by applying~$v'$. 

\medbreak

Assume that we have $D^+(u) \subseteq D^+(v)$. We will prove that there exists a word~$w$ such that $u = vw$ by induction on~$|v|$. Write in reduced forms $u=\gamma u'$ and $v=\delta v'$. By the previous result, we have $D^+(u) \subseteq D^+(\gamma)$ and $D^+(v) \subseteq D^+(\delta)$, hence $\gamma = \delta$. If $|v|=1$, this proves the result. If $|v|\ge 2$, then we deduce that we have $D^+(u') \subseteq D^+(v')$, hence, by induction, there exists a word~$w$ such that $u' = v'w$. It follows that $u=vw$.
\end{proof}

\begin{proposition}\label{prop:Gammafree}
The morphism~$\varphi$ is an isomorphism and the group~$\Gamma$ is free on the generators $\gamma_{1},\dotsc,\gamma_{g}$.
\end{proposition}
\begin{proof}
If $w$ is a non-empty word, then the previous lemma ensures that $wF^+ \ne F^+$. The result follows. 
\end{proof}

As a consequence, we now identify~$\Gamma$ with~$F_g$ and express the elements of~$\Gamma$ as words over the alphabet~$\Delta^\pm$. 
In particular, we allow us to speak of the length of an element~$\gamma$ of~$\Gamma$, that we denote by~$|\gamma|$. Set
\[O_{n} :=  \bigcup_{|\gamma|\le n} \gamma F^+.\]
Since the complement of~$F^+$ is the disjoint union of the open disks $D^-(\gamma)$ with $\gamma \in \Delta^\pm$, it follows from the description of the action that, for each $n\ge 0$, we have 
\[\pank - O_{n} = \bigsqcup_{|w|= n+1} D^-(w).\]
It follows from Lemma~\ref{lem:wordF} that, for each $n\ge 0$, $O_{n}$ is contained in the interior of~$O_{n+1}$. We set 
\[O := \bigcup_{n\ge 0} O_{n} = \bigcup_{\gamma\in\Gamma} \gamma F^+.\]

We now compute the orbits of discs under homographies of~$\pank$.
Set $\iota := \begin{bmatrix} 0&1\\1&0 \end{bmatrix} \in \PGL_{2}(k)$. It corresponds to the map $z \mapsto 1/z$ on~$\pank$. The first result follows from an explicit computation.

\begin{lemma}\label{lem:inversiondisc}
Let $\alpha\in k^\times$ and $\rho \in [0,|\alpha|)$.
Then, we have $\iota\big(D^+(\alpha,\rho)\big) = D^+\left(\frac{1}{\alpha}, \frac{\rho}{|\alpha|^2}\right)$.
\qed
\end{lemma}

\begin{lemma}\label{lem:radiusdisc}
Let $r>0$ and let $\gamma =
\begin{bmatrix} 
a & b \\
c & d
\end{bmatrix}$ in $\PGL_{2}(k)$ such that $\gamma \big( D^+(0,r) \big) \subseteq \aank$. Then, we have $|d|>r |c|$ and $\gamma \big( D^+(0,r) \big) = D^+\left(\frac{b}{d}, \frac{|ad-bc|\, r}{|d|^2}\right) $.
\end{lemma}
\begin{proof}
Let us first assume that $c=0$. Then, we have $d\ne 0$, so the inequality $|d|>r |c|$ holds, and~$\gamma$ is affine with ratio~$a/d$. The result follows. 

Let us now assume that $c\ne 0$. In this case, we have $\gamma^{-1}(\infty) = -\frac{d}{c}$, which does not belong to~$D(0,r)$ if, and only if, $|d| >r |c|$. Note that we have the following equality in~$k(T)$:
\[\frac{aT+b}{cT+d} = \frac{a}{c} - \frac{ad-bc}{c^2}\, \frac{1}{T+\frac{d}{c}}.\]
By Lemma~\ref{lem:inversiondisc}, there exist $\beta \in k$ and $\sigma>0$ such that $\iota \big( D^+(\frac{d}{c},r) \big) = D^+(\beta,\sigma)$. Then, we have $\gamma \big( D^+(0,r) \big) = D^+(\frac{a}{c} - \frac{ad-bc}{c^2}\,\beta, \big|\frac{ad-bc}{c^2}\big|\,\sigma)$ and the result follows from an explicit computation.
\end{proof}

\begin{lemma}\label{lem:quotient}
Let $D' \subseteq D$ be closed discs in~$\aank$. Let $\gamma \in \PGL_{2}(k)$ such that $\gamma D' \subseteq \gamma D \subseteq \A^{1,\an}_{k}$. Then, we have
\[\frac{\textup{radius of } \gamma D'}{\textup{radius of } \gamma D} = \frac{\textup{radius of } D'}{\textup{radius of } D}.\]
\end{lemma}
\begin{proof}
Let~$p$ be a $k$-rational point in~$D'$ and let~$\tau$ be the translation sending~$p$ to~0. Up to changing~$D$ into~$\tau D$, $D'$ into~$\tau D'$, $\gamma$ into~$\gamma \tau^{-1}$ and $\gamma'$ into~$\gamma' \tau^{-1}$, we may assume that~$D$ and~$D'$ are centered at~0. The result then follows from Lemma~\ref{lem:radiusdisc}.
\end{proof}

\begin{proposition}\label{prop:radiusn}
Assume that $\infty \in F^-$. Then, there exist $R>0$ and $c \in (0,1)$ such that, for each $\gamma \in \Gamma-\{\id\}$, $D^+(\gamma)$ is a closed disc of radius at most $R \,c^{|\gamma|}$. 
\end{proposition}
\begin{proof}
Let $\delta,\delta' \in \Delta^\pm$ such that $\delta' \ne \delta^{-1}$. By Lemma~\ref{lem:wordF}, we have $D^+(\delta' \delta) \subset D^-(\delta') \subseteq D^+(\delta')$. 
Set 
\[c_{\delta,\delta'} := \frac{\text{radius of } f_{\delta,\delta'}(D^+(\delta' \delta))}{\text{radius of } f_{\delta,\delta'}(D^+(\delta'))} \in (0,1).\]

For each $\gamma\in \Gamma$ such that $\gamma\delta'$ is a reduced word, by Lemma~\ref{lem:quotient}, we have
\[ \frac{\text{radius of } D^+(\gamma\delta'\delta)}{\text{radius of } D^+(\gamma\delta')} = \frac{\text{radius of } \gamma f_{\delta,\delta'}^{-1}f_{\delta,\delta'}(D^+(\delta'\delta))}{\text{radius of } \gamma f_{\delta,\delta'}^{-1} f_{\delta,\delta'}(D^+(\delta'))}  \le c_{\delta,\delta'}.\]

Set 
\[R := \max (\{\text{radius of } D^+(\gamma) \mid \gamma \in \Delta^\pm\})\]
and 
\[c := \max (\{c_{\gamma,\gamma'} \mid \gamma,\gamma' \in \Delta^\pm, \gamma' \ne \gamma^{-1}\}).\]
By induction, for each $\gamma \in \Gamma-\{\id\}$, we have 
\[\text{radius of } D^+(\gamma) \le R\, c^{|\gamma|}.\]
\end{proof}


\begin{corollary}\label{cor:loxodromic}
Every element of $\Gamma - \{\id\}$ is loxodromic.
\end{corollary}
\begin{proof}
In order to prove the result, we may extend the scalars. As a result, we may assume that $F^- \cap \pank(k) \ne \emptyset$, hence up to changing coordinates, that $\infty \notin F^-$. Let $\gamma \in \Gamma - \{\id\}$. By Proposition~\ref{prop:radiusn} the radii of the discs $\gamma^n(D^+(\gamma))$ tend to~0 when $n$~tends to~$\infty$, which forces~$\gamma$ to be loxodromic, by Remark~\ref{rem:nonloxodromic}.
\end{proof}

\begin{corollary}\label{cor:intersectiondiscs}
Let $w = (w_{n})_{\ne 0}$ be a sequence of reduced words over~$\Delta^\pm$ such that the associated sequence of discs $(D^+(w_{n}))_{n\ge 0}$ is strictly decreasing. Then, the intersection $\bigcap_{n\ge 0} D^+(w_{n})$ is a single $k$-rational point~$p_{w}$. Moreover, the discs $D^+(w_{n})$ form a basis of neighborhoods of~$p_{w}$ in~$\pank$.
\end{corollary}
\begin{proof} 
Let~$k_{0}$ be a finite extension of~$k$ such that $F^- \cap \PP^1(k_{0}) \ne \emptyset$. Consider the projection morphism $\pi_{0} \colon \pana{k_{0}} \to \pank$. For each $i\in\{1,\dotsc,g\}$, $\gamma_{i}$ may be identified with an element~$\gamma_{i,0}$ in $\PGL_{2}(k_{0})$. 
The family $(\pi_{0}^{-1}(D^-(\gamma_{i}^{\pm1}), 1\le i\le g, \eps=\pm1)$ is a Schottky figure adapted to $(\gamma_{1,0},\dotsc,\gamma_{g,0})$. We will denote with a subscript~0 the associated sets: $F^-_{0}$, $D^+_{0}(w)$, etc. Note that these sets are all equal to the preimages of the corresponding sets by~$\pi_{0}$. 

Up to changing coordinates on~$\pana{k_{0}}$, we may assume that $\infty \in F_{0}^-$. The sequence of discs $(D_{0}^+(w_{n}))_{n\ge 0}$ is strictly decreasing, so by Lemma~\ref{lem:wordF}, the length of~$w_{n}$ tends to~$\infty$ when~$n$ goes to~$\infty$ and, by Proposition~\ref{prop:radiusn}, the radius of~$D_{0}^+(w_{n})$ tends to~0 when~$n$ goes to~$\infty$. It follows that $\bigcap_{n\ge 0} D_{0}^+(w_{n})$ is a single point~$p_{w,0}$ of type 1 and that the discs $D_{0}^+(w_{n})$ form a basis of neighborhood of~$p_{w,0}$ in~$\pana{k_{0}}$.

Set $p_{w} := \pi_{0}(p_{w,0})$. It follows from the results over~$k_{0}$ that $\bigcap_{n\ge 0} D^+(w_{n}) =\{p_{w}\}$ and that the discs $D^+(w_{n})$ form a basis of neighborhoods of~$p_{w}$ in~$\pank$. 

It remains to show that~$p_{w}$ is $k$-rational. Note that~$p_{w}$ belongs to the closure of~$\PP^1(k)$, since it is the limit of the centers of the $D^+(w_{n})$'s. Since~$k$ is complete, $\PP^1(k)$ is closed in $\PP^1(\wka)$ and the result follows.
\end{proof}

\begin{corollary}\label{cor:limittype1}
The set~$O$ is dense in~$\pank$ and its complement is contained in~$\PP^1(k)$. \qed
\end{corollary}


\begin{definition}\label{def:limit}
We say that a point $x\in \pank$ is a \emph{limit point} if there exist $x_{0} \in \pank$ and a sequence $(\gamma_{n})_{n\ge 0}$ of distinct elements of~$\Gamma$ such that $\lim_{n\to\infty} \gamma_{n}(x_{0}) = x$.

The \emph{limit set}~$L$ of~$\Gamma$ is the set of limit points of~$\Gamma$.
\end{definition}

Let us add a short reminder on proper group actions.

\begin{definition}[\protect{\cite[III, \S4, D\'efinition~1]{BourbakiTG14}}]\label{def:proper}
We say that the action of a topological group~$G$ on a topological space~$E$ is \emph{proper} if 
the map
\[\begin{array}{ccc}
\Gamma \times E &\to & E \times E\\
(\gamma,x) &\mapsto & (x,\gamma\cdot x)
\end{array}\] 
is proper.
\end{definition}

\begin{proposition}[\protect{\cite[III, \S4, Propositions~3 and~7]{BourbakiTG14}}]\label{prop:proper}
Let~$G$ be a locally compact topological group and~$E$ be a Hausdorff topological space. Then, the action of~$G$ on~$E$ is proper if, and only if, for every $x, y \in E$, there exist neighborhoods~$U_{x}$ and~$U_{y}$ of~$x$ and~$y$ respectively such that the set $\{\gamma\in \Gamma \mid \gamma U_{x} \cap U_{y} \ne\emptyset\}$ is relatively compact (that is to say finite, if $G$ is discrete).

In this case, the quotient space $\Gamma\backslash E$ is Hausdorff.
\qed
\end{proposition}

%

We denote by~$C$ the set of points $x\in \pank$ that admit a neighborhood~$U_{x}$ satisfying $\{\gamma \in \Gamma \mid \gamma U_{x} \cap U_{x} \ne\emptyset\} =\{\id\}$. 
The set~$C$ is an open subset of~$\pank$ and the quotient map $C \to \Gamma\backslash C$ 
is a local homeomorphism. In particular, the topological space $\Gamma\backslash C$ 
is naturally endowed with a structure of analytic space \textit{via} this map.

\begin{theorem}\label{thm:actionproper}
We have $O = C = \pank - L$. Moreover, the action of~$\Gamma$ on~$O$ is free and proper and the quotient $\Gamma\backslash O$ is a Mumford curve of genus~$g$. 

Set $X := \Gamma\backslash O$ and denote by $p \colon O \to X$ the quotient map. Let $\Sigma_{O}$, $\Sigma_{F^+}$ and $\Sigma_{X}$ denote the skeleta of~$O$, $F^+$ and~$X$ respectively. Then, $\Sigma_{O}$ is the trace on~$O$ of the convex envelope of~$L$:
\[\Sigma_{O} = O \cap \bigcup_{x,y \in L} [x,y]\]
and we have
\[ p^{-1}(\Sigma_{X}) = \Sigma_{O} \textrm{ and } p(\Sigma_{O}) = p(\Sigma_{F^+}) = \Sigma_{X}. \]
\end{theorem}
\begin{proof}
Let $x \in L$. By definition, there exists $x_{0} \in \pank$ and a sequence $(\gamma_{n})_{n\ge 0}$ of distinct elements of~$\Gamma$ such that $\lim_{n\to\infty} \gamma_{n}(x_{0}) = x$. Assume that $x\in F^+$. Since $F^+$ is contained in the interior of~$O_{1}$, there exists $N\ge 0$ such that $\gamma_{N}(x_{0}) \in O_{1}$, hence we may assume that $x_{0} \in O_{1}$. Lemma~\ref{lem:wordF} then leads to a contradiction. It follows that~$L$ does not meet~$F^+$, hence, by $\Gamma$-invariance, $L$ is contained in $\pank - O$.

Let $y\in \pank - O$. By definition, there exists a sequence $(w_{n})_{n\ge 0}$ of reduced words over~$\Delta^\pm$ such that, for each $n\ge 0$, $|w_{n}|\ge n$ and $y \in D^-(w_{n})$. Let~$y_{0} \in F^-$. By Lemma~\ref{lem:wordF}, for each $n\ge 0$, we have $w_{n}(y_{0}) \in D^-(w_{n})$ and the sequence of discs $(D^+(w_{n}))_{n\ge 0}$ is strictly decreasing. By Corollary~\ref{cor:intersectiondiscs}, $(w_{n}(y_{0}))_{n\ge 0}$ tends to~$y$, hence $y\in L$. It follows that $\pank - O = L$.

Set 
\[U := F^+ \cup \bigcup_{\gamma \in \Delta^\pm} \gamma F^- = \pank - \bigsqcup_{|\gamma|=2} D^+(\gamma).\] 
It is an open subset of~$\pank$ and it follows from the properties of the action (see Lemma~\ref{lem:wordF}) that we have $\{\gamma\in \Gamma \mid \gamma U \cap U \ne \emptyset\} = \{\id\} \cup \Delta^\pm$. Using the fact that the stabilizers of the points of~$U$ are trivial, we deduce that $U\subseteq C$. Letting~$\Gamma$ act, it follows that $O \subseteq C$. Since no limit point may belong to~$C$, we deduce that this is actually an equality.

\medbreak

We have already seen that the action is free on~$O$. Let us prove that it is proper. Let $x, y \in O$. There exists $n\ge 0$ such that~$x$ and~$y$ belong to the interior of~$O_{n}$. By Lemma~\ref{lem:wordF}, the set $\{\gamma\in\Gamma \mid \gamma O_{n} \cap O_{n} \ne \emptyset\}$ is made of elements of length at most~$2n+1$. In particular, it is finite. We deduce that the action of~$\Gamma$ on~$O$ is proper.

\medbreak

The compact subset~$F^+$ of~$\pank$ contains a point of every orbit of every element of~$O$. It follows that $\Gamma\backslash O$ is compact. 
The set~$F^-$ is an open $k$-Swiss cheese and the map~$p$ is injective on it, which implies that $p_{|F^-}$ induces an isomorphism onto its image. In addition, one may check that each subset of the form $D^+(\gamma) - D^-(\gamma)$ for $\gamma  \in\{\gamma_{1}^{\pm1},\dotsc,\gamma_{g}^{\pm1}\}$ is contained in an open $k$-annulus on which~$p$ is injective.
It follows that any element of~$\Gamma\backslash O$ has a neighborhood isomorphic to a $k$-Swiss cheese, hence $\Gamma\backslash O$ is a Mumford curve.


\medbreak

Set $\Sigma := O \cap \bigcup_{x,y \in L} [x,y]$. It is clear that no point of~$\Sigma$ is contained in a virtual open disc inside~$O$, hence $\Sigma \subseteq \Sigma_{O}$. It follows from Proposition~\ref{prop:aan-S} that $\pank - \Sigma$ is a union of virtual open discs, hence $\Sigma_{O} \cap (\pank - \Sigma) = \emptyset$. We deduce that $\Sigma_{O} = \Sigma$. Note that it follows that $\Sigma_{F^+} = \Sigma_{O} \cap F^+$.

Let $x\in O - \Sigma_{O}$. Then $x$ is contained in a virtual open disc inside~$O$. Assume that there exists $\gamma \in \Gamma$ such that $x \in \gamma F^-$. Then, the said virtual open disc is contained in~$\gamma F^-$. Since $p_{|\gamma F^-}$ induces an isomorphism onto its image, $p(x)$ is contained in a virtual open disc in~$X$, hence $p(x) \notin \Sigma_{X}$. As above, the argument may be adapted to handle all the points of $O - \Sigma_{O}$. It follows that $p^{-1}(\Sigma_{X}) \subseteq \Sigma_{O}$.

Let $x \in \Sigma_{O}$. In order to show that $p(x) \in \Sigma_{X}$, we may replace $x$ by $\gamma(x)$ for any $\gamma\in \Gamma$, hence assume that $x \in F^+ \cap \Sigma_{O} = \Sigma_{F^+}$. From the explicit description of the action of~$\Gamma$ on~$F^+$, we may describe precisely the behaviour of~$p$ on $\Sigma_{F^+} = \Sigma_{F^-} \cup \partial F^+$: it is injective on~$\Sigma_{F^-}$ and identifies pairs of points in~$\partial F^+$. It follows that $p(x)$ belongs to a injective loop inside~$X$ and Remark~\ref{rem:skeletonX} then ensures that~$p(x) \in \Sigma_{X}$.
The results about the skeleta follow directly.

\medbreak

It remains to prove that the genus of $X = \Gamma\backslash O$ is equal to~$g$. The arguments above show that $\Sigma_{X} \simeq \Gamma\backslash \Sigma_{F^+}$ is a graph with cyclomatic number~$g$. The result now follows from Corollary~\ref{cor:Mumfordgenus}.
\end{proof}

\begin{figure}[ht]
\centering
    \begin{subfigure}[b]{0.44\textwidth}
\includegraphics[scale=.3]{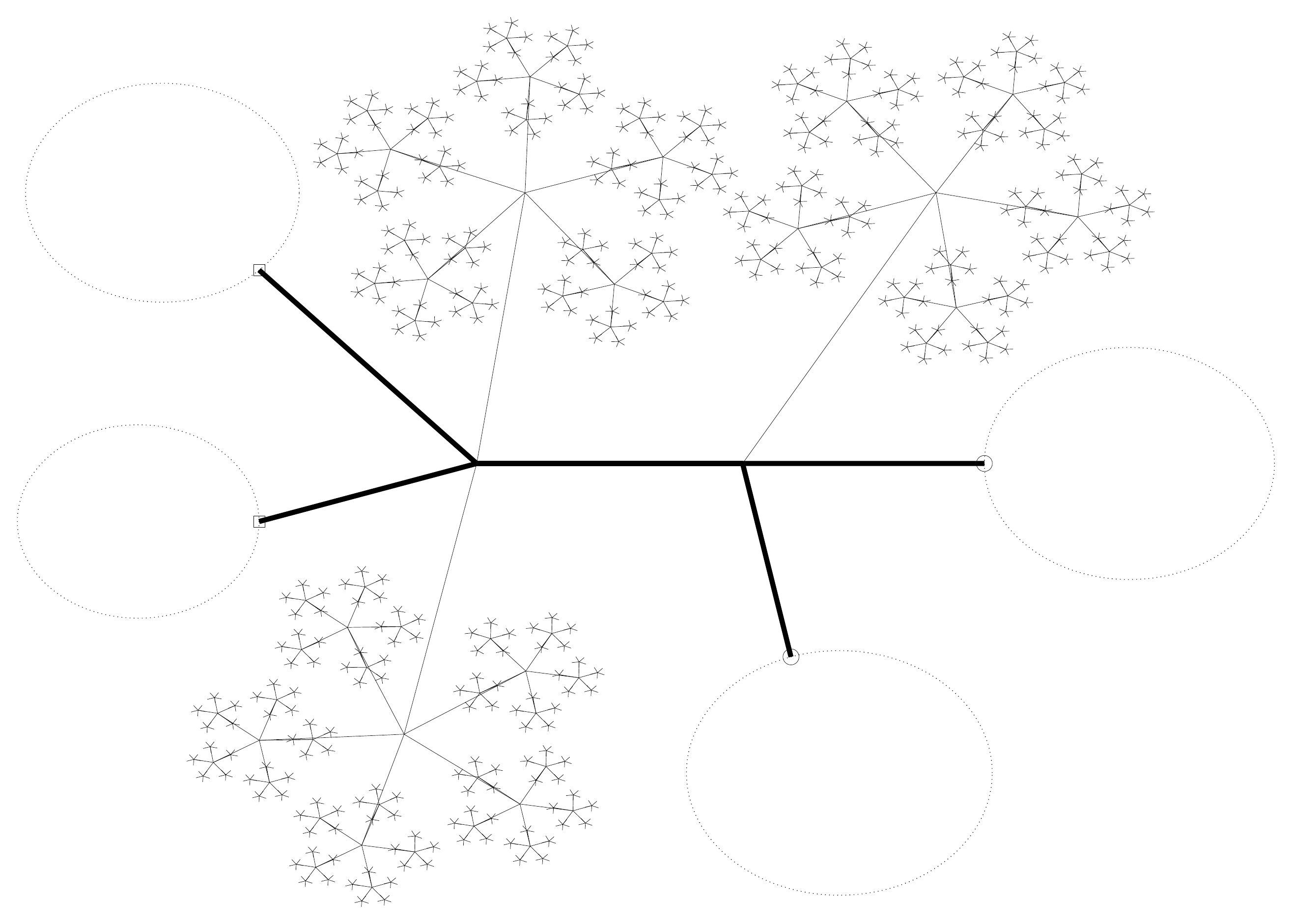}
    \end{subfigure}
    \hspace{1.5cm}
    \begin{subfigure}[b]{0.44\textwidth}
\includegraphics[scale=.3]{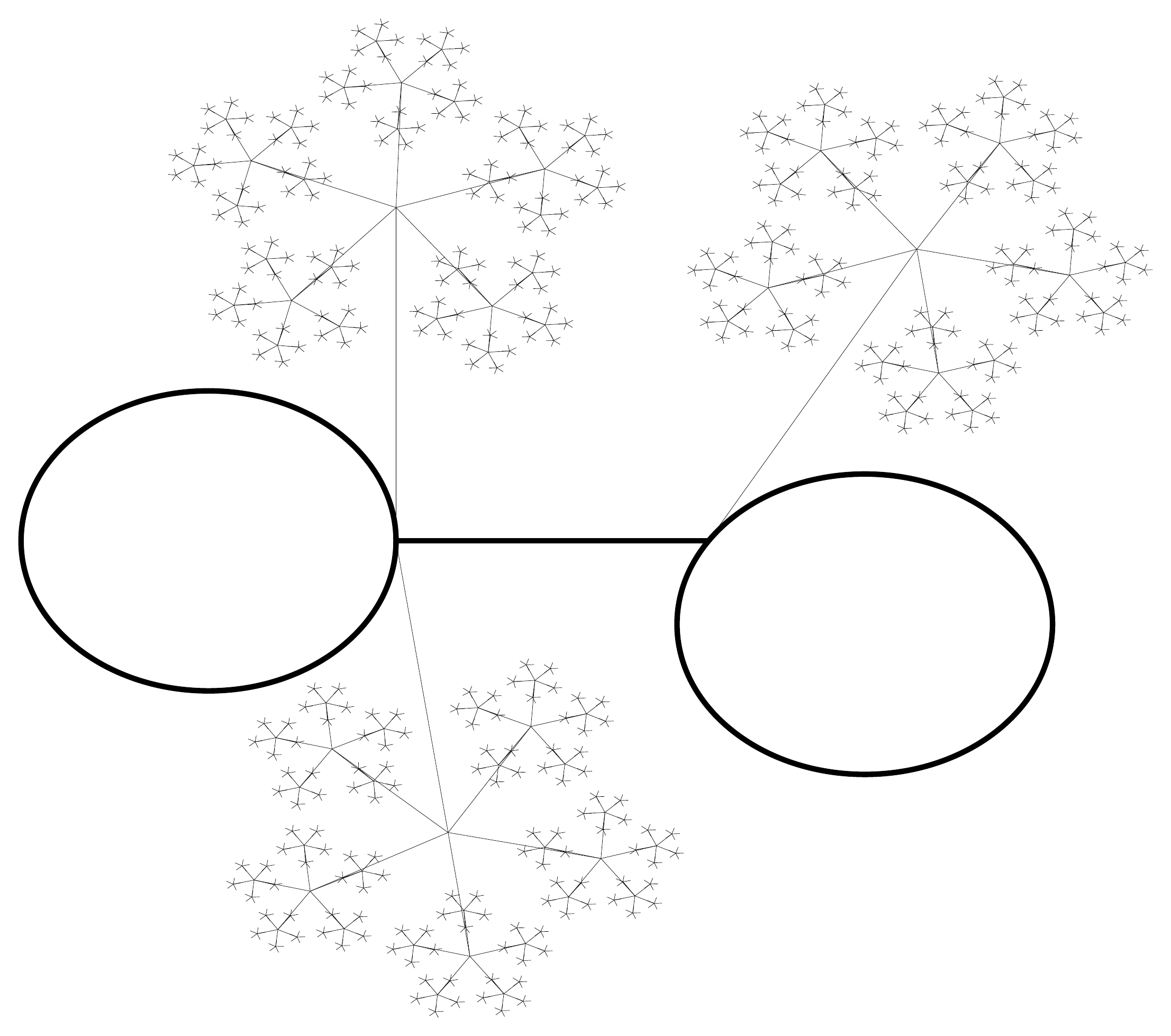}
    \end{subfigure}

 \caption{The closed fundamental domain $F^+$ (on the left) of the Schottky group $\Gamma$ is a Swiss cheese. The group $\Gamma$ identifies the ends of the skeleton $\Sigma_{F^+}$, so that the corresponding Mumford curve (on the right) contains the finite graph $\Sigma_X$.}
 \label{fig:Test}
\end{figure}

\begin{example}[Tate curves]
If $g=1$ in the theory above, one starts with the datum of an element $\gamma \in \PGL_2(k)$ and of two disjoint closed discs $D^+(\gamma)$  and $D^+(\gamma^{-1})$ in such a way that
$\gamma(\pank - D^+(\gamma^{-1}))$ is a maximal open disc inside $D^+(\gamma)$.
Since $\gamma$ is loxodromic, up to conjugation it is represented by a matrix of the form $\begin{bmatrix} q & 0 \\ 0 & 1 \end{bmatrix}$ for some $q \in k$ satisfying $0<|q|<1$.
In other words, up to a change of coordinate in $\pank$, the transformation $\gamma$ is the multiplication by $q$ and hence the limit set $L$ consists only of the two points $0$ and $\infty$.
The quotient curve obtained from applying Theorem \ref{thm:actionproper} is an elliptic curve, whose set of $k$-points is isomorphic to the multiplicative group $k^\times/ q^\Z$.
\end{example}

\begin{remark}\label{rem:Indra}
It follows from Theorem~\ref{thm:actionproper} and Corollary~\ref{cor:intersectiondiscs} that each point in the limit set may be described as the intersection of a nested sequence of discs of the form $\bigcap_{n \ge 0} D^+(w_{n})$, for a sequence of words~$w_{n}$ whose lengths tend to infinity. This is a rather concrete description, that could easily be implemented to any precision on a computer. The complex version of this idea gave rise to beautiful pictures in~\cite{IndrasPearls}.

Actually, we highly recommend the whole book~\cite{IndrasPearls} to the reader. It starts with the example of a complex Schottky group with two generators in a very accessible way and then presents a large amount of advanced material in a carefully explained way, with an original and colorful terminology, enriched with many pictures. Among the subjects covered are the Hausdorff dimension of the limit set (``fractal dust''), the degeneration of the notion of Schottky groups when the discs in the Schottky figures become tangent (``kissing Schottky groups''), etc. We believe that it is worth investigating those questions in the non-Archimedean setting too. In particular, finding a way to draw meaningful non-Archimedean pictures would certainly be very rewarding.
\end{remark}

\subsection{Group-theoretic version}\label{sec:groupSchottky}

We now give the general definition of Schottky group over~$k$ and explain how it relates to the geometric situation considered in the previous sections. As regards proper action, recall Definition~\ref{def:proper} and Proposition~\ref{prop:proper}.

\begin{definition}\label{def:SchottkyGroup}
A subgroup~$\Gamma$ of $\PGL_2(k)$ is said to be a \emph{Schottky group over~$k$} if
\begin{enumerate}
\item it is free and finitely generated;
\item all its non-trivial elements are loxodromic;
\item there exists a non-empty $\Gamma$-invariant connected open subset of~$\pank$ on which the action of~$\Gamma$ is free and proper.
\end{enumerate}

\end{definition}


\begin{remark}\label{rem:discrete}
Schottky groups are discrete subgroups of~$\PGL_{2}(k)$. Indeed any element of~$\PGL_{2}(k)$ that is close enough to the identity has both eigenvalues of absolute value~1, hence cannot be loxodromic.
\end{remark}

\begin{remark}

There are other definitions of Schottky groups in the literature. L.~Gerritzen and M.~van der Put use a slightly different version of condition~(iii) (see \cite[I (1.6)]{GerritzenPut80}). This is due to the fact that they work in the setting of rigid geometry, where the space consists only of our rigid points. We chose to formulate our definition this way in order to take advantage of the nice topological properties of Berkovich spaces and make it look closer to the definition used in complex geometry.

D.~Mumford considered a more general setting where $k$ is the fraction field of a complete integrally closed noetherian local ring and he requires only properties~(i) and~(ii) in his definition of Schottky group (see \cite[Definition~1.3]{Mumford72}). The intersection with our setting consists of the complete discretely valued fields~$k$.

However, when $k$ is a local field, all the definitions coincide (see \cite[I (1.6.4)]{GerritzenPut80} and Section~\ref{sec:localfields}).
\end{remark}

Schottky groups arise naturally when we have Schottky figures as in Section~\ref{sec:Schottkyfigures}. Indeed, the following result follows from Proposition~\ref{prop:Gammafree}, Corollary~\ref{cor:loxodromic} and Theorem~\ref{thm:actionproper}.

\begin{proposition}\label{prop:geometrygroup}
Let~$\Gamma$ be a subgroup of $\PGL_{2}(k)$ generated by finitely many elements $\gamma_{1},\dotsc,\gamma_{g}$. If there exists a Schottky figure adapted to $(\gamma_{1},\dotsc,\gamma_{g})$, then~$\Gamma$ is a Schottky group.
\qed
\end{proposition}

We now turn to the proof of the converse statement.

\begin{lemma}\label{lem:annuli}
Let~$\gamma$ be a loxodromic M\"obius transformation. Let~$A$ and~$A'$ be disjoint virtual flat closed annuli. Denote by~$I$ the open interval given as the interior of the path joining their boundary points. Assume that $\gamma A_{1} = A_{2}$ and $\gamma I \cap I = \emptyset$. For $\eps \in \{\emptyset,'\}$, denote by~$D^\eps$ the connected component of $\pank - A^\eps$ that does not meet~$I$. Then, for $\eps\in \{\emptyset,'\}$, $A^\eps$ is a flat closed annulus, $D^\eps$ is an open disc, $E^\eps := D^\eps \cup A^\eps$ is a closed disc and we have
\[
\gamma D = \pank - E' \textrm{ and } \gamma E = \pank - D'.
\]
\end{lemma}
\begin{proof}
For each $\eps \in \{\emptyset,'\}$, $D^\eps$ and~$E^\eps$ are respectively a virtual open disc and a virtual closed disc. Note that the set $\pank - A^\eps$ has two connected components, namely $D^\eps$ and $\pank - E^\eps$, and that the latter contains~$I$.

Since~$\gamma$ is an automorphism, it sends the connected component~$\pank-E$ of $\pank - A$ to a connected component~$C$ of $\pank - \gamma A = \pank - A'$. Denote by~$\eta$ and~$\eta'$ the boundary points of~$A$ and~$A'$. Let $z \in \pank - E$. The unique path $[\eta,z]$ between~$\eta$ and~$z$ then meets~$I$. Its image is the unique path $[\eta',\gamma(z)]$ between $\gamma(\eta) = \eta'$ and $\gamma(z)$. If $\gamma(z) \notin E'$, then this path meets~$I$, contradicting the assumption $\gamma I \cap I = \emptyset$. We deduce that $\gamma(z) \in E'$, hence that $C = D'$. It follows that we have 
\[
\gamma D = \pank - E' \textrm{ and } \gamma E = \pank - D',
\]
as wanted. 

In particular, $D$ and~$D'$ contain respectively the attracting and repelling fixed point of~$\gamma$. Since those points are $k$-rational, we deduce that~$D$ and~$D'$ are discs. The rest of the result follows.
\end{proof}

\begin{theorem}\label{thm:groupfigure}
Let~$\Gamma$ be a Schottky group over~$k$. Then, there exists a basis~$\beta$ of~$\Gamma$ and a Schottky figure~$\calB$ that is adapted to~$\beta$. 
\end{theorem}
\begin{proof}
By assumption, there exists a non-empty $\Gamma$-invariant connected open subset~$U$ of~$\pank$ on which the action of~$\Gamma$ is free and proper. The quotient $X := \Gamma\backslash U$ is then an $\aank$-like curve in the sense of Section~\ref{sec:A1like}. Since~$U$ is a connected subset of~$\pank$, it is simply connected, hence the fundamental group~$\pi_{1}(X)$ of~$X$ is isomorphic to~$\Gamma$. Since~$X$ is finitely generated, the topological genus~$g$ of~$X$ is finite.

%
%
%

Fix a skeleton~$\Sigma$ of~$X$ and consider the associated retraction $\tau \colon X \to \Sigma$. Fix $g$~elements $\gamma_{1},\dotsc,\gamma_{g}$ of~$\Gamma$ corresponding to disjoint simple loops in~$\Sigma_{X}$. Note that $\gamma_{1},\dotsc,\gamma_{g}$ is a basis of~$\Gamma$.


For each $i\in \{1,\dotsc,g\}$, pick a point $x_{i} \in \alpha_{i}$ that is not a branch point of~$\Sigma$. Its preimage by the retraction $A_{i} := \tau^{-1}(x_{i})$ is then a virtual flat closed annulus. 

Let~$Y'$ be an open subset of~$U$ such that the morphism $Y' \to X$ induced by the quotient is an isomorphism onto $X - \bigcup_{1\le i\le g}$. We extend it to a compact lift~$Y$ of~$X$ in~$U$ by adding, for each $i\in \{1,\dotsc,g\}$, two virtual flat annuli~$B_{i}$ and~$B'_{i}$ that are isomorphic preimages of~$A_{i}$. Up to switching the names, we may assume that $\gamma_{i} B_{i} = B'_{i}$.

Let $i\in \{1,\dotsc,g\}$. The complement of $B_{i}$ (resp. $B'_{i}$) has two connected components. Let us denote by~$D^-(\gamma_{i})$ (resp. $D^-(\gamma_{i}^{-1})$) the one that does not meet~$Y$. It is a virtual open disc. We set $D^+(\gamma_{i}^{-1}) = D^-(\gamma_{i}^{-1})\cup B_{i}$ and $D^+(\gamma_{i}) = D^-(\gamma_{i})\cup B'_{i}$. 

By construction of~$Y'$, for each $\gamma \in \Gamma - \{\id\}$, we have $\gamma Y' \cap Y' = \emptyset$. It now follows from Lemma~\ref{lem:annuli} that the family $(D^+(\gamma_{i}^\sigma), 1\le i\le g, \sigma = \pm)$ is a Schottky figure adapted to $(\gamma_{1},\dotsc,\gamma_{g})$.

\end{proof}

\begin{remark}\label{rem:free}
The fact that~$\Gamma$ is free is actually not used in the proof of Theorem~\ref{thm:groupfigure}. As a result, Proposition~\ref{prop:geometrygroup} shows that it is a consequence of the other properties appearing in the definition of a Schottky group. It could also be deduced from the fact that the fundamental group of a Berkovich curve (which is the same as that of its skeleton) is free.
\end{remark}

\subsection{Twisted Ford discs}

We can actually be more precise about the form of the discs in the Schottky figure from Theorem~\ref{thm:groupfigure}. To do so, we introduce some terminology.

\begin{definition}
Let $\gamma = \begin{bmatrix} 
a & b \\
c & d
\end{bmatrix} \in \PGL_2(k)$, with $c\ne 0$, be a loxodromic M\"obius transformation
and let $\lambda \in \R_{>0}$. We call open and closed \emph{twisted Ford discs} associated to $(\gamma,\lambda)$ the sets 
\[D_{(\gamma, \lambda)}^- := \Big\{z\in k \ \Big|\ \lambda |\gamma'(z)|=\lambda \frac{|ad-bc|}{|cz+d|^2} > 1\Big\}\]
and
\[D_{(\gamma, \lambda)}^+ := \Big\{z\in k \ \Big|\  \lambda |\gamma'(z)|=\lambda \frac{|ad-bc|}{|cz+d|^2} \ge 1\Big\}.\]
\end{definition}

\begin{lemma}\label{lem:explicitisometricdiscs}
Let $\alpha, \alpha',\beta \in k$ with $\alpha\ne \alpha'$ and $|\beta|<1$ and let $\lambda \in \R_{>0}$. Set $\gamma := M(\alpha, \alpha', \beta) = \begin{bmatrix} a&b\\c&d \end{bmatrix}$. The twisted Ford discs $D_{(\gamma, \lambda)}^-$ and $D_{(\gamma, \lambda)}^+$ have center 
\[\frac{\alpha' - \beta \alpha}{1-\beta} = -\frac d c\]
and radius 
\[\rho=\frac{(\lambda|\beta|)^{1/2}|\alpha-\alpha'|}{|1-\beta|} = \frac{(\lambda \,|ad-bc|)^{1/2}}{|c|}.\]
In particular, $\alpha' \in D^-_{(\gamma,\lambda)}$ if, and only if, $|\beta| < \lambda$.

The twisted Ford discs $D_{(\gamma^{-1}, \lambda^{-1})}^-$ and $D_{(\gamma^{-1}, \lambda^{-1})}^+$ have center 
\[\frac{\alpha - \beta \alpha'}{1-\beta} = \frac a c\]
and radius $\rho' =\rho/\lambda$. 

In particular, $\alpha \in D^-_{(\gamma^{-1},\lambda^{-1})}$ if, and only if, $|\beta| < \lambda^{-1}$.
\qed
\end{lemma}

\begin{lemma}
Let $\gamma \in \PGL_2(k)$ be a loxodromic M\"obius transformation that does not fix~$\infty$ and let $\lambda\in \R_{>0}$. Then, we have $\gamma(D_{(\gamma, \lambda)}^+)= \pank - {D}_{(\gamma^{-1}, \lambda^{-1})}^-$.
\end{lemma}
\begin{proof}
Let us write $\gamma = 
\begin{bmatrix} 
a & b \\
c & d
\end{bmatrix}$. Since~$\gamma$ does not fix~$\infty$, we have $c\neq 0$. Let $K$ be a complete valued extension of~$k$ and let $z\in K$. We have $|-c\gamma(z)+a| \, |cz+d| = |ad-bc|$, hence 
\[ z\in D_{(\gamma, \lambda)} \iff \lambda \frac{|ad-bc|}{|cz+d|^2}\geq 1 \iff \lambda^{-1}\frac{|ad-bc|}{|-c\gamma(z)+a|^2}\leq 1. \]
Since we have $\gamma^{-1} =
\begin{bmatrix} 
d & -b \\
-c & a
\end{bmatrix}$,
the latter condition describes precisely the complement of $D^-_{(\gamma^{-1}, \lambda^{-1})}$.
\end{proof}

\begin{lemma}\label{lem:pairFord}
Let $\gamma \in \PGL_2(k)$ be a loxodromic M\"obius transformation. 
Let $D^+(\gamma)$ and $D^+(\gamma^{-1})$ be disjoint closed discs in~$\pank$. Set 
\[D^-(\gamma) := \gamma(\pank - D^+(\gamma^{-1})) \textrm{ and } D^-(\gamma^{-1}) := \gamma^{-1}(\pank - D^+(\gamma)).\]
Assume that $D^-(\gamma)$ and $D^-(\gamma^{-1})$ are maximal open discs inside $D^+(\gamma)$ and $D^+(\gamma^{-1})$ respectively and that they are contained in~$\aank$.

Then, there exists $\lambda\in \R_{>0}$ such that, for each $\sigma \in \{-,+\}$, we have 
\[D^\sigma(\gamma) = D^\sigma_{\gamma,\lambda} \textrm{ and } D^\sigma(\gamma^{-1}) = D^\sigma_{\gamma^{-1},\lambda^{-1}}.\]
\end{lemma}
\begin{proof}
Denote by $\alpha$ and~$\alpha'$ the attracting and repelling fixed points of~$\gamma$ respectively. By the same argument as in Remark~\ref{rem:Schottkyfigure}, we have $\alpha \in D^-(\gamma^{-1})$ and $\alpha' \in D^-(\gamma)$. Let $r,r'>0$ such that $D^-(\gamma) = D^-(\alpha',r')$ and $D^-(\gamma^{-1}) = D^-(\alpha,r)$.

Write $\gamma = \begin{bmatrix} a & b\\ c&d \end{bmatrix}$ with $a,b,c,d \in k$. Since $\alpha, \alpha' \in \aank$, we have $c\ne 0$. By assumption, $\infty \in \gamma(D^-(\gamma^{-1}))$, hence $-d/c \in D^-(\gamma^{-1})$ and $D^-(\gamma^{-1}) = D^-(-d/c,r)$. Similarly, we have $D^-(\gamma) = D^-(a/c,r')$.

Writing
\[\frac{aT+b}{cT+d} = \frac{a}{c} - \frac{ad-bc}{c^2}\, \frac{1}{T+\frac{d}{c}},\]
it is not difficult to compute $\gamma(D^-(\gamma^{-1}))$ and prove that we have
\[r = \frac{|ad-bc|}{|c|^2\, r'} = \frac{|\beta|\, |\alpha-\alpha'|^2}{r'}.\]

 Set
\[\lambda := \frac{r^2}{|\beta|\, |\alpha-\alpha'|^2} = \frac{r}{r'} = \frac{|\beta|\, |\alpha-\alpha'|^2}{(r')^2}.\]
Since $D^+(\gamma)$ and $D^+(\gamma^{-1})$ are disjoint, we have $\max(r,r') < |\alpha-\alpha'|$, hence $|\beta| < \min(\lambda,\lambda^{-1})$. It follows that $D^-_{\gamma,\lambda}$ and $D^-_{\gamma^{-1},\lambda^{-1}}$ contains respectively~$\alpha'$ and~$\alpha$, hence
\[D^-_{(\gamma,\lambda)} = D^-(\alpha',r') = D^-(\gamma) \textrm{ and } D^-_{(\gamma^{-1},\lambda^{-1})} = D^-(\alpha,r) = D^-(\gamma^{-1}).\]
\end{proof}

\begin{corollary}
Let~$\Gamma$ be a Schottky group over~$k$ whose limit set does not contain~$\infty$. Then, there exists a basis $(\gamma_{1},\dotsc,\gamma_{g})$ of~$\Gamma$ and $\lambda_{1},\dotsc,\lambda_{g} \in \R_{>0}$ such that the family of discs $\big(D^+_{(\gamma_{i}^\eps,\lambda_{i}^\eps)}, 1\le i\le g, \eps \in \{\pm1\}\big)$ is a Schottky figure that is adapted to $(\gamma_{1},\dotsc,\gamma_{g})$.
\end{corollary}
\begin{proof}
By Theorem~\ref{thm:groupfigure}, there exists a basis $\beta = (\gamma_{1},\dotsc,\gamma_{g})$ of~$\Gamma$ and a Schottky figure $\calB = (D^+(\gamma_{i}^\eps), 1\le i\le g, \eps \in \{\pm1\})$ that is adapted to~$\beta$. As in Section~\ref{sec:Schottkyfigures}, define the open discs $D^-(\gamma_{i}^{\pm1}$ and set 
\[F^+ := \PP^1_{k} - \bigcup_{\substack{1\le i\le g \\ \eps=\pm1}} D^{-}(\gamma_{i}^\eps).\]
By Theorem~\ref{thm:actionproper}, since~$\infty$ is not a limit point of~$\Gamma$, there exists $\gamma \in \Gamma$ such that $\infty \in \gamma F^+$.

Set $\beta' := (\gamma \gamma_{1} \gamma^{-1},\dotsc,\gamma\gamma_{g}\gamma^{-1})$. It is a basis of~$\Gamma$ and the family of discs $\calB' := (\gamma D^+(\gamma_{i}^\eps), 1\le i\le g, \eps \in \{\pm1\})$ is a Schottky figure that is adapted to it. Since all the discs $\gamma D^+(\gamma_{i}^{\pm1})$ are contained in~$\aank$, we may now apply Lemma~\ref{lem:pairFord} to conclude.
\end{proof}

\subsection{Local fields}\label{sec:localfields}

When~$k$ is a local field, the definition of a Schottky group can be greatly simplified. Our treatment here borrows from \cite[I (1.6)]{GerritzenPut80} (see also \cite[Lemma~2.1.1]{Marden07} in the complex setting).

\begin{lemma}\label{lem:uniformconvergence}
Let $(\gamma_{n})_{n\in \N}$ be a sequence of loxodromic M\"obius transformations such that 
\begin{enumerate}
\item $(\gamma_{n})_{n\in \N}$ has no convergent subsequence in $\PGL_{2}(k)$;
\item the sequence of Koebe coordinates $((\alpha_{n},\alpha'_{n},\beta_{n}))_{n\in\N}$ converges to some $(\alpha,\alpha',\beta) \in (\PP^1(k))^3$.
\end{enumerate}
Then, $(\gamma_n)_{n\in\N}$ converges to the constant function~$\alpha$ uniformly on compact subsets of $\pank-\{\alpha'\}$.

\end{lemma}
\begin{proof}
By definition, for each $n\in \N$, we have $|\beta_{n}| <1$, which implies that $|\beta| <1$. 

Up to changing coordinates, we may assume that $\alpha,\alpha' \in k$. Up to modifying finitely many terms of the sequences, we may assume that, for each $n\in \N$, we have $\alpha_{n},\alpha'_{n} \in k$. In this case, for each $n\in \N$, we have 
\[\gamma_{n} =: \begin{bmatrix}
\alpha_{n} - \beta_{n} \alpha'_{n} & (\beta_{n}-1) \alpha_{n} \alpha'_{n}\\
1-\beta_{n} & \beta_{n} \alpha_{n} - \alpha'_{n}
\end{bmatrix}
\text{ in } \PGL_{2}(k).\]
The determinant of the above matrix is $\beta_{n}(\alpha_{n}-\alpha'_{n})^2$. Since $(\gamma_{n})_{n\in \N}$ has no convergent subsequence in $\PGL_{2}(k)$, we deduce that $\beta(\alpha-\alpha')^2 =0$. In each of the two cases $\beta=0$ and $\alpha=\alpha'$, it is not difficult to check that the claimed result holds.

%
%
\end{proof}

The result below shows that the definition of Schottky group may be simplified when $k$ is a local field. Note that, in this case, $\PP^1(k)$ is compact, hence closed in~$\pank$.

\begin{corollary}
Assume that $k$ is a local field. Let~$\Gamma$ be a subgroup of~$\PGL_{2}(k)$ all of whose non-trivial elements are loxodromic. 

Let $\Lambda$ be the set of fixed points of the elements of~$\Gamma - \{\id\}$ and let $\bar \Lambda$ be its closure in~$\pank$. Then, $\bar \Lambda$ is a compact subset of~$\pank$ that is contained in~$\PP^1(k)$ and the action of~$\Gamma$ on $\pank-\bar \Lambda$ is free and proper.

%


%
%
%
\end{corollary}
\begin{proof}
Since~$k$ is locally compact for the topology given by the absolute value, $\PP^1(k)$ is compact. By Remark~\ref{rem:topologyk}, the topology on~$k$ given by the absolute value coincides with that induced by the topology on~$\aank$. We deduce that~$\PP^1(k)$ is a compact subset of~$\pank$. It follows that $\bar \Lambda$ is contained in~$\PP^1(k)$ and that it is compact, as it is closed.

The action of~$\Gamma$ is obviously free on $\pank-\bar \Lambda$. Assume, by contradiction, that it is not proper. Then, there exist $x,y \notin \bar \Lambda$ such that, for every neighborhoods~$U$ and~$V$ of~$x$ and~$y$ respectively, the set $\{\gamma\in \Gamma \mid \gamma U \cap V \ne \emptyset\}$ is infinite. 

Since~$k$ is a local field, Corollary~\ref{cor:metrizable} ensures that the space~$\aank$ is metrizable. In particular, we may find countable bases of neighborhoods $(U_{n})_{n\in\N}$ and $(V_{n})_{n\in\N}$ of~$x$ and~$y$ respectively. By assumption, there exist a sequence $(\gamma_{n})_{n\in \N}$ of distinct elements of~$\Gamma$ and a sequence $(x_{n})_{n\in \N}$ of elements of $\pank-\bar\Lambda$ such that, for each $n\in\N$, we have $x_{n}\in U_{n}$ and $\gamma_{n}(x_{n}) \in V_{n}$. In particular, $(x_{n})_{n\in \N}$ converges to~$x$ and $(\gamma_{n}(x_{n}))_{n\in \N}$  converges to~$y$.

Since all the non-trivial elements of~$\Gamma$ are loxodromic, by the same argument as in Remark~\ref{rem:discrete}, the group~$\Gamma$ is discrete. As a result, up to passing to a subsequence, we may assume that the assumptions of Lemma~\ref{lem:uniformconvergence} are satisfied. Define~$\alpha$ and~$\alpha'$ as in this Lemma. Since~$x$ does not belong to~$\bar\Lambda$, it cannot be equal to~$\alpha'$. It follows that the sequences $(\gamma_{n}(x_{n}))_{n\in \N}$ and $(\gamma_{n}(x))_{n\in \N}$ converge to the same limit $y = \alpha$, and we get a contradiction since $\alpha\in \bar\Lambda$.

\end{proof}

\begin{corollary}
Assume that $k$ is a local field. Then, a subgroup~$\Gamma$ of~$\PGL_{2}(k)$ is a Schottky group if, and only if, it is finitely generated and all its non-trivial elements are loxodromic.
\qed
\end{corollary}

%
%

\section{Uniformization of Mumford curves}\label{sec:uniformization}

The main result of this section, Theorem \ref{thm:analyticuniformization}, states that the procedure described in Section~\ref{sec:Schottkyfigures} can be reversed: any Mumford curve may be uniformized by an open subset of the Berkovich projective line~$\pank$ with a Schottky group as group of deck transformations. The consequences of this result are many and far-reaching.
Some of them are discussed in Appendix~\ref{app:unif}.

This was first proved by D.~Mumford in his influential paper \cite{Mumford72}, where he introduces this as a non-Archimedean analogue of the uniformization of handlebodies by means of Schottky groups in the complex setting. His arguments make a heavy use of formal models of the curves. Here, we argue directly on the curves themselves, following the strategy of~\cite[Chapter IV]{GerritzenPut80} and~\cite[Proposition~4.6.6]{Lutkebohmert16}. Note, however that the proof in the first reference is flawed (since it relies on the wrong claim that every $k$-analytic curve of genus~0 embeds into~$\pank$, see Remark~\ref{rem:Liu})
and that the second reference assumes that the curve contains at least three rational points.

As an application, we discuss how Theorem \ref{thm:analyticuniformization} can be used to study the automorphism groups of Mumford curves. This is far from being the sole purpose of uniformization. Other important consequences are mentioned in Appendix \ref{app:unif}.

\subsection{The uniformization theorem}

In this section, we prove that any analytic Mumford curve as defined in \ref{def:Mumfordcurve} can be obtained as the quotient of an open dense subspace of $\pana{k}$ by the action of a Schottky group, leading to a purely analytic proof of Mumford's theorem.
We begin with a few preparatory results.

\begin{lemma}\label{lem:extensionofautomorphisms}
Let $L$ be a compact subset of~$\PP^1(k)$. Set $O := \pank - L$. 
\begin{enumerate}
\item Every bounded analytic function on~$O$ is constant.
\item Every automorphism of~$O$ is induced by an element of~$\PGL_{2}(k)$.
\end{enumerate}
\end{lemma}
\begin{proof}
(i) Let $F \in \calO(O)$. The function~$F$ is constant if, and only if, its pullback to $O_{\wka}$ is, hence we may assume that~$k$ is algebraically closed. 

Assume, by contradiction, that~$F$ is not constant. Then, there exists $x\in O$ and a branch~$b$ at~$x$ such that $F(x) \ne 0$ and $|F|$ is monomial at~$x$ along~$b$ with a positive integer exponent. We may assume that~$x$ is of type~2 or~3. Then, there exists $y \in O - \{x\}$ and $N \in \N_{\ge 1}$ such that, for each $z\in [x,y]$, we have $|F(z)| = |F(x)| \, \ell([x,z])^N$. 

Let us now consider a path~$[x,y]$, with $y \in \pank$, with the following property: for each $z \in (x,y)$, $|F|$ is monomial at~$z$ with positive integer slope along the branch in~$(x,y)$ going away from~$x$. By Zorn's lemma, we may find a maximal path~$[x,y]$ among those. 

We claim that $y$ is of type~1. If~$y$ is of type~4, then, by Theorem~\ref{thm:sigmabF}, $|F|$ is constant in the neighborhood of~$y$ in~$(x,y)$, and we get a contradiction. Assume that~$b$ is of type~2 or~3. Then, the exponent of~$|F|$ at~$y$ along the branch corresponding to~$[y,x]$ is negative. By Corollary~\ref{cor:Kirchhoff}, there exists a branch~$b$ at~$y$ such that~$|F|$ is monomial with positive exponent at~$y$ along~$b$, which contradicts the maximality. Finally, $y$~is of type~1. 

By assumption, $|F|$ has a positive integer exponent everywhere on~$(x,y)$. It follows that, for each $z \in (x,y)$, we have $|F(z)| \ge |F(x)| \, \ell([x,z])$. Since~$y$ is of type~1, by Lemma~\ref{lem:length}, we have $\ell([x,y]) = \infty$, hence~$F$ is unbounded. This is a contradiction.

\medbreak

(ii) Let~$\sigma$ be an automorphism of~$O$. 

Let us first assume that~$O$ contains at least 2 $k$-rational points. Up to changing coordinates, we may assume that $0, \infty \in O$. Let us choose an automorphism $\tau \in \PGL_{2}(k)$ that agrees with~$\sigma$ on~$0$ and~$\infty$. Then $\tau^{-1} \circ \sigma$ is an automorphism of~$O$ that fixes~0 and~$\infty$. In particular, it corresponds to an analytic function with a zero of order~1 at~0 and a pole of order~1 at~$\infty$. 
Let us consider the quotient analytic function $\varphi := (\tau^{-1} \circ \sigma)/\id$. There exist a neighborhood~$U$ of~0 and a neighborhood~$V$ of~$\infty$ on which~$\varphi$ is bounded. Since $\tau^{-1} \circ \sigma$ is an automorphism, it sends~$V$ to a neighborhood of~$\infty$, hence it is bounded on $O - V$. It follows that~$\varphi$ is bounded on $O - (U \cup V)$, hence on~$O$. By (i), we deduce that~$\varphi$ is constant, and the result follows.


Let us now handle the case where $O \cap \PP^1(k) = \emptyset$. There exists a finite extension~$K$ of~$k$ such that $O_{K}$ contains a $K$-rational point. Applying the previous argument after extending the scalars to~$K$, we deduce that $\sigma \in \PGL_{2}(K)$. Since $\sigma$ preserves $\PP^1(k)$, it actually belongs to~$\PGL_{2}(k)$. 
\end{proof}

\begin{lemma}\label{lem:compactintoP1}
Let~$Y$ be a connected $k$-analytic $\AA^1$-like curve of genus~0. 
Let $\calT = (S,\calA,\calD)$ be a triangulation of~$Y$. Assume that~$\calA$ is non-empty and consists of annuli. Let~$U$ be an open relatively compact subset of~$\Sigma_{\calT}$. Then, there exists an embedding of~$\tau_{\calT}^{-1}(U)$ into~$\pank$ such that the complement of~$\tau_{\calT}^{-1}(U)$ is a disjoint union of finitely many closed discs.
\end{lemma}
\begin{proof}
Recall that~$\Sigma_{\calT}$ is a locally finite graph (see Theorem~\ref{thm:triangulation}). As a consequence, the boundary~$\partial U$ of~$U$ in~$\Sigma_{\calT}$ is finite. For each~$z\in \partial U$, let~$I_{z}$ be an open interval in~$\Sigma_{\calT}$ having~$z$ as an end-point. Up to shrinking the~$I_{z}$'s, we may assume that they are disjoint.

Let~$z\in \partial U$. Set $A_{z} := \tau_{\calT}^{-1}(I_{z})$. Since every element of~$\calA$ is an annulus, up to shrinking~$I_{z}$ (so that it contains no points of~$S$), we may assume that~$A_{z}$ is an annulus. The open annulus~$A_{z}$ may be embedded into an open disc~$D_{z}$ such that the complement is a closed disc.

Let us construct a curve~$Y'$ by starting from~$\tau_{\calT}^{-1}(U)$ an gluing~$D_{z}$ along~$A_{z}$ for each $z\in \partial U$. By construction, the curve~$Y'$ is compact and of genus~0. Moreover, it contains rational points, as the~$D_{z}$ do. It follows from Theorems~\ref{thm:smoothcompact} and~\ref{thm:genus} that~$Y'$ is isomorphic to~$\pank$. By construction, 
\[\pank - \tau_{\calT}^{-1}(U) = \bigcup_{z\in \partial U} D_{z} - A_{z}\]
is a disjoint union of finitely many closed discs.
\end{proof}


We now state and prove the uniformization theorem.

\begin{theorem}\label{thm:analyticuniformization}
Let $X$ be a $k$-analytic Mumford curve. 
Then the fundamental group~$\Gamma$ of~$X$ is a Schottky group.
If we denote by $L$ the limit set of~$\Gamma$, then $O := \pank - L$ is a universal cover of~$X$.
In particular, we have $X \simeq \Gamma \backslash O$.
\end{theorem}
\begin{proof}
Assume that the genus of~$X$ is bigger than or equal to~2.

Let $p \colon Y \to X$ be the topological universal cover of~$X$. Since~$p$ is a local homeomorphism, we may use it to endow~$Y$ with a $k$-analytic structure. The set~$Y$ then becomes an $\AA^1$-like curve and the map~$p$ becomes a local isomorphism of locally ringed spaces. Note that the curve~$Y$ has genus~0. 

We claim that it is enough to prove that~$Y$ is isomorphic to an open subset of~$\pank$ whose complement lies in~$\PP^1(k)$. Indeed, in this case, $Y$ is simply connected, hence the fundamental group~$\Gamma$ of~$X$ may be identified with the group of deck transformations of~$p$. By Lemma~\ref{lem:extensionofautomorphisms}, it embeds into~$\PGL_{2}(k)$. It now follows from the properties of the universal cover and the fundamental group that~$\Gamma$ is a Schottky group (see Remark~\ref{rem:nonloxodromic} for the fact that the non-trivial elements of~$\Gamma$ are loxodromic). Moreover, by Theorem~\ref{thm:actionproper}, we have $Y \subseteq \pank -L$, where $L$ is the limit set of~$\Gamma$, hence $X = \Gamma\backslash Y \subset \Gamma\backslash (\pank-L)$. Since $\Gamma\backslash Y$ and $\Gamma\backslash (\pank-L)$ are both connected proper curves, they have to be equal, hence  $Y = \pank-L$.

\medbreak

In the rest of the proof, we show that~$Y$ embeds into~$\pank$ with a complement in~$\PP^1(k)$. Since~$X$ is a $k$-analytic Mumford curve of genus at least~2, it has a minimal skeleton~$\Sigma_{X}$ and the connected components of~$\Sigma_{X}$ deprived of its branch points are skeleta of open annuli over~$k$. Its preimage $p^{-1}(\Sigma_X)$ coincides with the minimal skeleton~$\Sigma_{Y}$ of~$Y$. Similarly, the connected components of~$\Sigma_{Y}$ deprived of its branch points are skeleta of open annuli over~$k$. We denote by $\tau_{Y} \colon Y \to \Sigma_{Y}$ the canonical retraction. 

\medbreak

Let~$x_{0} \in X$ and $y_{0} \in p^{-1}(x_{0})$. Let~$\ell_{X}$ be a loop in~$\Sigma_{X}$ based at~$x_{0}$ that is not homotopic to~0. It lifts to a path in~$\Sigma_{Y}$ between~$y_{0}$ and a point~$y_{1}$ of $p^{-1}(x_{0})$. We may then lift again~$\ell_{X}$ to a path in~$\Sigma_{Y}$ between~$y_{1}$ and a point~$y_{2}$ of $p^{-1}(x_{0})$. Repeating the procedure, we obtain a non-relatively compact path~$\lambda(\ell_{X})$ in~$\Sigma_{Y}$ starting at~$y_{0}$. Note that the length of~$\lambda(\ell_{X})$ is infinite since it contains infinitely many copies of~$\ell_{X}$.

More generally, all the maximal paths starting from~$y_{0}$ in~$\Sigma_{Y}$ are of infinite length, since they contain infinitely many lifts of loops from~$\Sigma_{X}$.

Since~$X$ is of genus at least~2, we may find two loops~$\ell_{X,0}$ and~$\ell_{X,1}$ based at~$x_{0}$ in~$\Gamma_{X}$ that are not homotopic to~0 and not homotopic one to the other. Set $\ell_{0} := \lambda(\ell_{X,0})$, $\ell_{\infty} := \lambda(\ell_{X,0}^{-1})$ and $\ell_{1} := \lambda(\ell_{X,1})$. Away from some compact set of~$Y$, the three paths $\ell_{0},\ell_{\infty}, \ell_{1}$ are disjoint. Up to moving~$x_{0}$ and~$y_{0}$, we may assume that 
\[ \ell_{0} \cap \ell_{1} = \ell_{\infty} \cap \ell_{1} = \ell_{0} \cap \ell_{\infty} = \{y_{0}\}.\]
For $i\in \{0,1,\infty\}$ and $r\in \R_{\ge 1}$, we denote by~$\xi_{i,r}$ the unique point of~$\ell_{i}$ such that $\ell([y_{0},\xi_{i,r}]) = r$.

%

\medbreak

Let $n\in \N_{\ge 1}$. Set 
\[ U_{n} := \{ z \in \Sigma_{Y} : \ell([y_{0},z]) < 2^{n}\} \textrm{ and } Y_{n} := \tau_{Y}^{-1
}(U_{n}).\]
We already saw that all the maximal paths starting from~$y_{0}$ in~$\Sigma_{Y}$ are of infinite length, hence~$U_{n}$ is relatively compact in~$\Sigma_{Y}$. Denote by~$\partial U_{n}$ the boundary of~$U_{n}$ in~$\Sigma_{Y}$. For each $z\in \partial U_{n}$, we have $\ell([y_{0},z]) = 2^n$. 

By Lemma~\ref{lem:compactintoP1}, there exists an open subset~$O_{n}$ of~$\pank$ and an isomorphism $\varphi_{n} \colon Y_{n} \xrightarrow[]{\sim} O_{n}$ such that $\pank - O_{n}$ is a disjoint union of closed discs. For each $z\in \partial U_{n}$, we denote by~$p_{z}$ the end-point of~$\varphi_{n}([y_{0},z))$ in $\pank - O_{n}$ and by~$D_{z}$ the connected component of $\pank - O_{n}$ whose boundary point is~$p_{z}$. To ease the notation, for $i \in \{0,1,\infty\}$, we set $D_{i,n} := D_{\xi_{i,2^n}}$. 

Let us fix a ($k$-rational) point at infinity on~$\pank$ and a coordinate~$T$ on $\aank \subset \pank$. We may assume that, for each $i \in \{0,1,\infty\}$, we have $i \in D_{i,n}$. By pulling-back the analytic function~$T$ on~$O_{n}$ by~$\varphi_{n}$, we get a analytic function on~$Y_{n}$. We denote it by~$\psi_{n}$. Recall that it is actually equivalent to give oneself~$\varphi_{n}$ or~$\psi_{n}$, see Lemma~\ref{lem:morphismtoA1}. 


\begin{lemma}\label{lem:varphin}
We have $\varphi_{n}(y_{0}) = \eta_{1}$. For each $r\in [1,2^n)$, we have
\[\varphi_{n}(\xi_{0,r}) = \eta_{1/r},\ \varphi_{n}(\xi_{\infty,r}) = \eta_{r} \textrm{ and } \varphi_{n}(\xi_{1,r}) = \eta_{1,1/r}.\]

Let~$C$ be a connected component~$C$ of $Y - (\ell_{0} \cup \ell_{\infty})$. For each $y \in C\cap Y_{n}$, we have 
\[|\psi_{n}(y)| = 
\begin{cases}
1/r &\textrm{ if the boundary point of } C \textrm{ is } \xi_{0,r};\\
r &\textrm{ if the boundary point of } C \textrm{ is } \xi_{\infty,r}.
\end{cases}
\]

Let $N \in \cn{1}{n}$. The image $\varphi_{n}(Y_{N})$ is an open Swiss cheese. More precisely, there exist $d \in \N_{\ge 2}$, $\alpha_{2},\dotsc,\alpha_{d} \in k^\ast$ and, for each $j\in \cn{2}{d}$, $r_{j}\in [2^{-N},|\alpha_{j}|)$ such that $\varphi_{n}(Y_{N})$ is the subset of~$\aank$ defined by the following conditions:
\[\begin{cases}
2^{-N} < |T| < 2^N;\\ 
|T-1| > 2^{-N};\\
\forall j \in \cn{2}{d},\ |T-\alpha_{j}| > r_{j}.
\end{cases}\]
\end{lemma}
\begin{proof}
It follows from the construction that, for each  $i \in \{0,1,\infty\}$, $\varphi_{n}([y_{0},\xi_{i,2^n})$ is an injective path joining~$\varphi_{n}(y_{0})$ to the boundary point of a disc centered at~$i$. Since those paths only meet at~$\varphi_{n}(y_{0})$, the only possibility is that $\varphi_{n}(y_{0}) = \{\eta_{1}\}$.

Let $r\in [1,2^n)$. 
Since lengths are preserved by automorphism (see Proposition~\ref{prop:isoannulus}), for each $i \in \{0,1,\infty\}$, we have $\ell([\eta_{1},\varphi_{n}(\xi_{i,r})]) = r$. Since $\varphi_{n}(\xi_{\infty,r})$ belongs to $[\eta_{1},\infty]$, it follows that $\varphi_{n}(\xi_{\infty,r}) = \eta_{r}$. By a similar argument, we have $\varphi_{n}(\xi_{0,r}) = \eta_{1/r}$ and $\varphi_{n}(\xi_{1,r}) = \eta_{1,1/r}$.

\medbreak
 

Recall that we have $I_{0} = \{\eta_{r} : r\in \R_{\ge 0}\} \subset \aank$. Let~$C$ be a connected component of~$\aank - I_{0}$ and let~$\eta_{r}$ be its boundary point. Then, for each $z \in C$, we have $|T(z)| = r$. 

We have $\varphi_{n}^{-1}(I_{0} \cap O_{n}) = (\ell_{0} \cup \ell_{\infty}) \cap Y_{n}$. By definition of~$\psi_{n}$, for each $y\in Y_{n}$, we have
$ |\psi_{n}(y)| = |T(\varphi_{n}(y))| $.
It follows that, for each connected component~$C$ of $Y - (\ell_{0} \cup \ell_{\infty})$ and each $y \in C\cap Y_{n}$, we have 
\[|\psi_{n}(y)| = 
\begin{cases}
1/r &\textrm{ if the boundary point of } C \textrm{ is } \xi_{0,r};\\
r &\textrm{ if the boundary point of } C \textrm{ is } \xi_{\infty,r}.
\end{cases}
\]

\medbreak

The set~$O_{n}$ is an open Swiss cheese. The set $\varphi_{n}(U_{N})$ is a connected open subset of its skeleton and $\varphi_{n}(Y_{N})$ is the preimage of it by the retraction. It follows that $\varphi_{n}(Y_{N})$ is an open Swiss cheese too, hence the complement in~$\pank$ of finitely many closed discs $E_{\infty},E_{0},\dotsc,E_{d}$. Let $z_{\infty}, z_{0}, \dotsc,z_{d}$ denote the corresponding boundary points. The set $\varphi_{n}(Y_{N})$ contains $\varphi_{n}(y_{0}) = \eta_{1}$ and, by construction of~$Y_{N}$, for each $i\in \{\infty\} \cup \cn{0}{d}$, we have $\ell([\eta_{1},z_{i}]) = 2^N$.

Since $0$, $1$ and~$\infty$ do not belong to~$O_{n}$, some of those discs~$E_{i}$ contain those points. Since $\varphi_{n}(Y_{N})$ contains~$\eta_{1}$, those discs are disjoint. We may assume that, for each $i \in \{0,1,\infty\}$, we have $i \in E_{i}$. The length property then implies that we have $z_{\infty} = \eta_{2^N}$, $z_{0} = \eta_{2^{-N}}$ and $z_{0} = \eta_{1,2^{-N}}$. In other words,
\[ \pank - (E_{\infty} \cup E_{0} \cup E_{1}) = \{x \in \aank :  2^{-N} < |T(x)| < 2^N, \ |T-1| > 2^{-N}\}.\]

For $j \in \cn{2}{d}$, let~$\alpha_{j}$ be a $k$-rational point of~$E_{j}$. The boundary point~$z_{j}$ of~$E_{j}$ is then of the form~$\eta_{\alpha_{j},r_{j}}$ for some $r_{j} \in \R_{\ge 0}$. Since~$E_{j}$ does not contain~0, we have $r_{j} < |\alpha_{j}|$. Moreover, the condition $\ell([\eta_{1},\eta_{\alpha_{j},r_{j}}]) = 2^{N}$ implies that $r_{j} \ge 2^{-N}$ (see Example~\ref{ex:leta}). The result follows.
\end{proof}

Let $N,n,m \in\N_{\ge 1}$ with $n\ge m> N$. The analytic function~$\psi_{m}$ has no zeros on~$Y_{m}$, hence the quotient $\psi_{n|Y_{m}}/(\psi_{m})$ defines an analytic function on~$Y_{m}$. Set 
\[h_{n,m} := \frac{\psi_{n|Y_{m}}}{\psi_{m}} -1 \in \calO(Y_{m}).\]

\begin{lemma}\label{lem:hnm}
For $N,n,m \in\N_{\ge 1}$ with $n\ge m> N$, we have $\|h_{n,m}\|_{Y_{N}} \le \max(2^{N-m},2^{-m/2})$.
\end{lemma}
\begin{proof}
By Lemma~\ref{lem:varphin}, for each $y\in Y_{m}$, we have $|\psi_{n}(y)| = |\psi_{m}(y)|$. It follows that $\|h_{m,n}\|_{Y_{m}} \le 1$. We now distinguish two cases.

\smallbreak

$\bullet$ Assume that $|h_{n,m}|$ is not constant on $Y_{N}$.

By Corollary~\ref{cor:maximumprinciple}, there exists $y \in \partial Y_{N}$ such that $\|h_{n,m}\|_{Y_{N}} = |h_{n,m}(y)|$ and $|h_{n,m}|$ has a negative exponent at~$y$ along the branch entering~$Y_{N}$. By harmonicity (see Theorem~\ref{thm:harmonicityanalytic}), there exist a branch~$b$ at~$y$ not belonging to~$Y_{N}$ such that the exponent of~$|h_{n,m}|$ along~$b$ is positive. Repeating the procedure, we construct a path joining~$y$ to a boundary point~$y'$ of~$Y_{m}$ such that $|h_{n,m}|$ has a positive exponent at each point of $[y,y')$ along the branch pointing towards~$y'$. It follows that we have
\[\|h_{n,m}\|_{Y_{n}} \ge |h_{n,m}(y)| \, \ell([y,y']) \ge \|h_{n,m}\|_{Y_{N}}\, 2^{m-N},\]
hence
\[ \|h_{n,m}\|_{Y_{N}} \le 2^{N-m}.\]

\smallbreak

$\bullet$ Assume that $|h_{n,m}|$ is constant on $Y_{N}$.

Let~$N'$ be the maximum integer smaller than or equal to~$m$ such that $|h_{n,m}|$ is constant on~$Y_{N'}$. Then, for every $r \in [1,2^{N'})$, we have $|h_{n,m}(\xi_{1,r})| = \|h_{n,m}\|_{Y_{N'}}$. We also have
\begin{align*} 
|h_{n,m}(\xi_{1,r})|  & = \frac{|(\psi_{n} - \psi_{m})(\xi_{1,r})|}{|\psi_{m}(\xi_{1,r})|}\\
& =  \frac{|(\psi_{n} - \psi_{m})(\xi_{1,r})|}{|T(\eta_{1,1/r})|}\\
& \le \max (|(\psi_{n}-1)(\xi_{1,r})|, |(\psi_{m}-1)(\xi_{1,r})| )\\
&\le |(T-1)(\eta_{1,1/r})|\\ 
&\le \frac1r. 
\end{align*}
We deduce that $\|h_{n,m}\|_{Y_{N'}} \le 2^{-N'}$.
 
If $N' < m$, it follows from the previous case that we have $\|h_{n,m}\|_{Y_{N'}}  \le 2^{N'-m}$. 

In any case, we have 
\[\|h_{n,m}\|_{Y_{N}} \le 2^{- m/2}.\]
\end{proof}

\medbreak

It follows from Lemma~\ref{lem:hnm} that the sequence $(\psi_{n})_{n > N}$ converges uniformly on~$Y_{N}$. Let~$\psi^{(N)}$ be its limit. It is an analytic function on~$Y_{N}$.

The functions~$\psi^{(N)}$ are compatible, by uniqueness of the limit, which gives rise to an analytic function $\psi \in \calO(Y)$. By Lemma~\ref{lem:morphismtoA1}, there exists a unique analytic morphism $\varphi \colon Y \to \aank$ such that the pull-back of~$T$ by~$\varphi$ is~$\psi$.

\medbreak

Let $N\in \N_{\ge 1}$.
By Lemma~\ref{lem:hnm}, there exists $m > N$ such that, for each $n\ge m$, we have $\|h_{n,m}\|_{Y_{N}} \le 2^{-2N}$. (For instance, one could choose $m=4N$.)
By Lemma~\ref{lem:varphin}, we have $\|\psi_{m}\|_{Y_{N}} = \|T\|_{\varphi_{m}(Y_{N})} = 2^N$. It follows that 
$\|\psi_{n} - \psi_{m}\|_{Y_{N}} \le \|\psi_{m}\|_{Y_{N}} \, \|h_{n,m}\|_{Y_{N}} \le 2^{-N}$.
By passing to the limit over~$n$, we deduce that 
\[\|\psi - \psi_{m}\|_{Y_{N}} \le 2^{-N}.\]

\begin{lemma}\label{lem:isoYN}
We have $\varphi(Y_{N}) = \varphi_{m}(Y_{N})$ and $\varphi_{|Y_{N}}$ is an isomorphism onto its image.
\end{lemma}
\begin{proof}
By Lemma~\ref{lem:varphin}, there exist $d \in \N_{\ge 2}$, $\alpha_{2},\dotsc,\alpha_{d} \in k^\ast$ and, for each $j\in \cn{2}{d}$, $r_{j}\in [2^{-N},|\alpha_{j}|)$ such that $\varphi_{m}(Y_{N})$ is the subset of~$\aank$ defined by 
\[\begin{cases}
2^{-N} < |T| < 2^N;\\ 
|T-1| > 2^{-N};\\
\forall j \in \cn{2}{d},\ |T-\alpha_{j}| > r_{j}.
\end{cases}\]
For $t \in (1,2^N)$, let~$W_{t}$ be the subset of~$\aank$ defined by 
\[\begin{cases}
2^{-N}t \le |T| \le 2^Nt^{-1};\\ 
|T-1| \ge 2^{-N}t;\\
\forall j \in \cn{2}{d},\ |T-\alpha_{j}| \ge r_{j}\, t.
\end{cases}\]
Each~$W_{t}$ is compact and the family $(W_{t})_{t \in (1,2^N)}$ is an exhaustion of~$\varphi_{m}(Y_{N})$. 

Let $n\ge m$. For $t \in (1,2^N)$, the set $\varphi^{-1}(W_{t}) \cap Y_{N}$ is the subset of points $y \in Y_{N}$ such that
\[\begin{cases}
2^{-N}t \le |\psi(y)| \le 2^Nt^{-1};\\ 
|\psi(y)-1| \ge 2^{-N}t;\\
\forall j \in \cn{2}{d},\ |\psi(y)-\alpha_{j}| \ge r_{j}\, t.
\end{cases}\]
From the inequality $\|\psi - \psi_{m}\|_{Y_{N}} \le 2^{-N}$, we deduce that $\varphi^{-1}(W_{t}) \cap Y_{N} = \varphi_{m}^{-1}(W_{t}) \cap Y_{N}$. 

It follows that $\varphi(Y_{N}) = \varphi_{m}(Y_{N})$ and that the morphism $\varphi_{|Y_{N}} \colon Y_{N} \to \varphi(Y_{N})$ is proper. Since~$Y_{N}$ is a smooth curve and~$\varphi_{|Y_{N}}$ is not constant, it is actually finite.

To prove that $\varphi_{|Y_{N}}$ is an isomorphism, it is enough to show that it is of degree~1. We will prove that, for each $r \in [1,2^N)$, we have $\varphi_{{|Y_{N}}}^{-1}(\xi_{\infty,r}) = \{\eta_{r}\}$. This implies the result, by Theorem~\ref{th:degreelength}.

Let $r \in [1,2^N)$. Let $y\in Y_{N}$ such that $\varphi(y) = \eta_{r}$. To prove that $y = \xi_{\infty,r}$, we may extend the scalars to~$\wka$. The point~$\eta_{r}$ of~$\aank$ is characterized by the following equalities:
\[ \begin{cases}
|T(\eta_{r})| =r;\\[3pt] 
\forall \alpha \in \wka \textrm{ with } |\alpha| = r,\  |(T-\alpha)(\eta_{r})| = r.
\end{cases}\]
Since $\varphi(y) = \eta_{r}$, we have 
\[ \begin{cases}
|\psi(y)| =r;\\[3pt] 
\forall \alpha \in \wka \textrm{ with } |\alpha| = r,\  |\psi(y)-\alpha| = r.
\end{cases}\]
Since $\|\psi-\psi_{m}\|_{Y_{N}} \le 2^{-N} < r$, the same equalities hold with~$\psi_{m}$ instead of~$\psi$. It follows that~$\psi_{m}(y) = \eta_{r}$, hence $y = \xi_{\infty,r}$ since $\psi_{m}$ is injective.
\end{proof}

It follows from Lemmas~\ref{lem:varphin} and~\ref{lem:isoYN} that, for each $N\in \N_{\ge 1}$, $\pank - \varphi(Y_{N})$ is a disjoint union of closed discs with radii smaller than or equal to~$2^{-N}$. It follows that  
\[ \pank - \varphi(Y) = \bigcap_{N\ge 1} \pank - \varphi(Y_{N})\]
is a compact subset of~$\PP^1(k)$ (see the proof of Corollary~\ref{cor:intersectiondiscs} for details on $k$-rationality). By Lemma~\ref{lem:isoYN} again, $\varphi$ induces an isomorphism onto its image.

\medbreak

We briefly sketch how the proof needs to be modified to handle the case of genus~0 and~1. One may use similar arguments but the paths $\ell_{0},\ell_{\infty},\ell_{1}$ have to be constructed in a different way. In genus~0, one first proves that~$X$ has rational points and consider paths joining~$y_{0}$ to them. (In this case, one may also argue more directly to prove that~$X$ is isomorphic to~$\pank$ by Theorems~\ref{thm:smoothcompact} and~\ref{thm:genus}.) In genus~1, the skeleton provides two paths and we can use a rational point to construct the third one. Such a point has to exist, since any annulus over~$k$ whose skeleton is of large enough length contains some.
\end{proof}

\begin{remark}\label{rem:Liu}
The most difficult part of the proof of Theorem~\ref{thm:analyticuniformization} consists in proving that the $k$-analytic curve~$Y$, which is known to be of genus~0, may be embedded into~$\pank$. Contrary to what happens over the field of complex numbers, this is not automatic. This problem was studied extensively by Q.~Liu under the assumption that $k$ is algebraically closed. He proved that the answer depends crucially on the maximal completeness of~$k$. If it holds, then any smooth connected $k$-analytic curve of finite genus may be embedded into the analytification of an algebraic curve of the same genus (hence into~$\pank$ in the genus~0 case), see \cite[Th\'eor\`eme~3]{TheseLiu} or \cite[Th\'eor\`eme~3.2]{Liu87}. Otherwise, there exists a smooth connected $k$-analytic curve of genus~0 with no embedding into~$\pank$, see \cite[Proposition~5.5]{TheseLiu}. Q.~Liu also prove several other positive results that hold over an algebraically closed base field.

The results of Q.~Liu are stated and proved in the language of rigid analytic geometry. We believe that it is worth adapting them to the setting of Berkovich geometry and that this could lead to a different point of view on the sufficient conditions for algebraizablity. One may also wonder whether it is necessary to assume that the base field is algebraically closed to obtain an unconditional positive result. The case of a discretely valued base field (hence maximally complete but not algebraically closed) is, of course, particularly interesting.
\end{remark}

\subsection{Automorphisms of Mumford curves}\label{sec:automorphisms}

In this section, we use the uniformization of Mumford curves to study their groups of $k$-linear automorphisms. 
The fundamental result, proven by Mumford in \cite[Corollary 4.12]{Mumford72} is the following theorem.
We include a proof of this fact that relies on the topology of Berkovich curves.

\begin{theorem}\label{thm:automorphismgroup}
Let $X$ be a $k$-analytic Mumford curve. Let $\Gamma \subset \PGL_2(k)$ be its fundamental group, and let $N := N_{\PGL_2(k)}(\Gamma)$ be the normalizer of $\Gamma$ in $\PGL_2(k)$.
Then, we have \[\Aut(X) \cong N/\Gamma .\]
\end{theorem}
\begin{proof}
Let $p \colon O \to X$ be the universal cover of $X$ provided by Theorem \ref{thm:analyticuniformization}, and let $\sigma\in\Aut(X)$.
Since~$p$ is locally an isomorphism of $k$-analytic curves, the automorphism $\sigma$ can be lifted to an analytic automorphism $\widetilde{\sigma} \in \Aut(O)$ such that $p\circ \widetilde{\sigma}=\sigma\circ p$.
By Lemma \ref{lem:extensionofautomorphisms}, $\widetilde{\sigma}$ extends uniquely to an automorphism of $\pana{k}$, that is, an element $\tau \in \PGL_2(k)$.
The automorphism $\tau$ has to normalize~$\Gamma$: in fact, for any $\gamma\in \Gamma$, the element $\tau \gamma \tau^{-1} \in \Aut(O)$ induces the automorphism $\sigma \sigma^{-1} = \id$ on~$X$. It follows that $\tau \gamma \tau^{-1}\in \Gamma$, so that $\tau\in N$.

Conversely, let $\tau \in N$. By definition, the limit set~$L$ of~$\Gamma$ is preserved by~$\tau$. It follows that~$\tau$ induces an automorphism of $O = \pank-L$. Moreover, for each $\gamma\in \Gamma$ and each $x \in \pank$, we have 
\[\tau (\gamma(x)) = (\tau \gamma \tau^{-1})(\tau(x)) \in \Gamma \cdot \tau(x).\]
It follows that~$\tau$ descends to an automorphism of $X \simeq \Gamma \backslash O$.
\end{proof}

As it was the case for the uniformization, Mumford's proof relies on non-trivial results in formal geometry. 
The Berkovich analytic proof turns out to be shorter and much less technical due to the fact that the uniformization of a Mumford curve can be interpreted as a universal cover of analytic spaces.

\smallskip
Recall from Remark \ref{rem:skeletonX} that the skeleton $\Sigma_X$ of the Mumford curve $X$ is a finite metric graph.
We will denote by $\Aut(\Sigma_X)$ the group of isometric automorphisms of~$\Sigma_X$.
An interesting feature of the automorphism group of an analytic curve that is immediate in the Berkovich setting, is the existence of a \emph{restriction homomorphism}
\[\begin{array}{rcl}
 \rho: \Aut(X)& \too & \Aut(\Sigma_X)\\
 \sigma & \longmapsto & \sigma_{|\Sigma_X}.
\end{array}\]

\begin{proposition}\label{prop:AutomorphismRestriction}
Let $X$ be a Mumford curve of genus at least 2.
Then, the restriction homomorphism $\rho: \Aut(X) \to  \Aut(\Sigma_X)$ is injective.
\end{proposition}
\begin{proof}
Let $\sigma\in \Aut(X)$ such that $\rho(\sigma)=\id$, that is, $\sigma$ acts trivially on the skeleton $\Sigma_X$.
Then, as in the proof of Theorem~\ref{thm:automorphismgroup} one can lift $\sigma$ to an automorphism of the universal cover $p:O \longrightarrow X$.
By possibly composing this lifting with an element of the Schottky group, we can find a lifting $\widetilde{\sigma}$ that fixes a point $x$ in the preimage $p^{-1}(\Sigma_X) \subset O$.
Since $\sigma$ fixes $\Sigma_X$ pointwise, then $\widetilde{\sigma}$ fixes the fundamental domain in $p^{-1}(\Sigma_X)$ by the action of the Schottky group $\Gamma_X$ containing $x$.
By continuity of the action of $\Gamma_X$ on $p^{-1}(\Sigma_X)$, the automorphism $\widetilde{\sigma}$ has to fix the whole $p^{-1}(\Sigma_X)$ pointwise.
But then the corresponding element $\tau\in \PGL_2(k)$ obtained by extending $\widetilde \sigma$ thanks to Lemma~\ref{lem:extensionofautomorphisms}(ii) has to fix the limit set of $\Gamma_X$, which is infinite when $g(X) \geq 2$.
It follows that $\tau$ is the identity of $\PGL_2(k)$, hence that $\sigma$ is the identify automorphism.
\end{proof}

\begin{remark}
The previous proposition can be proved also using algebraic methods as follows. 
The fact that $g(X)\geq 2$ implies that $\Aut(X)$ is a finite group.
Then, for every $\sigma\in \Aut(X)$, $Y := X/\langle \sigma \rangle$ makes sense as a $k$-analytic curve, and the quotient map $f_\sigma:X \to Y$ is a ramified covering. 
Let us now suppose that $\rho(\sigma)=\id$. Then $Y$ contains an isometric image of the graph~$\Sigma_{X}$, whose cyclomatic number is~$g(X)$, by Corollary~\ref{cor:Mumfordgenus}. It follows from the definition of the genus that $g(Y) \ge g(X)$.
We can now apply Riemann-Hurwitz formula to find that 
\[2g(X)-2 = \deg(f_\sigma) (2g(Y)-2) + R,\]
where $R$ is a positive quantity.
Since $g(Y) \ge g(X)\geq 2$, we deduce that $\deg(f_\sigma)=1$, hence $\sigma=\id$.
\end{remark}

The proposition shows that $\Aut(\Sigma_X)$ controls $\Aut(X)$, but it is a very coarse bound when the genus is high.
Much better bounds are known, as one can see in the examples below and in the first part of Appendix \ref{app:unif}, containing an outline of further results about automorphisms of Mumford curves, including the case of positive characteristic.

\begin{example}
Let $X$ be a Mumford curve such that $\Aut(\Sigma_X) = \{1\}$.
Then Proposition~\ref{prop:AutomorphismRestriction} ensures that $X$ has no non-trivial automorphisms as well.
Since up to replacing $k$ with a suitable field extension every stable metric graph can be realized as the skeleton of a Mumford curve, one can build in this way plenty of examples of Mumford curves without automorphisms.
For example, the graph of genus 3 in Figure~\ref{fig:graph3} below has a trivial automorphism group, as long as the edge lengths are generic enough, for example when all lengths are different.

\begin{figure}[h]
\centering
\begin{tikzpicture}[
	xscale=.9,
  yscale=.7,
  level distance=3cm,
  line width=8pt,
]
\coordinate (a) at (0,0);
\coordinate (b) at (-2.5,-5);
\coordinate (c) at (2.5,-5);
\draw[thin] (a) to[out=-160,in=80] (b);
\draw[thin] (a) to[out=-105,in=40] (b);
\draw[thin] (b) to[out=15,in=165] (c);
\draw[thin] (a) to[out=-75,in=140] (c);
\draw[thin] (a) to[out=-20,in=100] (c);
\fill (a) circle[radius=2pt];
\node[anchor=south] at (a) {$v_1$};
\fill (b) circle[radius=2pt];
\node[anchor=east] at (b) {$v_2$};
\fill (c) circle[radius=2pt];
\node[anchor=west] at (c) {$v_3$};
\end{tikzpicture}
\caption{The metric graph $\Sigma_X$ has trivial group of automorphisms if the edge lengths are all different.}
\label{fig:graph3}
\end{figure}
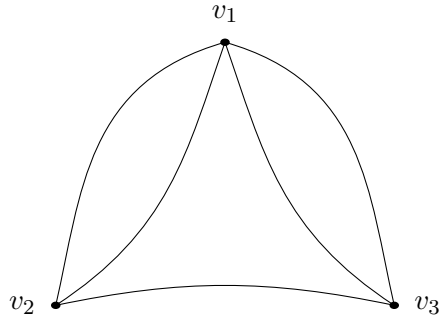

This graph can be obtained by pairwise identifying the ends of a tree as in Figure~\ref{fig:tree3}.

\begin{figure}[h]
\centering
\begin{tikzpicture}[
  scale=.8,
  level distance=3cm,
  line width=8pt,
]
\coordinate (a) at (0,0);
\coordinate (b) at (-1.75,-2);
\coordinate (c) at (1.75,-2);
\draw[thin] (a) -- (b);
\draw[thin] (a) -- (c);
\draw[thin] (a) -- (-3,-.5);
\draw[thin] (a) -- (3,-.5);
\draw[thin] (b) -- (-4.5,-4);
\draw[thin] (b) -- (-1,-4);
\draw[thin] (c) -- (1,-4);
\draw[thin] (c) -- (4.5,-4);
\fill (a) circle[radius=2pt];
\node[anchor=south] at (a) {$p^{-1}(v_1)$};
\fill (b) circle[radius=2pt];
\node[anchor=east] at (b) {$p^{-1}(v_2)$};
\fill (c) circle[radius=2pt];
\node[anchor=west] at (c) {$p^{-1}(v_3)$};
\node[circle, draw, ultra thin] at (-3,-.5) {};
\node[circle, draw, ultra thin] at (-4.5,-4) {};
\node[regular polygon,regular polygon sides=3, draw, ultra thin, scale=0.6] at (-1,-4) {};
\node[regular polygon,regular polygon sides=3, draw, ultra thin, scale=0.6] at (1,-4) {};
\node[draw, ultra thin] at (3,-.5) {};
\node[draw, ultra thin] at (4.5,-4) {};
\end{tikzpicture}
\caption{The graph in the previous figure is obtained from its universal covering tree by pairwise gluing the ends of the finite sub-tree $\Sigma_F$. The gluing is made by identifying the ends that are marked with the same shape.}
\label{fig:tree3}
\end{figure}
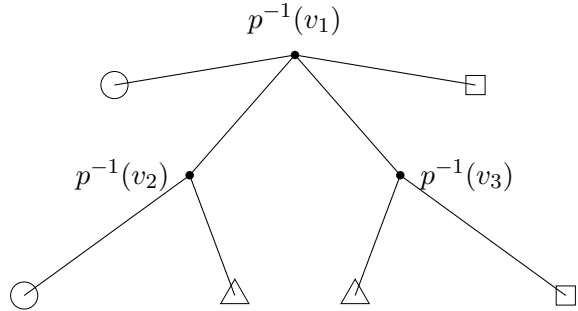

One can realize this tree inside $\pank$ as the skeleton of a fundamental domain under the action of a Schottky group in many ways.
As an example, if $k=\Q_p$ with $p\geq 5$, a suitable Schottky group is obtained by carefully choosing the Koebe coordinates that give rise to the desired skeleton.
One can for instance pick $\Gamma=\langle M(0,\infty,p^3), M(1,2,p^4), M(p,p-2,p^3) \rangle$ and verify that it gives rise to a fundamental domain whose skeleton is the tree in Figure~\ref{fig:tree3}.
As a consequence of Theorem \ref{thm:automorphismgroup}, the normalizer of $\Gamma$ in $PGL_2(k)$ is the group $\Gamma$ itself.
\end{example}

\begin{example}\label{ex:hyperelliptic}
Assume that $k$ is algebraically closed and that its residue characteristic is different from~2 and~3. Let $\pi, \rho$ be elements of~$k$ satisfying $|\pi|<1$, $\rho^3=1$ and $\rho \neq 1$. Fix the following elements of $\PGL_2(k)$:
\[a= \begin{bmatrix}
-\pi & 0 \\ -2 & \pi
\end{bmatrix}, \; b= \begin{bmatrix}
1+\pi - \rho & (1+\pi)(\rho-1) \\ 1 - \rho & (1+\pi) \rho - 1
\end{bmatrix}. \]
These elements are of finite order, respectively two and three.
The fixed rigid points of $a$ are $0$ and $\pi$, while the fixed rigid points of $b$ are $1$ and $1+\pi$.

Thanks to our assumption that $\Char(\widetilde{k})\ne2$, the transformation $a$ acting on $\pank$ fixes the path joining $0$ and $\pi$, and sends every open disc whose boundary point lies on this path to a disjoint open disc with the same boundary point.
For example, the image by $a$ of the disc $D^-(-\frac{\pi}{\pi-2},1)$ is the disc\footnote{Recall that on the projective line we consider also discs ``centered in $\infty$'' such as this one.}~$\pana{k} - D^+\big(0, |\pi|\big)$, and vice versa.

The same happens for the action of $b$: the path joining $1$ and $1+\pi$ is fixed, while any open disc with its boundary point on this path is sent to a disjoint open disc with the same boundary point.
Since $b$ is of order three, the orbit of such a disc consists of three disjoint discs.
For example, the orbit of $D^-(0,1)$ contains $b\big(D^-(0,1)\big) = D^-\big (1-\frac{\pi}{(1+\pi)\rho-1}, 1 \big)$ and $b^2\big(D^-(0,1)\big) = D^-\big (1-\frac{\pi}{(1+\pi)\rho^2-1}, 1 \big)$.

Let us consider the elements $\gamma_1 :=abab^2$ and $\gamma_2 :=ab^2ab$. Using the geometry of $a$ and $b$ described above, one can check that the 4-tuple $\big( D^+(\gamma_1), D^+(\gamma_1^{-1}), D^+(\gamma_2), D^+(\gamma_2^{-1})\big)$
 represented in Figure~\ref{fig:SchottkyFigureEx} provides a Schottky figure adapted to $(\gamma_1,\gamma_2)$.


\begin{figure}[h]
\centering
\begin{tikzpicture}[
  scale=.8,
  level distance=3cm]
\coordinate (o) at (1,1);
\coordinate (p) at (1,-4);
\coordinate (au) at (-5,-1.5);
\coordinate (gu) at (-4,1);
\coordinate (gupp) at (-5,0);
\coordinate (ggu) at (-4,-4);
\coordinate (ggupp) at (-5,-3);
\coordinate (u) at (2,-3);
\coordinate (up) at (7,-1.5);
\coordinate (bo) at (5,-4);
\coordinate (babi) at (6,-3);
\coordinate (bio) at (5,1);
\coordinate (biab) at (6,0);
\draw[thin] (au) -- (-1,-1.5) -- (3, -1.5);
\draw[thin] (-3.5,0) -- (-1,-1.5) -- (-3.5,-3) ;
\draw[thin] (gu) -- (-3.5,0) -- (gupp);
\draw[thin] (ggu) -- (-3.5,-3) -- (ggupp);
\draw[thick] (o) -- (1,-1.5) -- (p);
\draw[thick] (u) -- (3, -1.5) -- (up);
\draw[thin] (bo) -- (3, -1.5) -- (bio);
\draw[thin] (babi) -- (4.2, -3);
\draw[thin] (biab) -- (4.2, 0);
\fill (o) circle[radius=2pt];
\fill (bo) circle[radius=2pt];
\fill (babi) circle[radius=2pt];
\fill (bio) circle[radius=2pt];
\fill (biab) circle[radius=2pt];
\fill (p) circle[radius=2pt];
\fill (u) circle[radius=2pt];
\fill (au) circle[radius=2pt];
\fill (gu) circle[radius=2pt];
\fill (gupp) circle[radius=2pt];
\fill (ggu) circle[radius=2pt];
\fill (ggupp) circle[radius=2pt];
\fill (up) circle[radius=2pt];
\fill (-1,-1.5) circle[radius=2pt];
\fill (3, -1.5) circle[radius=2pt];
\node[anchor=east] at (o) {$0$};
\node[anchor=north] at (bo) {$\gamma_1^{-1}(\pi)$};
\node[anchor=west] at ($(babi)+(0.1,0)$) {$\gamma_1^{-1}(0)$};
\node[anchor=south] at ($(bio)+(0,0.1)$) {$\gamma_2^{-1}(0)$};
\node[anchor=west] at ($(biab)+(0.1,0)$) {$\gamma_2^{-1}(\pi)$};
\node[anchor=north] at (p) {$\pi$};
\node[anchor=north] at (u) {$1$};
\node[anchor=north] at (up) {$1+\pi$};
\node[anchor=east] at (au) {$-\frac{\pi}{\pi-2}$};
\node[anchor=south] at ($(gu)+(0,0.1)$) {$\gamma_2(1)$};
\node[anchor=east] at ($(gupp)-(0.1,0)$) {$\gamma_2(1+\pi)$};
\node[anchor=north] at ($(ggu)-(0,0.2)$) {$\gamma_1(1)$};
\node[anchor=east] at ($(ggupp)- (0.1,0)$) {$\gamma_1(1+\pi)$};
\node[anchor=south] at (-1,-1.5) {$x_1$};
\node[anchor=south] at (2.8, -1.5) {$x_2$};
\draw (-3.7,0.2) ellipse (43pt and 28pt);
\draw (-3.7,-3.2) ellipse (43pt and 28pt);
\draw (4.85,0.2) ellipse (37pt and 27pt);
\draw (4.85,-3.15) ellipse (37pt and 27pt);
\node[anchor=east] at (-0.5,1) {$D^+(\gamma_2)$}; 
\node[anchor=east] at (-0.5,-4) {$D^+(\gamma_1)$}; 
\node[anchor=west] at (1.7,1.1) {$D^+(\gamma_2^{-1})$};
\node[anchor=west] at (1.7,-4) {$D^+(\gamma_1^{-1})$}; 
\end{tikzpicture}
\caption{The Schottky figure associated with $(\gamma_1, \gamma_2)$}
\label{fig:SchottkyFigureEx}
\end{figure}
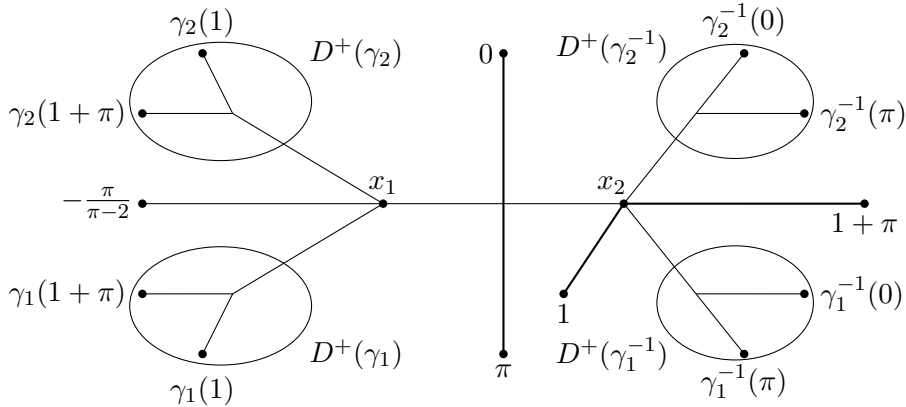

Thanks to Proposition \ref{prop:geometrygroup}, the existence of a Schottky figure ensures that $\Gamma=\langle \gamma_1, \gamma_2 \rangle$ is a Schottky group of rank 2.
Denote its limit set by~$L$. By Theorem~\ref{thm:actionproper}, the quotient $X := \Gamma\backslash(\pank-L)$ makes sense as a $k$-analytic space and it is a Mumford curve of genus~2. 
Let $p \colon (\pank-L) \to X$ denote the universal cover. 
It also follows from Theorem~\ref{thm:actionproper} that the topology of~$X$ may be described quite explicitly from the action of $\Gamma$. We deduce in this way that the skeleton~$\Sigma_X$ of~$X$ is the metric graph represented in Figure \ref{fig:SkeletonEx}.
By measuring the lengths of the paths joining the boundaries of the discs in the Schottky figure, one can verify that the three edges of $\Sigma_X$ have equal lengths.


\begin{figure}[h]
\centering
\begin{tikzpicture}[
  level distance=3cm]
  \coordinate (a) at (0,-1.5);
  \coordinate (b) at (3.5, -1.5);
\draw[thin] (a) -- (b);
\draw[thin] (a) to[out=90,in=90] (b);
\draw[thin] (a) to[out=-90,in=-90] (b);
\fill (a) circle[radius=2pt];
\fill (b) circle[radius=2pt];
\node[anchor=east] at (a) {$p(x_1)$};
\node[anchor=west] at (b) {$p(x_2)$};
\end{tikzpicture}
\caption{The skeleton $\Sigma_X$ of the Mumford curve uniformized by $\Gamma$}
\label{fig:SkeletonEx}
\end{figure}
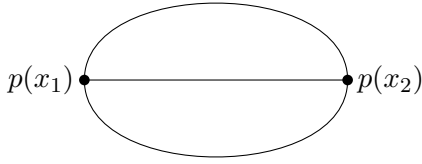

Let us now compute the automorphism group $\Aut(X)$.
By Theorem \ref{thm:automorphismgroup}, this can be done by computing the normalizer $N$ of $\Gamma$ in $\PGL_2(k)$.
The elements $a$ and $b$ lie in $N$, since $\gamma_i a = a\gamma_i^{-1}$ for $i=1,2$ and $\gamma_1 b = b \gamma_2^{-1}$, but we can also find elements in~$N$ that do not belong to the subgroup generated by $a$ and $b$.
Let 
\[c := \begin{bmatrix}
1+\pi & -\pi(1+\pi) \\ 2 & -(1+\pi)
\end{bmatrix} \in \PGL_2(k).\]
A direct computation shows that the transformation $c$ is such that $c^2=\id$, $cac=a$ and $cbc=b^2$, so that $c$ belongs to~$N$.
The group $N'=\langle a,b,c \rangle \subset \PGL_2(k)$ is then contained in $N$, and the quotient~$N'/\Gamma$ is isomorphic to the dihedral group $D_6$ of order 12.
In fact, if we call $\alpha, \beta, \gamma$ the respective classes of $a,b,c$ in $N'/\Gamma$, we have that $\alpha \beta = \beta \alpha$, and then $\langle \alpha, \beta \rangle$ is a cyclic group of order~6.
However, the same computation above shows that $\gamma$ does not commute with $\beta$.
The group~$D_6$ is also the automorphism group of the skeleton $\Sigma_X$, and so, by Proposition \ref{prop:AutomorphismRestriction}, we have $N=N'$ and the restriction homomorphism $\Aut(X)\to \Aut(\Sigma_X)$ is an isomorphism.

Note that one can extract quite a lot of information from the study of the action of $N$ on $\pank$.
In this example, $\alpha \in \Aut(X)$ is an order 2 automorphism known as the \emph{hyperelliptic involution}, since it induces a degree 2 cover of the projective line $\varphi: X\to \pank$.
This last fact can be checked on the skeleton $\Sigma_X$ by noting that $\alpha(p(x_1))=p(x_2)$, and hence $\alpha$ has to switch the ends of every edge of~$\Sigma_X$.
As a result, the quotient $X/\langle \alpha \rangle$ is a contractible Mumford curve,
and hence it is isomorphic to~$\pank$.

This description of $X$ as a cover of $\pank$ is helpful to compute an explicit equation for the smooth projective curve whose analytification is $X$.
In fact, a genus 2 curve that is a double cover of the projective line can be realized as the smooth compactification of a plane curve of equation
\[y^2=\prod_{i=1}^6(x-a_i),\]
where the $a_i \in k$ are the ramification points of the cover, and the involution defining the cover sends~$y$ to~$-y$. 

In order to find the $a_i$, we shall first compute the branch locus $B \subset X$ of the hyperelliptic cover. The fixed points of $a$ are $0$ and $\pi$, so the corresponding points $p(0), p(\pi)$ are in $B$.
The other branch points can be obtained by finding those $x\in \pank(k)$ satisfying the condition $\gamma_i(x)=a(x)$ for $i=1,2$.
We have $\gamma_1(b(0))=abab^2b(0)=aba(0)=a(b(0))$, and the same applies to $b(\pi)$, so the images by~$p$ of these two points are also in $B$.
In the same way, we find that $\gamma_2(b^2(0))=a(b^2(0))$ and $\gamma_2(b^2(\pi))=a(b^2(\pi))$.
We have found in this way that $B=\{p(0), p(b(0)), p(b^2(0)), p(\pi), p(b(\pi)), p(b^2(\pi))\}$.

To find the ramification locus, we have to compute $\varphi(B)$. Since $\langle \alpha \rangle$ is a normal subgroup of $\Aut(X)$, the element $\beta$ acts as an automorphism of order 3 of $X/\langle \alpha \rangle \cong \pank$.
Up to a change of coordinate of this projective line, we can suppose that the fixed points of $\beta$ are $0$ and $\infty$, so that $\beta$ is the multiplication by a primititve third root of unity, and that the first ramification point is $a_1=\varphi(p(0))=1$.
Then, after possibly reordering them, the remaining ramification points are $a_2=\rho$, $a_3=\rho^2$ and $a_4, \rho a_4, \rho^2 a_4$, with $|a_4-1|<1$.

With a bit more effort, we can actually compute the value of $a_4$.
To do this, notice that the function $p$ is injective when restricted to the open fundamental domain \[F^-=\pank - \big(D^+(\gamma_1) \cup D^+(\gamma_1^{-1}) \cup D^+(\gamma_2) \cup D^+(\gamma_2^{-1}) \big).\]
If we set $F'=\varphi \circ p (F^-)$ we then have a two-fold cover $F^- \longrightarrow F' \subset \pank$ induced by $\varphi \circ p$, which can be explicitly written as a rational function $z \mapsto \frac{z^2}{(z-\pi)^2}$ (this function can be found by looking at the action of $a$ on $F^-$ explicitly).
Note that $F^-$ contains both the fixed rigid points of $b$, \textit{i.e.} $1+\pi$ and $1$, and those of $a$, \textit{i.e.} $0$ and $\pi$.
When we reparametrize the projective line on the target of $\varphi$ to get the wanted equation, we are imposing the conditions $\alpha \circ p (1+\pi) \mapsto \infty$, $\alpha \circ p (1) \mapsto 0$ and $\alpha \circ p (0) \mapsto 1=a_1$.
These choices leave only one possibility for the ramification point $a_4$: it is $\big( \frac{1-\pi}{1+\pi}\big)^2$.
We have now found the equation of the plane section of our Mumford curve: it is
\[y^2 = (x^3-1)\cdot \Big(x^3- \frac{(1-\pi)^6}{(1+\pi)^6} \Big ). \]
Note that $|a_4-a_1|=\left| \frac{(1-\pi)^2}{(1+\pi)^2} - 1 \right|=|\pi|$.
\end{example}

A different example of a hyperelliptic Mumford curve with a similar flavour is discussed in the expository paper \cite{CornelissenKato05}, accompanied with figures and other applications of automorphisms of Mumford curves. 

\begin{example}\label{ex:ArtinSchreierMumford}
The curve in Example~\ref{ex:hyperelliptic} has the same automorphism group in every characteristic (different from 2 and 3).
However, Mumford curves in positive characteristic have in general more automorphism than in characteristic 0.
An interesting class of examples are the so-called \emph{Artin-Schreier-Mumford curves}, first introduced by Subrao in \cite{Subrao75}.
We sketch here the main results and refer to \cite{CornelissenKatoEtAl10} for more detailed proofs of these facts.
Let $p$ be a prime, $q=p^e$ be a power of $p$, and $k=\F_q((t))$.
Let $X$ be the analytification of the curve defined inside $\PP^1_k \times \PP^1_k$ by the equation
\[(y^q-y)(x^q-x)=f(t) \; \; \text{with} \; \; f\in t\F_q[[t]].\]

This is an ordinary curve in characteristic $p>0$ with many automorphisms, and for this reason has caught the attention of cryptographers and positive characteristic algebraic geometers alike.
One way to study its automorphisms is to observe that $X$ is a Mumford curve.
A Schottky group attached to it can be constructed by fixing an element $v\in k$ and looking at the automorphisms of $\pank$ of the form
\[ a_u = \begin{bmatrix}
1 & u \\ 0 & 1
\end{bmatrix} , b_u = \begin{bmatrix}
v & 0 \\ u & v
\end{bmatrix} \in \PGL_2(k), \; u\in \F_q^\times. \]
These transformations are all of order $p$, $a_u$ represent translations by elements of $\F_q^\times$ and $b_u$ their conjugates under the inversion $z \mapsto \frac{v}{z}$.
The subgroup $\Gamma_v=\langle a_u^{-1}b_{u'}^{-1}a_ub_{u'} : (u, u')\in {\F_q^\times}^2 \rangle$ of $\PGL_2(k)$ is a Schottky group of rank $(q-1)^2$, and for a certain value of $v$\footnote{The relationship between $v$ and $f(t)$ is not immediate, and it is the object of the paper \cite{CornelissenKatoEtAl10}.} it gives rise to the curve $X$ by Schottky uniformization.
The immediate consequence of this fact, is that $X$ is a Mumford curve of genus $(q-1)^2$.

The group of automorphisms $\Aut(X)$ is isomorphic to a semi-direct product $(\Z/p\Z)^{2e} \rtimes D_{q-1}$, and its action is easy to describe using the equation of the curve: the elementary abelian subgroup $(\Z/p\Z)^{2e}$ consists of those automorphisms of the form $(x,y)\mapsto (x+\alpha,y+\beta)$ with $(\alpha,\beta)\in (\F_q)^2$, while the dihedral subgroup $D_{q-1}$ is generated by $(x,y) \mapsto (y,x)$ and $(x,y)\mapsto (\gamma x,\gamma^{-1}y)$  for $\gamma\in {\F}_q^\times$.
We deduce that the order of~$\Aut(X)$ is~$2(q-1)q^2$.
In characteristic 0, it is not possible to have these many automorphisms, thanks to bounds by Hurwitz and Herrlich that would give rise to a contradiction (see Appendix \ref{app:unif} for the precise statement of these bounds).
\end{example}

%
%


\addtocontents{toc}{\protect\setcounter{tocdepth}{0}}

\appendix

\section{Further readings}

\addcontentsline{toc}{chapter}{\textbf{Appendix A. \hspace{.2em} Further readings}}
%

\addtocontents{toc}{\protect\setcounter{tocdepth}{2}}

The theory of Berkovich curves has several applications to numerous fields of mathematics, and uniformization plays a role in many of these. 
A complete description of these applications goes far beyond the scope of the present text, but we would like to provide the interested reader with some hints about the state of the art and where to find more details in the existing literature, as well as point out which simplifications adopted in this text are actually instances of a much richer theory.

\subsection{Berkovich spaces and their skeleta}\label{app:Berkurves}

We provided a short introduction to the theory of Berkovich curves and their skeleta in Section~\ref{sec:smoothcurves} of this text.

The first discussion of this topic appears already in Chapter 4 of Berkovich's foundational book \cite{Berkovich90}. In this context, the definition of the skeleton of a Berkovich curve $X$ makes use of formal models and the semi-stable reduction theorem, that states that for the analytification of a smooth proper and geometrically irreducible algebraic curve over $k$, there exists a finite Galois extension $K$ of $k$ such that the base change $X_K$ has a semi-stable formal model.
Berkovich showed that the dual graph of the special fiber of any semi-stable formal model embeds in the curve $X_{K}$ and that it is invariant by the action of the Galois group $\Gal(K/k)$ over~$X_{K}$, which allows to define skeleta of~$X$ as quotients of skeleta of~$X_{K}$. This construction is again found in A. Thuillier's thesis \cite{ThuillierPhD}, where it is exploited to define a theory of harmonic functions on Berkovich curves.

A fruitful approach to the study of skeleta has been the one we adopted in Definition~\ref{def:skeleton}, \textit{via} the use of triangulations.
This was first introduced by Ducros in \cite{Ducros08} to study \'etale cohomology groups of Berkovich curves.
In the case where $k$ is algebraically closed, a comprehensive exposition of skeletons, retractions, and harmonic functions on non-Archimedean curves can be found in the paper \cite{BakerPayneEtAl14}.
There, the authors are motivated by connections with tropical geometry, as, for a given algebraic variety over $k$, the skeletons of its analytification are tightly related to its tropicalization maps.
Other than in the aforementioned paper, these connections are exposed in \cite{Werner16}, where the higher-dimensional cases are highlighted as well.

As for higher-dimensional spaces, Berkovich introduced skeleta in~\cite{Berkovich99}. They are simplicial sets onto which the spaces retract by deformation. They are constructed using semi-stable formal models and generalizations of them, so they are not known to exist in full generality, but Berkovich nonetheless managed to use them to prove that smooth spaces are locally contractible (hence admits universal covers). 

The connections with tropical geometry have been proven fruitful, among other things, to study finite covers of Berkovich curves $Y\to X$ over $k$.
The general pattern is that these covers are controlled by combinatorial objects that are enhanced versions of compatible pairs $(\Sigma_Y, \Sigma_X)$ of skeletons of the curves $Y$ and $X$.
Assume that $k$ is algebraically closed. Whenever the degree of such a cover is coprime with the residue characteristic of $k$, the papers \cite{AminiBakerEtAl15} and \cite{AminiBakerEtAl15a} give conditions on a pair $(\Sigma_Y, \Sigma_X)$ to lift to a finite morphism of curves $Y\to X$.
In the case of covers of degree divided by the residue characteristic of $k$, the situation is still far from understood, but progress has been made thanks to the work of Temkin and his collaborators in the papers \cite{CohenTemkinEtAl16,Temkin17,BreznerTemkin20}.
The main tool used in these works is the \emph{different function}.

With regard to higher-dimensional varieties, a new approach to skeletons was proposed by Hrushovski-Loeser \cite{HrushovskiLoeser} using techniques coming from model theory. They are able to define skeleta of analytifications of quasi-projective varieties and deduce the remarkable result that any such space has the homotopy type of a CW-complex.

In the specific case of curves over an algebraically closed base field, the paper~\cite{CubidesPoineau} uses triangulations in order to give a more concrete model-theoretic version of Berkovich curves (and morphisms between them). In particular, the authors manage to give an explicit description of definable subsets of curves and prove some tameness properties.

Without the assumption that $k$ is algebraically closed, or rather that $X$ has a semi-stable formal model over the valuation ring of $k$, the structure of analytic curves is much harder to grasp, due among other things to the difficulty of classifying virtual discs and virtual annuli.
The curious reader will find much food for thought in the book by Ducros \cite{DucrosRSS}, which can nevertheless be of difficult reading for a first approach. 
If $k$ is a discrete valuation field, a generalization of potential theory on Berkovich curves is provided in \cite{BakerNicaise} thanks to a careful study of regular models, and the introduction of the notion of \emph{weight function}. In regard to the problem of determining a minimal extension necessary for the existence of a semi-stable model, an approach via triangulations has been recently proposed in \cite{FantiniTurchetti19}.

Finally, let us mention that we choose to introduce Berkovich curves as $\A^1$-like curves because we are convinced that this is a natural framework for studying uniformization, but the general theory is much richer, and contains many examples of Berkovich curves that are not $\A^1$-like.

\subsection{Non-Archimedean uniformization in arithmetic geometry}\label{app:Shimura}

In the case of curves over the field of complex numbers, Schottky uniformization can be seen  in the context of the classical uniformization theorem for Riemann surfaces, proven independently by Koebe and Poincar\'e in 1907.
This states that every simply connected complex analytic curve $X$ is conformally equivalent to the complex projective line, the complex affine line, or the Poincar\'e upper-half plane.
As a consequence, the universal covering space of $X$ is also one of these, and when $X$ is compact, Koebe-Poincar\'e uniformization factors through the Schottky uniformization $\big(\pana{\C} - L \big) \to X$.
A remarkable book on complex uniformization \cite{SaintGervais10} has been written by the group of mathematicians known under the collective name of Henri Paul de Saint-Gervais. 
It constitutes an excellent reference both on the historical and mathematical aspects of the subject.

In the non-Archimedean case, the history of uniformization is much more recent. The uniformization theory of elliptic curves over a non-Archimedean field $(k, \va)$ was the main motivation underlying J. Tate's introduction of rigid analytic geometry in the 1960s.
Using his novel approach, Tate proved that every elliptic curve with split multiplicative reduction over $k$ is analytically isomorphic to the multiplicative group $k^\times/q^\Z$ for some $q$ in $k$ with $0<|q|<1$.
Tate's computations were known to experts, but remained unpublished until 1995, when they were presented in \cite{Tate} together with a discussion on further aspects of this theory, including automorphic functions, a classification of isogenies of Tate curves, and a brief mention of how to construct ``universal'' Tate curves over the ring $\Z[[q]][q^{-1}]$ using formal geometry.
These formal curves appeared for the first time in the paper \cite{DeligneRapoport73} by P. Deligne and M. Rapoport, who attributed it to M. Raynaud and called them \emph{generalized elliptic curves}.
In \emph{loc. cit.} the authors exploited them to give a moduli-theoretic interpretation at the cusps of the modular curves $X_0(Np)$ with $p \nmid N$.
Further reading in this direction include the foundational paper \cite{KatzMazur85}, that concerns the case of modular curves $X(Np^n)$ and \cite{Conrad07}, that provides a more contemporary perspective on generalized elliptic curves.

Interpreting the Schottky uniformization of Mumford curves of \cite{Mumford72} as a higher genus generalization of Tate's theory, inspired several novel arithmetic discoveries.
One of the most important is the uniformization of Shimura curves, fundamental objects in arithmetic geometry that vastly generalize modular curves.
In \cite{Cherednik76}, I. Cherednik considered a Shimura curve $\calC$ associated with a quaternion algebra $B$ over $\Q$. 
For a prime $p$ where $B$ is ramified, he proved that the $p$-adic analytic curve $(\calC\times_\Q \Q_p)^{\an}$ can be obtained as a quotient of \emph{Drinfeld $p$-adic halfplane} $\pana{\Q_p} - \pank(\Q_p)$, by the action of a Schottky group.
This Schottky group can be as a subgroup of a different quaternion algebra $B'$ over $\Q$, constructed explicitly from $B$ via a procedure known as \emph{interchange of invariants}.
The theory obtained in this way is classically referred to as \emph{Cherednik-Drinfeld uniformization}, since V. Drinfeld gave a different proof of this result in \cite{Drinfeld76}, building on a description of $\calC$ as a moduli space of certain abelian varieties.
The excellent paper \cite{BoutotCarayol91} provides a detailed account of these constructions.\\
By generalizing Drinfeld's modular interpretation, the approach can be extended to some higher dimensional Shimura varieties, resulting in their description as quotients of the Drinfeld upper-half space via a uniformization map introduced independently by G. Mustafin \cite{Mustafin78} and A. Kurihara \cite{Kurihara80}.
For a firsthand account of the development of this uniformization, we refer the reader to the book \cite{RapoportZink96} by M. Rapoport and T. Zink.

Non-Archimedean uniformization of Shimura varieties has remarkable consequences.
First of all, it makes possible to find and describe integral models of Shimura varieties, since the property of being uniformizable imposes restrictions on the special fibers of such models.
Furthermore, it gives a way to compute \'etale and $\ell$-adic cohomology groups, as well as the action of the absolute Galois group $\Gal(\overline{\Qp}/\Qp)$ on these, making it a powerful tool for studying Galois representations.
All the aforementioned results were shown in the framework of formal and rigid geometry. 
However, more contemporary approaches to uniformization of Shimura varieties and Rapoport-Zink spaces make use of Berkovich spaces (see \cite{Varshavsky98, JordanLivneEtAl03}), or Huber adic geometry in the form of perfectoid spaces (see \cite{ScholzeWeinstein13} and \cite{Caraiani19}). 
In particular, the perfectoid approach can be used to vastly generalize the uniformization of Shimura varieties and establish a theory of \emph{local Shimura varieties}. This construction is exposed in the lecture notes \cite{ScholzeWeinstein20} by P. Scholze and J. Weinstein.

Local and global uniformization of Shimura varieties are investigated in relation to period mappings, Gauss-Manin connections, and uniformizing differential equations in the book by Y.Andr\'e \cite{Andre03}, where striking similarities between the complex and $p$-adic cases are highlighted.
For more results about the relevance of Shimura varieties, not necessarily with regard to uniformization, we refer to \cite{Milne05}.

Finally, let us mention that Tate's uniformization of elliptic curves with split multiplicative reduction generalizes to abelian varieties.
This is also a result of Mumford, contained in the paper \cite{Mumford72a}, that can be regarded as a sequel to \cite{Mumford72}, since the underlying ideas are very similar.
In this case, the uniformization theorem is formulated by stating that a totally degenerate abelian variety of dimension $g$ over $k$ is isomorphic to the quotient of the analytic torus $(\G_{m,k}^g)^{\an}$ by the action of a torsion free subgroup of $(k^\times)^g$.
This applies in particular to Jacobians of Mumford curves, a case surveyed in detail in the monograph \cite{Lutkebohmert16}.
We shall remark that Mumford's constructions are more general than their presentation in this text: they work not only over non-Archimedean fields, but more generally over fields of fractions of complete integrally closed noetherian rings of any dimension.


\subsection{The relevance of Mumford curves}\label{app:unif}

The uniformization theorem in the complex setting is a very powerful tool, and one of the main sources of analytic methods applied to the study of algebraic curves.
This leads to the expectation that, in the non-Archimedean setting, Mumford curves can be more easily studied, turning out to be a good source of examples for testing certain conjectures.
This is indeed the case for several topics in algebraic curves and their applications, as we could already sample in Section \ref{sec:automorphisms} on the subject of computing the group of automorphisms of curves.

This appendix is a good place to remark that Examples \ref{ex:hyperelliptic} and \ref{ex:ArtinSchreierMumford} in that section are instances of a much deeper theory.
For a smooth projective algebraic curve $C$ of genus $g\geq 2$ over a field of characteristic zero, the Hurwitz bound ensures that the finite group of automorphisms $\Aut(C)$ is of order at most $84(g-1)$. 
This bound is sharp: there exist curves of arbitrarily high genus whose automorphism groups attain it, the so-called \emph{Hurwitz curves}.
However, if we know that $C$ is (the algebraization of) a Mumford curve, 
F. Herrlich proved a better bound in \cite{Herrlich80}.
Namely, if we denote by $p$ the residue characteristic of $K$, he showed that:

\[
|\Aut(C)|\leq
\begin{cases}
 48(g-1) & p=2\\
24(g-1) & p=3\\
30(g-1) & p=5\\
12(g-1) &  \mbox{otherwise}.
\end{cases}
\]
This result relies on the characterization of automorphism groups of Mumford curves as quotients $N/\Gamma$, where $\Gamma$ is a Schottky group associated with $C$ and $N$ its normalizer in $\PGL_2(K)$ (see Theorem \ref{thm:automorphismgroup}).
One can show that the group $N$ acts discontinuously on an infinite tree that contains the universal covering tree of the skeleton $\Sigma_{C^\mathrm{an}}$, and use Serre's theory of groups acting on trees to prove that $N$ is an amalgam of finite groups.
In his paper, Herrlich achieves the bounds above by classifying those amalgams that contain a Schottky group as a normal subgroup of finite index.

Over a field of characteristic $p>0$, the Hurwitz bound is replaced by the Stichtenoth bound, stating that $|\Aut(C)| \leq 16g^4$, unless $C$ is isomorphic to a Hermitian curve.
When $C$ is a Mumford curve, this bound can be improved in principle using Herrlich's strategy.
However, this is not an easy task, as one has to overcome the much bigger difficulties that arise in positive characteristic.
This has been achieved recently by M. Van der Put and H. Voskuil, who prove in \cite[Theorem 8.7]{VoskuilPut19} that $|\Aut(C)|<\max\{12(g-1) , g\sqrt{8g+1}+3\}$ except for three occurrences of (isomorphism classes of) $X$, which happen when $p=3$ and $g=6$.
Moreover, in \cite[Theorem 7.1]{VoskuilPut19} they show that the bound is achieved for any choice of the characteristic $p>0$.
The bound corrects and extends a bound given by G. Cornelissen, F. Kato and A. Kontogeorgis in \cite{CornelissenKatoEtAl01}.

Another application of uniformization of Mumford curves is the \emph{resolution of non-singularities} for hyperbolic curves\footnote{A hyperbolic curve in this context is a genus $g$ curve with $n$ marked points satisfying the inequality $2g-2+n>0$.} over $\overline{\Q}_p$.
Given such a curve $X$, and a smooth point $P$ of the special fiber of a semi-stable model of $X$, it is an open problem to find a finite \'etale cover $Y\longrightarrow X$ such that a whole irreducible component of the special fiber of the stable model of $Y$ lies above $P$.
Earlier versions of this problem were introduced and proved by S. Mochizuki \cite{Mochizuki96} and A. Tamagawa \cite{Tamagawa04}, that showed connections with important problems in anabelian geometry.
The interest of the version proposed here is also motivated by anabelian geometry: F. Pop and J. Stix proved in \cite{PopStix17} that any curve for which resolution of non-singularities holds satisfies also a valuative version of Grothendieck's section conjecture.
In the paper \cite{Lepage11}, E. Lepage uses Schottky uniformization in a Berkovich setting to show that resolution of non-singularities holds when $X$  is a hyperbolic Mumford curve.
His approach consists in studying $\mu_{p^n}$-torsors of the universal cover of $X$, which are better understood since they can be studied using logarithmic differentials of rational functions.
With this technique, he can show that there is a dense subset of type 2 points $\mathcal{V} \in X$, with the following property: every $x\in \mathcal{V}$ can be associated with a $\mu_{p^n}$-torsor $\tau:Y\to X$ such that $\tau^{-1}(x)$ is a point of positive genus.
This last condition ensures that the corresponding residue curve is an irreducible component of the stable model of $Y$.

Mumford curves have been also proven useful in purely analytic contexts, for instance to study potential theory and differential forms.
Using the fact that all type 2 points in a Mumford curve are of genus 0, P. Jell and V. Wanner \cite{JellWanner18} are able to establish a result of Poincar\'e duality and compute the Betti numbers of the tropical Dolbeaut cohomology arising from the theory of bi-graded real valued differential forms developed in \cite{ChambertLoirDucros12}.

Finally, let us mention that archimedean and non-archimedean Schottky uniformizations can be studied in a unified framework thanks to work of the authors~\cite{PoineauTurchetti}, where a moduli space $\calS_g$ parametrizing Schottky groups of fixed rank~$g$ over all possible valued fields is constructed for every $g\geq 2$.
This construction is performed in the framework of \emph{Berkovich spaces over $\Z$} developed in~\cite{A1Z, Poineau13, LemanissierPoineau}.
More precisely, the space $\calS_g$ is realized as an open, path-connected subspace of $\A^{3g-3, \mathrm{an}}_\Z$, it is endowed with a natural action of the group $\mathrm{Out}(F_g)$ of outer automorphisms of the free group, and exhibits interesting connections with other constructions of moduli spaces, in the frameworks of tropical geometry and geometric group theory.
The space $\calS_g$ seems to be ideal to study phenomena of degeneration of Schottky groups from archimedean to non-archimedean.

A different take on the interplay between archimedean and non-archimedean Schottky uniformizations is provided by Y. Manin's approach to Arakelov geometry.
In the paper \cite{Manin91} several formulas for computing the Green function on a Riemann surface using Schottky uniformization and are explicitly inspired by Mumford's construction.
These formulas involve the geodesics lengths in the hyperbolic handlebody uniformized by the Schottky group associated with such a surface, suggesting connections between hyperbolic geometry and non-archimedean analytic geometry.
This result has been reinterpreted in term of noncommutative geometry by C. Consani and M. Marcolli \cite{ConsaniMarcolli04} by replacing the Riemann surface with a noncommutative space that encodes certain properties of the archimedean Schottky uniformization.
This noncommutative formalism has led to applications both in the non-archimedean world (see for example \cite{ConsaniMarcolli03}) and in the archimedean one, for instance to Riemannian geometry in \cite{CornelissenMarcolli08}.
We think that the theory of Berkovich spaces could fit nicely in this picture, and it would be an interesting project to investigate the relations between noncommutative geometric objects related to Schottky uniformization (e.g. graph $C^\star$-algebras) and Mumford curves in the Berkovich setting.


\addtocontents{toc}{\vspace{1\baselineskip}}

\bibliographystyle{alpha}
\bibliography{MumfordBiblio}

\end{document}